\documentclass[11pt, reqno]{amsart}
\usepackage{amsmath, amssymb, amsthm, graphicx,marvosym}
\usepackage{cancel}
\usepackage{cite}
\usepackage[nobysame]{amsrefs}[11pts]
\usepackage{bm}
\usepackage{color}
\newtheorem{theorem}{Theorem}[section]
\newtheorem{definition}{Definition}[section]
\newtheorem{lemma}{Lemma}[section]
\newtheorem{remark}{Remark}[section]

\newtheorem{corollary}{Corollary}[section]
\numberwithin{equation}{section}
\numberwithin{figure}{section}

\makeatletter \@addtoreset{equation}{section} \makeatother

\def\tilde{\widetilde}

\newcommand{\fr}{\frac}
\newcommand{\ef}{\eqref}

\newcommand{\beq}{\begin{equation}}
\newcommand{\eeq}{\end{equation}}

\newcommand{\dv}{{\mathrm {div}}}

\begin{document}

\title[Degenerate compressible Navier-Stokes equations]{On regular solutions for   three-dimensional full compressible Navier-Stokes equations with degenerate viscosities and far field vacuum}

\author{Qin Duan}
\address[Q. Duan]{College of Mathematics and Computational Science, Shenzhen University, Shenzhen, China; The Institute of Mathematical Sciences, Chinese University of Hong Kong, Hong Kong.}
\email{\tt qduan@szu.edu.cn}

\author{Zhouping Xin}
\address[Z.P. Xin]{The Institute of Mathematical Sciences, The Chinese University of Hong Kong, Shatin, N.T., Hong Kong.}
\email{\tt zpxin@ims.cuhk.edu.hk}

\author{Shengguo Zhu}
\address[S. G.  Zhu]{School of Mathematical Sciences, CMA-Shanghai,  and MOE-LSC,  Shanghai Jiao Tong University, Shanghai, 200240, China; Mathematical Institute, University of Oxford,  Oxford OX2 6GG, UK.}
\email{\tt zhushengguo@sjtu.edu.cn}


\date{\today}

\begin{abstract}

In this paper, the Cauchy problem for the   three-dimensional (3-D) full  compressible  Navier-Stokes equations  (\textbf{CNS}) with zero thermal conductivity is considered.  First,   when shear and bulk viscosity coefficients  both depend on  the absolute  temperature $\theta$ in a  power  law ($\theta^\nu$ with $\nu>0$)  of Chapman-Enskog,   based on some elaborate analysis of this system’s intrinsic singular structures,  we identify one class of initial data admitting a   local-in-time regular solution with far field vacuum in terms of the mass density $\rho$, velocity $u$ and   entropy $S$. 
Furthermore, it is shown that within its life span of  such a regular  solution, 
the velocity  stays in an inhomogeneous Sobolev space,  i.e., $u\in H^3(\mathbb{R}^3)$,    $S$ has uniformly  finite lower and upper bounds in the whole space, and  the laws of  conservation of  total mass, momentum and  total energy are all satisfied. Note that  due to the appearance  of the vacuum, the momentum equations are  degenerate both in the time evolution and viscous stress tensor,  and the physical entropy for polytropic gases behaves singularly,
which make the  study on corresponding  well-posedness challenging. For  proving  the existence,  we first  introduce    an   enlarged reformulated structure by considering  some new variables, which  can transfer the degeneracies  of the full \textbf{CNS}   to the possible singularities of some special source terms related with $S$, and then carry out some  singularly weighted energy estimates carefully designed  for this reformulated system.

\end{abstract}
\maketitle

\tableofcontents

\section{Introduction}

The motion of a  compressible viscous, heat-conductive, and Newtonian polytropic fluid occupying a spatial domain $\Omega \subset \mathbb{R}^3$  is governed by the following full \textbf{CNS}:
\begin{equation}\label{1}
\left\{\begin{aligned}
&\rho_t+\text{div}(\rho u)=0,\\[5pt]
&(\rho u)_t+\text{div}(\rho u\otimes u)+\nabla P=\text{div}\mathbb{T},\\[5pt]
&(\rho \mathcal{E})_t+\text{div}(\rho \mathcal{E} u+Pu)=\text{div}(u\mathbb{T})+\text{div}(\kappa\nabla\theta).
\end{aligned}\right.
\end{equation}
Here and throughout, $\rho\geq 0$ denotes the mass density, $u=(u^{(1)},u^{(2)},u^{(3)})^{\top}$ the fluid velocity, $P$ the pressure of the fluid,  $\theta$ the absolute temperature, $\mathcal{E}=e+\frac{1}{2}|u|^2$  the specific total energy,  $e$ the specific internal energy, $x=(x_1,x_2,x_3)^\top\in \Omega$  the Eulerian  spatial coordinate  and finally $t\geq 0$  the time coordinate. The  equation of state for polytropic fluids satisfies 
\begin{equation}\label{2}
    P=R\rho\theta=(\gamma-1)\rho e=Ae^{S/c_v}\rho^\gamma,\hspace{2mm} e=c_v\theta,\hspace{2mm}c_v=\frac{R}{\gamma-1},
\end{equation}
where $R$ is the gas constant, $A$ is a   positive constant, $c_v$ is  the specific heat at constant volume,  $\gamma>1$ is the adiabatic exponent and $S$ is the entropy. $\mathbb{T}$ is the viscosity stress tensor given by
\begin{equation}\label{3}
\mathbb{T}=2\mu D(u)+\lambda \text{div}u \mathbb{I}_3,
\end{equation}
where $D(u)=\frac{\nabla u+(\nabla u)^\top}{2}$ is the deformation tensor, $\mathbb{I}_3$ is the $3\times 3$ identity matrix, $\mu$ is the shear viscosity coefficient, and $\lambda+\frac{2}{3}\mu$ is the bulk viscosity coefficient. $\kappa$ denotes the   coefficient of thermal conductivity.

In the theory of gas dynamics,  the compressible Navier-Stokes equations can be derived from the Boltzmann equations through the Chapman-Enskog expansion, cf. Chapman-Cowling \cite{chap} and Li-Qin \cite{tlt}. Under some proper physical assumptions,  the viscosity coefficients $(\mu,\lambda)$ and the   coefficient of thermal conductivity $\kappa$ are not constants but functions of the absolute temperature $\theta$ such as:
\begin{equation}
\label{eq:1.5g}
\begin{split}
\mu(\theta)=&r_1 \theta^{\frac{1}{2}}F(\theta),\quad  \lambda(\theta)=r_2 \theta^{\frac{1}{2}}F(\theta), \quad \kappa(\theta)=r_3 \theta^{\frac{1}{2}}F(\theta)
\end{split}
\end{equation}
for some   constants $r_i$ $(i=1,2,3)$.
  Actually for the cut-off inverse power force models, if the intermolecular potential varies as $r^{-\Upsilon}$,
 where $ r$ is intermolecular distance,  then 
 \begin{equation}
\label{eq:1.6g}
F(\theta)=\theta^{\varpi}\quad  \text{with}\quad   \varpi=\frac{2}{\Upsilon} \in [0,\infty)
\end{equation}
in \ef{eq:1.5g}.   In particular, for Maxwellian molecules,
$\Upsilon = 4$  and $\varpi=\frac{1}{2}$;
 for rigid elastic spherical molecules,
$\Upsilon=\infty$ and  $\varpi=0$; while for  ionized gas, $\Upsilon=1$ and $\varpi=2$.  

In the current  paper, we will consider the following case:
\begin{equation}\label{4}
\mu(\theta)=\alpha\theta^\nu,\hspace{2mm}\lambda(\theta)=\beta\theta^\nu,\hspace{2mm}\kappa=0,
\end{equation}
i.e., the thermal conductivity vanishes, where $(\alpha,\beta,\nu)$ are all constants satisfying
\begin{equation}\label{5}
\alpha>0,\quad  2\alpha+3\beta\geq 0 \quad \text{and} \quad 0<(\gamma-1)\nu <1.
\end{equation}
In terms of $(\rho,u,S)$, it follows from \eqref{2} and \eqref{4} that   \eqref{1} can be rewritten as
\begin{equation}\label{8}\left\{\begin{aligned}
\displaystyle
&\rho_t+\text{div}(\rho u)=0,\\[1pt]
\displaystyle
&(\rho u)_t+\text{div}(\rho u\otimes u)+\nabla P
=A^\nu R^{-\nu}\text{div}(\rho^{\delta}e^{\frac{S}{c_v}\nu}Q(u)),\\[1pt]
\displaystyle
&P \big(S_t+u\cdot \nabla S\big)=A^{\nu}R^{1-\nu}\rho^{\delta}e^{\frac{S}{c_v}\nu}H(u),
\end{aligned}\right.\end{equation}
where $\delta=(\gamma-1)\nu$, and 
\begin{equation}\label{QH}
Q(u)=\alpha(\nabla u+(\nabla u)^\top)+\beta\text{div}u\mathbb{I}_3,\quad H(u)=\text{div}(uQ(u))-u\cdot \text{div}Q(u).\end{equation}
Let $\Omega=\mathbb{R}^3$. We  study  the local-in-time well-posedness of smooth solutions $(\rho,u,S)$ with finite total mass and finite total  energy to the Cauchy problem    $\eqref{8}$  with \eqref{2}, \eqref{5}, \eqref{QH},  and  the following initial data and far field behavior:
\begin{align}
(\rho,u,S)|_{t=0}=(\rho_0(x)>0,u_0(x),S_0(x))\quad  &\text{for} \quad x\in\mathbb{R}^3, \label{6}\\[2pt]
(\rho,u,S)(t,x)\rightarrow(0,0,\bar{S})\quad  \text{as} \quad |x|\rightarrow\infty\quad &\text{for}\quad t\geq 0,  \label{7}
\end{align}
where $\bar{S}$ is some constant. 


Throughout this paper, we adopt the following simplified notations, most of them are for the standard homogeneous and inhomogeneous Sobolev spaces:
\begin{equation*}\begin{split}
 & \|f\|_s=\|f\|_{H^s(\mathbb{R}^3)},\quad |f|_p=\|f\|_{L^p(\mathbb{R}^3)},\quad \|f\|_{m,p}=\|f\|_{W^{m,p}(\mathbb{R}^3)},\\[1pt]  &|f|_{C^k}=\|f\|_{C^k(\mathbb{R}^3)},  \quad \|f\|_{X_1 \cap X_2}=\|f\|_{X_1}+\|f\|_{X_2},\\[1pt]  
 & D^{k,r}=\{f\in L^1_{loc}(\mathbb{R}^3): |f|_{D^{k,r}}=|\nabla^kf|_{r}<\infty\},\\[1pt]
&D^{1}_*=\{f\in L^6(\mathbb{R}^3):  |f|_{D^1_*}= |\nabla f|_{2}<\infty\},\quad  D^k=D^{k,2},  \\[1pt]
 & |f|_{D^{k,r}}=\|f\|_{D^{k,r}(\mathbb{R}^3)},\quad  |f|_{D^{1}_*}=\|f\|_{D^{1}_*(\mathbb{R}^3)}, \quad \int  f   =\int_{\mathbb{R}^3}  f \text{d}x, \\[1pt]
 &  X([0,T]; Y(\mathbb{R}^3))= X([0,T]; Y),\quad \|(f,g)\|_X=\|f\|_{X}+\|g\|_{X}.
\end{split}
\end{equation*}
 A detailed study of homogeneous Sobolev spaces  can be found in Galdi \cite{gandi}.

Under the assumption that    $(\mu,\lambda,\kappa)$ are all constants,
 when  $\inf_x {\rho_0(x)}>0$, the local well-posedness of classical solutions to the Cauchy problem of  \eqref{1} follows from the standard symmetric hyperbolic-parabolic structure which satisfies the well-known Kawashima's condition, cf. \cites{itaya1, itaya2, KA, nash, serrin,tani}. However, such an approach fails  in the presence of the vacuum due to some new difficulties, for example,  the degeneracy of the time evolution in the momentum equations:
 \begin{equation}\label{dege}
\displaystyle
 \underbrace{\rho(u_t+u\cdot \nabla u)}_{Degenerate   \ time \ evolution \ operator}+\nabla P= \text{div}(2\mu D(u)+\lambda \text{div}u \mathbb{I}_3).
\end{equation}
 Generally vacuum will appear in the far field under some physical requirements such as finite total mass and total energy in the whole space   $ \mathbb{R}^3$.  One of the  main issues in the presence of vacuum is   to understand the behavior of the  fluids velocity, temperature  and entropy near the vacuum.   For general initial data containing vacuum, the local well-posedness of strong solutions to the Cauchy problem  of the 3-D full  \textbf{CNS} was first obtained by Cho-Kim  \cite{CK} in a homogeneous Sobolev space in terms of $(\rho, u,\theta)$ under suitable initial  compatibility conditions,  which  has been extended recently   to be  the   global-in-time ones with small energy but large oscillations by Huang-Li \cite{huangli}, Wen-Zhu \cite{wenzhu} and so on.
 It should be noticed that, in \cite{CK, huangli,wenzhu}, the solution was established in some homogeneous  space,  that is,  $\sqrt{\rho}u$ rather than $u$ itself  has the $ L^\infty ([0,T]; L^2)$ regularity,  the regularities obtained for $(\rho, u,\theta)$ do not provide any information about the entropy near the vacuum, and in general the solutions do not lie in inhomogeneous Sobolev spaces and guarantee  the boundedness of the  entropy   near the vacuum.
 In fact,   if $\rho_0$ has compact support,  Li-Wang-Xin \cite{ins} prove that any non-trivial  classical solutions with finite energy to the Cauchy problem of \eqref{1} with constant viscosities and heat
conduction  do not exist in general in the standard  inhomogeneous Sobolev space for any short time, which indicates in particular that the homogeneous Sobolev space is crucial as studying the well-posedness (even locally in time) for the Cauchy problem of the \textbf{CNS} in the presence of such kind of  vacuum. However, when  the initial density vanishes only at far fields with a slow decay rate, recently in Li-Xin \cite{lx1,lx2,lx3}, for the Cauchy problem of the \textbf{CNS} with constant viscosities and heat conduction, it is shown that the uniform boundedness of $S$  and the $L^2$ regularity of $u$   can be propagated  within the solution's life span. Despite these important progress, it remains unclear whether   the laws of  conservation of   momentum and  total energy hold for the solutions  obtained in \cites{CK, huangli,lx1,lx2,lx3, wenzhu}.
We also  refer the readers to \cites{ hoff,jensen, song, fu2, fu3, KA2, lx1, lx2, mat,zx} and   the references therein  for some other related  results on  global existence of weak or strong solutions.

In contrast to the fruitful development of the classical setting, the treatment on the
physical case \eqref{eq:1.5g}-\eqref{eq:1.6g} is lacking due to some new difficulties introduced in such relations, which
lead to strong  degeneracy and  nonlinearity both in viscosity and  heat
conduction besides the degeneracy in the time evolution. 
Recently, for the Cauchy problem of the  isentropic system  (i.e.,   $\eqref{8}_1$-$\eqref{8}_2$  with $S(t,x)$  being constant), by introducing    an elaborate (linear) elliptic approach on the singularly weighted regularity estimates for  $u$ 
and  a symmetric hyperbolic system with singularity for some density related  quantity, Xin-Zhu \cite{zz2} identifies a class of initial data admitting one unique 3-D local regular solution with far field vacuum and finite energy to the  Cauchy problem of \eqref{8} for the case $0<\delta <1$ in some inhomogeneous Sobolev  spaces.  Indeed,  the momentum equations for   isentropic flows can be written as 
\begin{equation}\label{dengshang}
\displaystyle
 \underbrace{\rho(u_t+u\cdot \nabla u)}_{Degenerate   \ time \ evolution \ operator}+\nabla P= \underbrace{\text{div}(\rho^{\delta}Q(u))}_{Degenerate\ elliptic \ operator}.
\end{equation}
Since the coefficients of the time evolution and the viscous stress tensor near the vacuum
are  powers of $\rho$, it is easy  to compare  the order  of the degeneracy  of  these two operators near the vacuum, which enable us  to  select the dominant operator to control the behavior of  $u$  and lead  to the "hyperbolic-strong singular elliptic" coupled structure in \cite{zz2} and  the  `` quasi-symmetric hyperbolic"--``degenerate elliptic" coupled structure in Xin-Zhu \cite{zz}.
Some other interesting results on the well-posedness  with vacuum  to the  isentropic  \textbf{CNS}  with degenerate viscosities   can also  be found in \cites{bd6,bd, bd2, bvy, ding, zhu, sz3, sz333,  hailiang, lz,zhangting, vassu, sundbye2}  and the references therein.

Since $e$ and $S$ are fundamental dynamical variables for viscous compressible fluids, it is of great importance to study the corresponding theory for the non-isentropic  \textbf{CNS}. Yet, as indicated even in the constant $(\mu, \lambda, \kappa)$ case \cite{lx1,lx2,lx3}, this is a subtle and difficult problem in the presence of vacuum. 
Indeed, for  considering the well-posedenss of classical solutions with vacuum to the full  \textbf{CNS} \eqref{1}-\eqref{3} with degenerate viscosities of the form  \eqref{eq:1.5g}-\eqref{eq:1.6g} with $\kappa=0$, the coefficients' structures of the time evolution operator and the viscous stress tensor are different, and the entropy plays important roles and satisfies a highly degenerate  nonlinear transport equation near the vacuum, which cause substantial difficulties in the analysis and make it difficult to adapt the method for the isentropic case in \cite{zz2}. It should be pointed out tha due to the physical requirements on $\theta$ and $S$ near the vacuum, it is of more advantages to formulate the \textbf{CNS} \eqref{1}-\eqref{3} in terms of $(\rho, u,S)$ instead of $(\rho, u,\theta)$ in contrast to \cite{CK,huangli, wenzhu} as  illustrated below.
  Due  to the  equation  of state for polytropic fluids \eqref{2}, one has  in the fluids region that $\rho>0$,
\begin{equation}\label{keyobsevarion1}
\displaystyle
\theta=AR^{-1}\rho^{\gamma-1} e^{S/c_v},
\end{equation}
which  implies  that  $\eqref{1}_2$-$\eqref{1}_3$ can be rewritten into 
\begin{equation}\label{doubledegebbb}
\left\{\begin{aligned}
\displaystyle
& \underbrace{\rho(u_t+u\cdot \nabla u)}_{Degenerate   \ time \ evolution \ operator}+\nabla P= \underbrace{A^\nu R^{-\nu}\text{div}(\rho^{\delta}e^{\frac{S}{c_v}\nu}Q(u))}_{Degenerate\ elliptic \ operator},\\
\displaystyle
 &\underbrace{P \big(S_t+u\cdot \nabla S\big)}_{Degenerate   \ time \ evolution \ operator}=\underbrace{A^{\nu}R^{1-\nu}\rho^{\delta}e^{\frac{S}{c_v}\nu}H(u)}_{Strong   \ nonlinearity},
\end{aligned}\right.
\end{equation}
Thus, if  $S$ has uniform boundedness  in $\mathbb{R}^d$, then it is still possible to  compare  the orders of the degeneracy of the time evolution and  viscous stress tensor near the vacuum via the powers of $\rho$, and then to  choose proper structures to control the behaviors of the physical quantities.  However, due to the high degeneracy in the time evolution operator of the entropy equation $\eqref{doubledegebbb}_2$, the physical entropy for  polytropic  gases behaves singularly in the presence of vacuum, and it is thus a challgenge to study its dynamics.   It is worth pointing out that, when vacuum appears, the system formulated in terms of $(\rho, u,\theta)$ is not equivalent to the one formulated in terms of $(\rho, u,S)$, since the boundedness and regularities of $(\rho, \theta)$ cannot provide those  of $S$ near the vacuum. In fact,  for the constant $(\mu, \lambda, \kappa)$ case, most of the current progress on the well-posedness theory  with vacuum \cite{CK,fu2, fu3, huangli, wenzhu}  to \eqref{1} are based on the formulation in terms of $(\rho, u,\theta)$ except \cite{lx1,lx2,lx3}, where the   De Giorgi type iteration is  carried out to the entropy equation for establishing the lower and upper bounds of the entropy. However, since the assumption that $(\mu, \lambda, \kappa)$ are all constants plays a key role in the analysis of  \cite{lx1,lx2,lx3}, it seems difficult to adapt their arguments to the degenerate  system \eqref{8}.  As far as we know,   there have  no any  results on the boundedness and regularities of the  entropy  near the vacuum for flows with degenerate viscosities  in the existing literatures.

Note that with the constraints  \eqref{4}-\eqref{5},   the momentum equations and the entropy equation $\eqref{8}_2$-$\eqref{8}_3$ in the fluids region can be formally rewritten as
\begin{equation}\label{qiyi}
\left\{\begin{aligned}
\displaystyle
&u_t+u\cdot\nabla u +\frac{A\gamma}{\gamma-1}e^{\frac{S}{c_v}}\nabla\rho^{\gamma-1}+A\rho^{\gamma-1} \nabla e^{\frac{S}{c_v}}+A^\nu R^{-\nu}\rho^{\delta-1}e^{\frac{S}{c_v}\nu} Lu\\
=&A^\nu R^{-\nu}\frac{\delta}{\delta-1}\nabla\rho^{\delta-1}\cdot Q(u)e^{\frac{S}{c_v}\nu} + A^\nu R^{-\nu}\rho^{\delta-1}\nabla e^{\frac{S}{c_v}\nu} \cdot Q(u),\\
\displaystyle
& S_t+u\cdot \nabla S=A^{\nu-1}R^{1-\nu}\rho^{\delta-\gamma}e^{\frac{S}{c_v}(\nu-1)}H(u),
\end{aligned}\right.
\end{equation}
where $L$ is  the Lam\'e operator defined by
\begin{equation*}Lu \triangleq -\alpha\triangle u-(\alpha+\beta)\nabla \mathtt{div}u.
\end{equation*}
Then, to establish   the existences  of  classical  solutions with vacuum to \eqref{8},   one would  encounter some  essential difficulties  as follows:
\begin{enumerate}
\item
first,  the right hand side of $\eqref{qiyi}_1$ contains some strong singularities  as:
$$\nabla\rho^{\delta-1}\cdot Q(u)e^{\frac{S}{c_v}\nu} \quad \text{and} \quad \rho^{\delta-1}\nabla e^{\frac{S}{c_v}\nu} \cdot Q(u),$$
whose controls require certain singularly weighted energy estimates which are highly non-trivial due to the singularities in the entropy equation $\eqref{qiyi}_2$. Furthermore,  the second  term is  more singular than  the first  one.

\item
second, even in the case that the uniform boundedness of the entropy $S$ can be obtained, the coefficient $A^\nu R^{-\nu}\rho^{\delta-1}e^{\frac{S}{c_v}\nu}$ in front of the Lam\'e operator $ L$ will tend to $\infty$ as $\rho\rightarrow 0$ in the far filed. Then   it is  necessary  to show that the term $\rho^{\delta-1}e^{\frac{S}{c_v}\nu} Lu$ is well defined in some Sobolev space near the vacuum. 

\item at last but more importantly,  the time evolution equation $\eqref{qiyi}_2$  for the entropy $S$ also contains a  strong singularity   as:
$$ A^{\nu-1}R^{1-\nu}\rho^{\delta-\gamma}e^{\frac{S}{c_v}(\nu-1)}H(u),$$
and its  singularity can be measured  in the level of $\rho^{\delta-\gamma}$ near the vacuum. This singularity  will be the main obstacle preventing one from getting the  uniform boundedness of the entropy and high order regularities, thus whose analysis  becomes extremely crucial.
\end{enumerate}

Therefore, the four quantities
$$(\rho^{\gamma-1},\ \nabla \rho^{\delta-1},\ \rho^{\delta-1} Lu, \ e^{\frac{S}{c_v}})$$
will play significant roles in our analysis on the higher order  regularities of $(u,S)$. Due to  this observation, we first introduce a proper class of solutions called regular solutions to  the Cauchy problem \eqref{8} with \eqref{2} and \eqref{QH}-\eqref{7}.
\begin{definition} \label{def21} Let $T>0$ be a finite constant. The triple $(\rho,u,S)$  is called a regular solution to the Cauchy problem \eqref{8} with \eqref{2} and \eqref{QH}-\eqref{7}   in $[0,T]\times\mathbb{R}^3$  if $(\rho,u,S)$ satisfies this problem in the sense of distributions and{\rm:}
\begin{enumerate}
\item
$\rho>0,\hspace{2mm}\rho^{\gamma-1}\in C([0,T]; D^1_*\cap D^3),\hspace{2mm}\nabla\rho^{\delta-1}\in C([0,T]; L^\infty \cap D^2)$;\\
 \item $ u\in C([0,T];H^3)\cap L^2([0,T];H^4),\quad  \rho^{\frac{\delta-1}{2}}\nabla u \in C([0,T];L^2),$\\[6pt]
\quad $  \rho^{\delta-1}\nabla u\in L^\infty([0,T];D^1_*),\quad  \rho^{\delta-1}\nabla^2 u\in  L^\infty([0,T];H^1)\cap L^2([0,T];D^2)  $;\\
\item $S\in L^\infty([0,T]\times\mathbb{R}^3),\quad  e^{\frac{S}{c_v}}-e^{\frac{\bar{S}}{c_v}}\in C([0,T];D^1_*\cap D^3)$.
\end{enumerate}

\end{definition}
\begin{remark}\label{rz1}
First, it follows from  Definition \ref{def21}  that $\nabla \rho^{\delta-1}\in L^\infty$, which means that the vacuum occurs if and only  in the far field.

Second, we introduce some physical quantities to  be used in this paper:
\begin{align*}
m(t)=&\int \rho(t,\cdot)\quad \textrm{(total mass)},\\
\mathbb{P}(t)=&\int \rho(t,\cdot)u(t,\cdot) \quad \textrm{(momentum)},
\end{align*}
\begin{align*}
E_k(t)=&\frac{1}{2}\int \rho(t,\cdot)|u(t,\cdot)|^{2}\quad \textrm{ (total kinetic energy)},\\
E_p(t)=&\frac{1}{\gamma-1}\int_{} P(\rho(t,\cdot), S(t,\cdot))  \ \text{ (potential energy)},\\[6pt] 
\displaystyle
E(t)=&E_k(t)+E_p(t) \  \text{ (total energy)}.
\end{align*}
It then  follows from Definition \ref{def21} that a regular solution satisfies the conservation of   total mass, momentum and total energy (see Lemma \ref{lemmak}).
Note that the conservation of  momentum or total energy is not  clear   for strong solutions with far field vacuum  to the flows of constant $(\mu,\lambda,\kappa)$ obtained in \cites{CK, huangli, lx1,lx2,lx3, wenzhu}.
In this sense, the definition of regular solutions above  is consistent with the physical background of the compressible Navier-Stokes equations.
\end{remark}

The regular solutions select  $(\rho, u,S)$  in a physically reasonable way when the density approaches the  vacuum at far fields. With the help of the regularity assumptions in Definition \ref{def21}, $\eqref{8}$ can be reformulated into a system consisting of: a transport equation for the density,  a special quasi-linear parabolic system with some  singular source terms for the velocity, and  a transport equation  with one   singular source term for the entropy. Furthermore,  the coefficients in front  of the   Lam\'e operator $ L$ will tend to $\infty$ as $\rho\rightarrow 0$ at far fields. It will be shown that the problem becomes trackable by studying a suitably  designed enlarged system  through an  elaborate linearization and  approximation process.

The first main result in this paper can be stated as follows.
\begin{theorem} \label{th21}Let parameters $(\gamma,\nu,\alpha,\beta)$ satisfy
\begin{equation}\label{can1}
\gamma>1,\quad 0<\delta=(\gamma-1)\nu<1,\quad\alpha>0,\quad2\alpha+3\beta\geq 0,\quad 4\gamma+3\delta \leq 7.
\end{equation}
If the initial data $(\rho_0,u_0,S_0)$ satisfy
\begin{equation}\label{2.7}\begin{aligned}
&\rho_0>0,\quad\rho_0^{\gamma-1}\in  D^1_*\cap D^3,\quad\nabla\rho_0^{\delta-1}\in L^q\cap D^{1,3}\cap D^2,\\
&\nabla\rho_0^{\frac{\delta-1}{2}}\in L^6,\quad u_0\in H^3, \quad S_0-\bar{S}\in D^1_*\cap D^3,\\
\end{aligned}
\end{equation}
for some $q\in (3,\infty)$,
and the   compatibility conditions:
\begin{equation}\label{2.8}\begin{aligned}
&\nabla u_0=\rho_0^{\frac{1-\delta}{2}}g_1,\quad 
&Lu_0=\rho_0^{1-\delta}g_2,\\
&\nabla(\rho_0^{\delta-1}Lu_0)=\rho_0^{\frac{1-\delta}{2}}g_3,\quad &\nabla^2e^{\frac{S_0}{c_v}}=\rho_0^\frac{1-\delta}{2}g_4,
\end{aligned}\end{equation}
for some functions $(g_1,g_2,g_3,g_4)\in L^2$, then there exist a time $T_*>0$ and a unique regular solution $(\rho,u,S)$ in $[0,T_*]\times \mathbb{R}^3$  to the Cauchy problem \eqref{8} with \eqref{2} and \eqref{QH}-\eqref{7}, and the following additional regularities hold:
 \begin{equation}\label{2.9}\begin{aligned}
& \rho^{\gamma-1}_t \in C([0,T_*];H^2),\quad u_t\in C([0,T_*];H^1)\cap L^2([0,T_*];D^2),\\[2pt] &t^{\frac{1}{2}}u\in L^\infty([0,T_*];D^4),\quad  t^{\frac{1}{2}}u_t\in  L^\infty([0,T_*];D^2)\cap L^2([0,T_*];D^3),\\[2pt]
&u_{tt}\in L^2([0,T_*];L^2),\quad t^{\frac{1}{2}}u_{tt}\in L^\infty([0,T_*];L^2)\cap L^2([0,T_*];D^1_*),\\[2pt]
&  \nabla\rho^{\delta-1}\in C([0,T_*]; L^q \cap D^{1,3}\cap D^2),\quad   S-\bar{S}\in C([0,T_*];D^1_*\cap D^3),\\[2pt]
&S_t\in L^\infty([0,T_*];L^\infty\cap L^3\cap D^1_*\cap D^2),\quad S_{tt}\in L^2([0,T_*];D^1_*).
\end{aligned}\end{equation}
Moreover, it holds that 
\begin{enumerate}

 \item $(\rho,u,S)$ is a classical solution to  the  problem \eqref{8} with \eqref{2} and \eqref{QH}-\eqref{7} \rm(or the problem   \eqref{1}-\eqref{3} with \eqref{4} and \eqref{6}-\eqref{7}\rm) in $(0,T_*]\times \mathbb{R}^3$;
\item if one assumes   $m(0)<\infty$ additionally, then $(\rho,u,S)$ satisfies the conservation of   total mass, momentum and total energy:
$$  m(t)=m(0),\quad  \mathbb{P}(t)=\mathbb{P}(0),\quad  E(t)=E(0) \quad \text{for} \quad t\in [0,T_*]. $$
\end{enumerate}

\end{theorem}

\begin{remark} $\eqref{2.7}$-$\eqref{2.8}$ identifies a class of admissible initial data that provide unique solvability to  the problem \eqref{8} with \eqref{2} and \eqref{QH}-\eqref{7}. Such initial  data include
\begin{equation}\label{example}
\rho_0(x)=\frac{1}{(1+|x|^2)^{\varkappa}},\quad u_0(x)\in C^3_0(\mathbb{R}^3),\quad S_0=\bar{S}+f(x),
\end{equation}
for some $f(x)\in D^1_*\cap D^3$, where 
\begin{equation}\label{exampledege}
\hspace{2mm}\frac{1}{4(\gamma-1)}<\varkappa<\frac{1-3/q}{2(1-\delta)}\quad \text{and} \quad \frac{3}{2}+\frac{\delta}{2}<\gamma+\delta\leq 2.
\end{equation}

Note that, when  $\mu$, $\lambda$ and $\kappa$  are all constants,  in order to obtain the uniform boundedness of the entropy in the whole space, in  \cite{lx1, lx2,lx3}, it is  required  that   the initial density vanishes only at far fields with a rate no more than $O(\frac{1}{|x|^2})$, i.e., 
\begin{equation}\label{examplecons}
\hspace{2mm}\max\Big\{\frac{1}{4(\gamma-1)},\frac{3}{4}\Big\}<\varkappa \leq 1\quad \text{and} \quad \gamma>\frac{5}{4}.
\end{equation}
It follows from \eqref{exampledege}-\eqref{examplecons} that, different from the constant viscous flows,  the requirement on the decay rate of $\rho_0$ of the degenerate viscous flow depends on the value of   the parameter $\delta$. If $\delta$ is sufficient close to $1$, then $\varkappa>1$ in \eqref{example} is still admissible, which  means that the range of initial data can be enlarged substantially.
\end{remark}

\begin{remark}\label{rxiangrong}
 The compatibility conditions \eqref{2.8} are also necessary for the existence of  regular  solutions $(\rho, u,S)$ obtained in Theorem \ref{th21}.
In particular,
\begin{itemize}
          \item $\nabla  u_0=\rho_0^{\frac{1-\delta}{2}}g_1$ plays a key role in the derivation of $\rho^{\frac{1-\delta}{2}} \nabla u \in L^\infty([0,T_*];L^2)$; \\
        \item  $Lu_0=\rho_0^{1-\delta}g_2$ is crucial  in the derivation of  $u_t \in L^\infty([0,T_*];L^2)$,  which will be used in the uniform estimates for  $|u|_{D^2}$;  \\
        \item while $$\nabla(\rho_0^{\delta-1}Lu_0)=\rho_0^{\frac{1-\delta}{2}}g_3 \quad \text{and} \quad \nabla^2e^{\frac{S_0}{c_v}}=\rho_0^\frac{1-\delta}{2}g_4$$
        are indispensible in the derivation of $\rho^{\frac{\delta-1}{2}}\nabla u_t \in  L^\infty([0,T_*];L^2)$,  which will be used in the uniform estimates for  $|u|_{D^3}$.
\end{itemize}
\end{remark}

A natural and important  question is   whether the local   solution obtained  in Theorem \ref{th21} can be extended globally in time. In contrast to the classical theory for both constant and degenerate  viscous flows \cites{KA2, kelao, mat, sundbye2, wenzhu}, we show the following somewhat surprising phenomenon that such an extension is impossible if the velocity field decays to zero as $t\rightarrow \infty$,  the laws of  conservation of  total mass and  momentum are both satisfied,  and the initial total momentum is non-zero.
First, based on the  physical quantities introduced  above,  we define  a solution class $ D(T)$  as follows:
\begin{definition}\label{d2}
Let $T>0$ be a positive time.  For   the Cauchy problem \eqref{1}-\eqref{3} and \eqref{6}-\eqref{7},  a classical solution     $(\rho,u,S)$ is said to be in $ D(T)$ if $(\rho,u,S)$  satisfies the following conditions:
\begin{itemize}
 \item  $m(t)$,  $\mathbb{P}(t)$ and $E_k(t)$    all belong to  $L^\infty([0,T])$;\\

 \item Conservation of total mass:  $\frac{d}{dt}m(t)=0$ for  any $ t\in [0,T]$;\\

 \item  Conservation of momentum:   $\frac{d}{dt}\mathbb{P}(t)=0$  for  any $ t\in [0,T]$.

\end{itemize}
\end{definition}

Then one has:
\begin{theorem}\label{th25}
Assume $m(0)>0$, $|\mathbb{P}(0)|>0$, and the  parameters  $(\gamma, \mu,\lambda,\kappa)$ satisfy
\begin{equation}\label{canshu-22}
\gamma \geq  1, \quad \mu\geq 0, \quad 2\mu+3\lambda\geq 0,\quad \kappa\geq 0.
\end{equation}  Then for   the Cauchy problem  \eqref{1}-\eqref{3} and \eqref{6}-\eqref{7}, there is no classical solution $(\rho,u,S)\in D(\infty)$  with
\begin{equation}\label{eq:2.15}
\limsup_{t\rightarrow \infty} |u(t,\cdot)|_{\infty}=0.
\end{equation}

\end {theorem}

As an immediate corollary of   Theorems \ref{th21}-\ref{th25},  one can obtain that

\begin{corollary}\label{co23}
For   the Cauchy problem \eqref{8} with \eqref{2} and \eqref{QH}-\eqref{7}, if one assumes   $0<m(0)<\infty$ and $|\mathbb{P}(0)|>0$ additionally, then  there is no global  regular  solution $(\rho,u,S)$ with the regularities  in Theorem \ref{th21}   satisfying \eqref{eq:2.15}.

\end {corollary}

Moreover, as a corollary of  Theorem \ref{th25}, Remark \ref{rz1} and \cites{CK,  wenzhu}, one can show that for the \textbf{CNS} with constant  $(\mu, \lambda, \kappa)$, there exists a  global classical solution which preserves   the conservation of total mass,  but  not  the conservation of momentum  for all the time $t\in (0,\infty)$.    Indeed,  consider the Cauchy problem of   \eqref{1}-\eqref{3} with  the following initial data and far field behavior:
\begin{align}
(\rho,u,\theta)|_{t=0}=(\rho_0(x),u_0(x),\theta_0(x))\quad  &\text{for} \quad x\in\mathbb{R}^3, \label{c6}\\[6pt]
(\rho,u,\theta)(t,x)\rightarrow(0,0,0)\quad  \text{as} \quad |x|\rightarrow\infty\quad &\text{for}\quad t\geq 0. \label{c7}
\end{align}

 For simplicity,  $C_0$ will denote a  positive constant that  depends only on fixed constants $\mu$, $\lambda$, $\kappa$, $R$, $\gamma$ and the initial data $(\rho_0,u_0,\theta_0)$.
\begin{corollary}\label{th:2.20-HLX}
Let  $\mu, \lambda$ and $\kappa$ all be constants  satisfying
$$
\mu>0, \quad  2\mu+3\lambda \geq 0, \quad \kappa>0.
$$
For two constants   $\overline{\rho}>0$ and $\overline{\theta}>0$, suppose that the initial data $(\rho_0,u_0,\theta_0)$ satisfies
\begin{equation}\label{initialhxl}
\begin{split}
& 0\leq \inf \rho_0 \leq \sup \rho_0 \leq \overline{\rho},\quad \rho_0\in L^1 \cap H^2\cap W^{2,p}, \\[6pt]
&  0\leq \inf \theta_0 \leq \sup \theta_0 \leq \overline{\theta},\quad (u_0,\theta_0)\in  D^1_*\cap D^2, 
\end{split}
\end{equation}
for some $p\in (3,6)$, and the compatibility conditions
\begin{equation}\label{changxiangrong}
    \begin{cases}
   -\mu \triangle u_0-(\mu+\lambda) \nabla \text{div}u_0+\nabla P(\rho_0,\theta_0)=\rho^{\frac{1}{2}}_0 g_5,\\[8pt]
   \kappa \triangle \theta_0+\frac{\mu}{2}| \nabla u_0+(\nabla u_0)^\top|^2+\lambda (\text{div}u_0)^2=\rho^{\frac{1}{2}}_0 g_6, 
    \end{cases}
\end{equation}
for some $g_i\in L^2$ $(i=5,6)$. Then there exists a positive constant $\zeta$ depending on $\mu$, $\lambda$, $\kappa$, $R$, $\gamma$, $\bar{\rho}$  and $\bar{\theta}$ such that if
\begin{equation}\label{smallness}
m(0)\leq \zeta,
\end{equation}
the Cauchy problem  \eqref{1}-\eqref{3} with \eqref{c6}-\eqref{c7} has a unique global classical solution $(\rho,u,\theta)$ in $(0,\infty)\times \mathbb{R}^3$ satisfying, for any $0<\tau <T<\infty$,
\begin{align}
& m(t)=m(0),\quad  \  t   \in [0,\infty); \\[4pt]
& 0\leq \rho(t,x)\leq 2\overline{\rho},\quad \theta(t,x)\geq 0,  \quad (t,x)   \in [0,\infty)\times \mathbb{R}^3;\label{lagrangian1q}\\[4pt]
&\sup_{t\in [0,T]} \int (\rho |\dot{u}|^2+|\nabla u|^2+|\nabla \theta|^2\big)(t,\cdot)\leq C_0;\label{lagrangian2q}\\[4pt]
& \rho \in C([0,T];H^2\cap W^{2,p}),\quad  \rho_t \in C([0,T];H^1); \label{lagrangian3q}\\[4pt]
&  (u,\theta)\in C([0,T]; D^1_*\cap D^2)\cap L^2([0,T] ; D^3)\cap L^\infty([\tau,T];D^3); \label{lagrangian4q}\\[4pt]
& (u_t,\theta_t)\in L^2([0,T]; D^1_*)\cap L^\infty([0,T]; D^1_*\cap D^2),\label{lagrangian5q}
\end{align}
for $\dot{u}=u_t+u\cdot \nabla u$, and the following large-time behavior
\begin{equation}\label{largetimehlx}
\begin{split}
\lim_{t\rightarrow \infty} \int \big(\rho |\theta|^2+|\nabla u|^2+|\nabla \theta|^2\big)(t,\cdot)=0.
\end{split}
\end{equation}

Furthermore, if  $m(0)>0$ and $|\mathbb{P}(0)|>0$, then the solution obtained above cannot preserve the conservation of the momentum  for all the time $t\in (0,\infty)$.
\end {corollary}

\begin{remark}\label{zhunbei2} Note that for the regular solution $(\rho,u,S)$ obtained in Theorem \ref{th21}, $u$  stays in the inhomogeneous Sobolev space $H^3$ instead of the homogenous one $D^1_*\cap D^2$ in \cite{CK,wenzhu}  for   flows with constant viscosity and heat conductivity coefficients.

 Based on the conclusions obtained in Theorem \ref{th21} and \cite{ins},  there is a natural question  that whether   the conclusion obtained in \cite{ins} can  be applied to the  degenerate  system considered here?
Due to strong degeneracy near the vacuum in $\eqref{1}_2-\eqref{1}_3$,   such  questions are not easy and  will be discussed in the  future work Xin-Zhu \cite{XZ-L2}.

\end{remark}

\begin{remark}
The  framework established  in this paper is applicable to other physical dimensions, say 1 and 2, with some minor modifications. This is clear from the analysis carried out in the following sections.

\end{remark}

The rest of this paper is   organised as follows.   In Section $2$, we first  reformulate the  Cauchy problem \eqref{8} with \eqref{2} and \eqref{QH}-\eqref{7} into a specifically chosen enlarged  form, which makes the problem  trackable through an elaborate linearization and  approximation process. Then we outline the  main strategy to establish  the  well-posedness theory.
Section $3$ is devoted to proving  the  local-in-time well-posedness  theory stated in  Theorem \ref{th21}, which can be  achieved in five steps: 
\begin{enumerate}

 \item construct global approximate solutions away from the vacuum for
a specially designed linearized problem with an artificial viscosity
$$\sqrt{\rho^{\delta-1}+\epsilon^2}e^{\frac{S}{c_v}\nu} Lu$$
in the  momentum equations and 
$\inf_{x\in \mathbb{R}^3}\rho^{\gamma-1}_0=\frac{\gamma-1}{A\gamma}\eta$
for some positive constants $\epsilon>0$ and  $\eta>0$;
\item establish the a priori estimates independent of both $\epsilon$ and $\eta$;

\item  then pass to the limit $\epsilon\rightarrow 0$ to recover the solution
of the corresponding  linearized problem  away from the vacuum with only physical viscosities;

\item  prove the unique solvability away from the vacuum  of the reformulated nonlinear
problem through a standard iteration process;

\item  finally take  the limit $\eta\rightarrow 0$ to recover the solution
of the reformulated nonlinear problem  with  physical viscosities and far field vacuum.
\end{enumerate}
The global non-existence results stated in  Theorem \ref{th25} and Corollary \ref{co23}, and the non-conservation of momentum  stated in Corollary \ref{th:2.20-HLX} are proved in Section 4.
  Finally, for convenience of readers, we provide one appendix to list some basic facts which have been  used frequently in this paper.


\section{Reformulation and main strategy}\label{s2}

In this section, we first reformulate the highly degenerate system   \eqref{8}
into an enlarged  trackable system, and then sketch the  main strategies of our analysis.

\subsection{Reformulation} 
Set $\delta=(\gamma-1)\nu$, and 
\begin{equation}\label{2.1}\begin{aligned}
&\phi=\frac{A\gamma}{\gamma-1}\rho^{\gamma-1},\quad l=e^{\frac{S}{c_v}},\quad \psi=\frac{\delta}{\delta-1}\nabla\rho^{\delta-1},\quad n=\rho^{2-\delta-\gamma}.
\end{aligned}\end{equation}
It follows from  \eqref{8} with \eqref{2} and \eqref{QH}-\eqref{7} that 
\begin{equation}\label{2.3}\left\{\begin{aligned}
\displaystyle
&\  \phi_t+u\cdot\nabla\phi+(\gamma-1)\phi \text{div}u=0,\\[6pt]
\displaystyle
&\  u_t+u\cdot \nabla u+a_1\phi\nabla l+l\nabla\phi+a_2l^\nu\phi^{2\iota} Lu\\[6pt]
\displaystyle
=&a_2\phi^{2\iota}\nabla l^\nu\cdot Q(u)+a_3l^\nu\psi  \cdot  Q(u),\\[6pt]
\displaystyle
&\  l_t+u\cdot\nabla l=a_4l^\nu n\phi^{4\iota}H(u),\\[4pt]
\displaystyle
&\ \psi_t+\sum_{k=1}^3 A_k(u) \partial_k\psi+B(u)\psi+\delta a\phi^{2\iota}\nabla \text{div} u=0,
\end{aligned}\right.\end{equation}
where
\begin{equation}\label{2.2}
\begin{split}
a_1=&\frac{\gamma-1}{\gamma},\quad a_2=a\Big(\frac{A}{R}\Big)^\nu,\quad a_3=\Big(\fr{A}{R}\Big)^\nu,\\
a_4=&\fr{A^{\nu-1}a^2(\gamma-1)}{R^\nu},\quad 
\iota=\fr{\delta-1}{2(\gamma-1)},\quad a=\Big(\fr{A\gamma}{\gamma-1}\Big)^{\fr{1-\delta}{\gamma-1}},
\end{split}
\end{equation}
 $A_k(u)=(a^k_{ij})_{3\times 3}$ for  $i$, $j$, $k=1$, $2$, $3$,
are symmetric  with
$$a^k_{ij}=u^{(k)}\quad \text{for}\ i=j;\quad \text{otherwise}\  a^k_{ij}=0, $$
 and $B(u)=(\nabla u)^\top+(\delta-1)\text{div}u\mathbb{I}_3$.

\eqref{2.3} will be regarded as a system for the new variable $(\phi,u,l,\psi)$, which is equivalent to the \eqref{8} with \eqref{2} and \eqref{QH}-\eqref{7} if one imposes the initial data
\begin{equation}\label{2.4}
\begin{split}
&(\phi,u,l,\psi)|_{t=0}=(\phi_0,u_0,l_0,\psi_0)\\
=&\Big(\fr{A\gamma}{\gamma-1}\rho_0^{\gamma-1}(x),u_0(x),e^{S_0(x)/c_v},\fr{\delta}{\delta-1}\nabla\rho_0^{\delta-1}(x)\Big)\quad  \text{for} \quad x\in\mathbb{R}^3,
\end{split}
\end{equation}
and   far filed behavior:
\begin{equation}\label{2.5}
(\phi,u,l,\psi)\rightarrow (0,0,\bar{l},0) \hspace{2mm} \text{as} \hspace{2mm}|x|\rightarrow \infty \hspace{2mm} \text{for} \quad t\geq 0,
\end{equation}
where $\bar{l}=e^{\frac{\bar{S}}{c_v}}>0$ is some positive constant.

Note that  the reformulated system \eqref{2.3}   consists of  
\begin{itemize}
 \item  one scalar transport equation $\eqref{2.3}_1$ for $\phi$;

 \item one  {\it singular parabolic}  system  $\eqref{2.3}_2$ for the velocity $u$;

 \item  one scalar transport equation   $\eqref{2.3}_3$ but with a possible singular first order term for $l$;

 \item  one symmetric hyperbolic system  $\eqref{2.3}_4$ but with a possible singular second order term for $\psi$.
\end{itemize}
Such a structure is important in establishing the local well-posedness of the problem \ef{2.3}-\ef{2.5} and thus to 
 solve the Cauchy problem \eqref{8} with \eqref{2} and \eqref{QH}-\eqref{7} locally in time.  We will establish  the following theorem.

\begin{theorem}\label{3.1} Let $\ef{can1}$ hold. If the initial data $(\phi_0,u_0,l_0,\psi_0)$ satisfies:
\begin{equation}\label{a}\begin{aligned}
&\phi_0>0,\quad \phi_0\in D^1_*\cap D^3,\quad u_0\in H^3,\quad l_0-\bar{l}\in D^1_*\cap D^3,\\
&\inf_{x\in \mathbb{R}^3} l_0>0, \quad \psi_0\in L^q \cap D^{1,3}\cap D^2,\quad \nabla\phi_0^\iota\in L^6,
\end{aligned}\end{equation}
for some $q\in (3,\infty)$, and the   compatibility conditions:
\begin{equation}\label{2.8*}\begin{aligned}
\nabla u_0=&\phi_0^{-\iota}g_1,&\quad Lu_0=\phi_0^{-2\iota}g_2,\\
\nabla(\phi_0^{2\iota}Lu_0)=&\phi_0^{-\iota}g_3,&\quad \nabla^2l_0=\phi_0^{-\iota}g_4,
\end{aligned}\end{equation}
for some $(g_1,g_2,g_3,g_4)\in L^2$, then there exist a time $T^*>0$ and a unique strong solution $(\phi,u,l,\psi=\fr{a\delta}{\delta-1}\nabla\phi^{2\iota})$  in $[0,T^*]\times \mathbb{R}^3$   to the Cauchy problem \ef{2.3}-\ef{2.5}, such that 
\begin{equation}\label{b}\begin{aligned}
&\phi\in C([0,T^*];D^1_*\cap D^3),\quad \psi\in  C([0,T^*];L^q \cap D^{1,3}\cap D^2),\\
&u\in C([0,T^*];H^3)\cap L^2([0,T^*];H^4),\quad \phi^{2\iota}\nabla u\in L^\infty ([0,T^*];D^1_*),\\
&\phi^{2\iota}\nabla^2 u\in C([0,T^*];H^1)\cap L^2([0,T^*];D^2),\\
&\phi^\iota\nabla u\in C([0,T^*];L^2),\quad\phi^\iota\nabla u_t\in L^\infty([0,T^*];L^2),\\
&u_t\in  C([0,T^*];H^1)\cap L^2([0,T^*];D^2),\quad (\phi^{2\iota}\nabla^2u)_t\in L^2([0,T^*];L^2),\\
&u_{tt}\in L^2([0,T^*];L^2),\quad t^{\fr{1}{2}}u\in L^\infty([0,T^*];D^4),\\
&t^{\fr{1}{2}}u_t\in L^\infty([0,T^*];D^2)\cap L^2([0,T^*];D^3),\quad  l-\bar{l}\in  C([0,T^*];D^1_* \cap D^3),\\
& \inf_{(t,x)\in [0,T^*]\times \mathbb{R}^3} l>0, \quad l_t\in  C([0,T^*];D^1_*\cap D^2),\quad l_{tt}\in L^2([0,T^*];D^1_*).
\end{aligned}\end{equation}
\end{theorem}
\begin{remark}\label{strongsolution}
In Theorem \ref{3.1}, $(\phi,u,l,\psi=\fr{a\delta}{\delta-1}\nabla\phi^{2\iota})$  in $[0,T^*]\times \mathbb{R}^3$ is called a strong solution   to the Cauchy problem \ef{2.3}-\ef{2.5}, if it satisfies \ef{2.3}-\ef{2.5} in the sense of distributions, and  satisfies the  equations \eqref{2.3}-\ef{2.2} a.e. $(t,x)\in (0,T^*]\times \mathbb{R}^3$.
\end{remark}

\subsection{Main strategy}
Now we sketch the  main strategy for the proof of Theorem \ref{3.1}. 
\subsubsection{Closed energy estimates based on the singular structure introduced}
We now formally  indicate   how to obtain  closed energy estimates based on the singular structure described above. 

Note first that  the   velocity $u$  can be controlled by the following equations:
$$
u_t+u\cdot \nabla u+a_1\phi\nabla l+l\nabla\phi+\underbrace{a_2\phi^{2\iota}l^\nu Lu}_{\textbf{SSE}}
=\underbrace{a_2\phi^{2\iota}\nabla l^\nu\cdot Q(u)+a_3l^\nu\psi \cdot  Q(u)}_{\textbf{SSS}},
$$
where \textbf{SSE} denotes the elliptic operator which is strongly singular near the vacuum, while \textbf{SSS} represents the source terms which are also highly singular near the vacuum. 
Due to \eqref{4}, $\kappa=0$, so the entropy is expected to be bounded below uniformly such that  $l=e^{\frac{S}{c_v}}$ and $\phi^{2\iota}$ with $\iota<0$  should  have   uniformly positive lower bounds in the whole space. Then
for this   quasi-linear parabolic system, one can find formally that  even though the coefficients $a_2\phi^{2\iota} l^\nu$ in front  of  Lam\'e operator $Lu$ will tend to $\infty$ as $\rho\rightarrow 0$ in the far filed, yet this structure could give a better a priori estimate on $u$ in $H^3$ than those of \cites{CK,sz3,sz333,zz}  if one can control  the possible singular term $\psi$ in $L^q\cap D^{1,3}\cap D^2$, $l-\bar{l}$ in $D^1_* \cap D^3$ and the  first-order product term  
$\phi^{2\iota}\nabla l^\nu\cdot Q(u)$
with singular coefficient   in proper spaces successfully. 

On the one hand,  $\ef{2.3}_4$
implies that the subtle  term $\psi$  could be controlled by a  symmetric hyperbolic  system with a possible singular higher order term $\delta a\phi^{2\iota}\nabla \text{div} u$, while
$l$ is governed by  $\eqref{2.3}_3$,
which is a scalar  transport equation with a  possible singular term $a_4l^\nu n\phi^{4\iota} H(u)$.
Thus in order to close the desired  estimates, one needs to control both  $\phi^{2\iota}\nabla \text{div} u$ in $L^q\cap D^{1,3}\cap D^2$ and $l^\nu n\phi^{4\iota} H(u)$ in $H^2$.

First, the necessary   estimates on $\phi^{2\iota}\nabla \text{div} u$    can be obtained by regarding the momentum equations as the following inhomogeneous Lam\'e  equations:
\begin{equation*}
\begin{split}
a_2 L(\phi^{2\iota}u)=&-l^{-\nu}(u_t+u\cdot \nabla u+a_1\phi\nabla l+l\nabla\phi)+a_3\psi\cdot Q(u) \\
&+a_2l^{-\nu}\phi^{2\iota}\nabla l^\nu\cdot Q(u)-\frac{\delta-1}{\delta}\Big(\frac{A}{R}\Big)^\nu G(\psi,u)=W,
\end{split}
\end{equation*}
where
$$
G(\psi,u)= \alpha  \psi\cdot \nabla u+\alpha \text{div}(u\otimes \psi)+(\alpha+\beta)\big(\psi \text{div}u+\psi \cdot \nabla u+u  \cdot \nabla \psi\big).
$$
In fact, one has
\begin{equation}\label{singularelliptic}
\begin{split}
|\phi^{2\iota}\nabla^2u|_{D^1_*}\leq & C(|\psi|_\infty|\nabla^2u|_2+|\phi^{2\iota}\nabla^3u|_2)\\
\leq & C(|\psi|_\infty|\nabla^2u|_2+|\nabla \psi|_3|\nabla u|_6+|\nabla^2\psi|_2|u|_\infty+|W|_{D^1_*}),
\end{split}
\end{equation}
for some constant  $C>0$ independent of the lower bound of $\phi$ provided that
$$
\phi^{2\iota}u\rightarrow 0\qquad \text{as} \qquad |x| \rightarrow \infty,
$$
which can be  verified  in  an approximation process from non-vacuum flows  to the flow with   far field vacuum. Similar calculations can be done for  $|\phi^{2\iota}\nabla^2u|_{D^2}$.

Next, we turn to the estimates on   $l^\nu n\phi^{4\iota} H(u)$, which are more complicated and depend on the  estimates of $n$ and $\phi^{4\iota}|\nabla u|^2$. 
An observation used here  is that the initial assumption (\ref{2.7}) and  the definition of $n$ in \eqref{2.1} imply that
$$
n(0,x) \in L^\infty\cap D^{1,q}\cap D^{1,6} \cap D^{2,3} \cap D^3.
$$
It is easy to check  that $n$ can be controlled by the following  hyperbolic equation:
\begin{equation}\label{eqofn}
n_t+u\cdot\nabla n+(2-\delta-\gamma)n\dv u=0,
\end{equation}
which, along with the expected regularities of $u$, implies that
$$
n(t,x) \in L^\infty\cap D^{1,q}\cap D^{1,6} \cap D^{2,3} \cap D^3
$$
within the solution's life span.
In fact, it follows from  the assumption \eqref{can1}  that the upper bound of $n$ depends on that of $\phi$,  and the  estimates on  its space  derivatives can be obtained from  those of $\psi$ via the relation $
n=(a\phi^{2\iota})^{\frac{2-\delta-\gamma}{\delta-1}}$, which, along with the equation \eqref{eqofn} and the  weighted  estimates on $u$, yields  the  estimates on  its time   derivatives.  While $\phi^{4\iota} |\nabla u|^2$ can be controlled by using the  weighted  estimates on $u$ including $|\phi^{\iota}\nabla u|_2$, $|\phi^{\iota}\nabla u_t|_2$, $|\phi^{2\iota}\text{div}u|_\infty$,  $|\phi^{2\iota}\nabla^2u|_{D^2}$, $|\phi^{2\iota}\nabla^2u_t|_{2}$ and so on.

Finally,  in order to deal with  the high singularity  and  nonlinearity of the source term  $\phi^{2\iota}\nabla l^\nu\cdot Q(u)$ in the time evolution equations  $\eqref{2.3}_2$ of $u$,  one still needs one singular weighted estimate on the entropy: $|\phi^{\frac{\iota}{2}}\nabla l|_6$, which can be obtained from the transport structure of the equation $\eqref{2.3}_3$ of $l$. It is worth pointing out that for achieving this key estimate, we require that $4\gamma +3\delta\leq 7$ in \eqref{2.1}.


\subsubsection{An  elaborate linearization of the nonlinear singular problem} To prove Theorem \ref{3.1}, it is crucial to carry out the strategy of energy estimates discussed above for suitably chosen approximate solutions while are constructed by an elaborate linear scheme.  In \S 3.1,   we design  an  elaborate linearization \eqref{ln}  of the nonlinear one \ef{2.3}-\ef{2.5} based on a careful analysis on the structure of the nonlinear system \eqref{2.3}, and the  global approximate solutions for this linearized problem when $\phi(0,x)=\phi_0$ has positive lower bound $\eta$ are established.  The choice of the linear scheme for this problem needs to be  careful due to the appearance of the far field vacuum.
Some  necessary structures should be preserved  in order  to establish the desired a priori estimates as  mentioned above.  For the  problem \eqref{2.3}-\eqref{2.5}, a crucial point is how to deal with the estimates on $\psi$. According to the analysis in the above paragraphs, we need to keep the two factors $\phi^{2\iota}$  and $\nabla \text{div} u$ of the source term $\delta a\phi^{2\iota}\nabla \text{div} u$  in equations $\ef{2.3}_4$ in the same step. Then let $v=(v^{(1)},v^{(2)}, v^{(3)})^\top\in \mathbb{R}^3$ be  a  known vector, $g$ and $w$ be  known real (scalar) functions satisfying $(v(0,x), g(0,x),w(0,x))=(u_0, \phi^{2\iota}_0,l_0)$ and (\ref{4.1*}). A natural linearization of the system \eqref{2.3} seems to be 
  \begin{equation}
\begin{cases}
\label{eq:cccq-fenxi}
\displaystyle
\  \ \phi_t+v\cdot\nabla\phi+(\gamma-1)\phi \text{div}v=0,\\[2pt]
\displaystyle
\  \ u_t+v\cdot \nabla v+a_1\phi\nabla l+l\nabla\phi+a_2\phi^{2\iota}l^\nu Lu\\[2pt]
\displaystyle
=a_2g\nabla l^\nu\cdot Q(v)+a_3l^\nu \psi \cdot  Q(v),\\[2pt]
\displaystyle
\  \  l_t+v\cdot\nabla l=a_4 w^\nu  n  g^2H(v),\\[2pt]
\displaystyle
\  \ \psi_t+\sum_{k=1}^3 A_k(v) \partial_k\psi+B(v)\psi+\delta a g\nabla \text{div} v=0.
 \end{cases}
\end{equation}

However, it  should be noted  that, in (\ref{eq:cccq-fenxi}),   the important relationship
 $$
 \psi=\frac{a\delta}{\delta-1}\nabla \phi^{2\iota}
  $$
  between $\psi$ and $\phi$ cannot be guaranteed  due to the term  $g \nabla \text{div} v$ in   $(\ref{eq:cccq-fenxi})_4$. Then one would encounter the following difficulties in deriving    $L^2$ estimate for  $u$:
  \begin{equation}\label{relationvainish}
  \begin{split}
&\frac{1}{2} \frac{d}{dt}|u|^2_2+a_2\alpha |l^{\frac{\nu}{2}}\phi^{\iota}\nabla u|^2_2+a_2(\alpha+\beta)|l^{\frac{\nu}{2}}\phi^{\iota}\text{div} u|^2_2\\
=&-\int \big(v\cdot \nabla v+a_1\phi\nabla l+l\nabla\phi  +a_2 l^\nu  \underbrace{\nabla \phi^{2\iota}}_{\neq \frac{\delta-1}{a\delta}\psi}\cdot Q(u)  \big)\cdot u\\
&+\int \big(-a_2  \phi^{2\iota} \nabla l^\nu \cdot Q(u) +a_2g\nabla l^\nu\cdot Q(v)+a_3l^\nu \psi \cdot Q(v) \big)\cdot u.
  \end{split}
  \end{equation}
Since  $\nabla \phi^{2\iota}$ does not coincide with $\frac{\delta-1}{a\delta}\psi$ in \eqref{eq:cccq-fenxi}, it seems extremely difficult  to control the term $a_2 l^\nu   \nabla \phi^{2\iota}\cdot Q(u)$ in the above energy estimates.
In order to overcome this difficulty, in  (\ref{ln}),   we  first linearize the equation for  $h=\phi^{2\iota}$ as:
\begin{equation}\label{h}
h_t+v\cdot \nabla h+(\delta-1)g\text{div} v=0,
\end{equation}
 and then use $h$ to define 
 $\psi=\frac{a\delta}{\delta-1}\nabla h$ again. Here, it should be pointed out that, due to  the term $(\delta-1)g\text{div} v$ in \eqref{h}, the relation  $h=\phi^{2\iota}$  between $h$ and $\phi$ no  longer exists in the linear problem.
 The linear equations for $u$ are chosen as
  \begin{equation*}
  \begin{split}
&u_t+v\cdot \nabla v+a_1\phi\nabla l+l\nabla\phi+a_2\sqrt{h^2+\epsilon^2}l^\nu Lu\\
=& a_2 g\nabla l^\nu \cdot Q(v)+a_3l^\nu \psi  \cdot Q(v),
  \end{split}
  \end{equation*}
for any positive constant $\epsilon >0$. Here the appearance of $\epsilon$ is used to compensate the lack of lower bound of $h$. It follows from (\ref{h}) and the relation  $\psi=\frac{a\delta}{\delta-1}\nabla h$ that
$$
\psi_t+\sum_{k=1}^3 A_k(v) \partial_k\psi+(\nabla v)^\top\psi+a\delta \big(g\nabla \text{div} v+\nabla g \text{div} v\big)=0,
$$
which  turns out to be the right   structure to ensure the  desired estimates on $\psi$.

Another subtle issue is that  to linearize the  equation  $\eqref{2.3}_3$ for the entropy, one would face the problem how to define $n$ (by $\phi$ or $h$?) since the relation $h=\phi^{2\iota}$ does not hold for the linearized scheme above.
Here, in order to make full use of the estimates on $\psi$ and the singular weighted estimates on $u$, we will define $n$ as 
$$
n=(ah)^{\frac{2-\delta-\gamma}{\delta-1}}.$$

Then in \S 3.2,    the uniform a priori estimates independent for the lower bound $(\epsilon,\eta)$ of the solutions $(\phi,u,l,h)$  to  the linearized problem   (\ref{ln}) are established. Based on these uniform estimates, one can obtain the local-in-time well-posedness of the regular solution with far field vacuum  by passing to the limit $\epsilon \rightarrow 0$, an iteration process that connects the linear approximation  problems to the nonlinear one, and finally  passing to the limit as $\eta \rightarrow 0$.

\section{Local-in-time well-posedness with far field vacuum}
In this section,  the proof for  the local-in-time well-posedness of strong  solutions with far field vacuum  stated in   Theorem \ref{3.1} will be given.
\subsection{Linearization}
Let $T>0$ be some positive time. In order to solve the nonlinear problem  $\ef{2.3}$-$\ef{2.5}$, we consider the following  linearized problem for $(\phi^{\epsilon,\eta}, u^{\epsilon,\eta}, \\ l^{\epsilon,\eta}, h^{\epsilon,\eta})$  in $[0,T]\times \mathbb{R}^3$:
\begin{equation}\label{ln}\left\{\begin{aligned}
&\phi^{\epsilon,\eta}_t+v\cdot\nabla\phi^{\epsilon,\eta}+(\gamma-1)\phi^{\epsilon,\eta} \text{div}v=0,\\[2pt]
&u^{\epsilon,\eta}_t+v\cdot \nabla v+a_1\phi^{\epsilon,\eta}\nabla l^{\epsilon,\eta}+l^{\epsilon,\eta}\nabla\phi^{\epsilon,\eta}+a_2(l^{\epsilon,\eta})^\nu \sqrt {(h^{\epsilon,\eta})^2+\epsilon^2} Lu^{\epsilon,\eta}\\
=&a_2g\nabla (l^{\epsilon,\eta})^\nu\cdot Q(v)+a_3(l^{\epsilon,\eta})^\nu \psi^{\epsilon,\eta} \cdot Q(v),\\[2pt]
&l^{\epsilon,\eta}_t+v\cdot\nabla l^{\epsilon,\eta}=a_4w^\nu n^{\epsilon,\eta} g^2H(v),\\[2pt]
&h^{\epsilon,\eta}_t+v\cdot \nabla h^{\epsilon,\eta}+(\delta-1)g\text{div}v=0,\\[2pt]
&(\phi^{\epsilon,\eta},u^{\epsilon,\eta},l^{\epsilon,\eta},h^{\epsilon,\eta})|_{t=0}=(\phi^\eta_0,u^\eta_0,l^\eta_0,h^\eta_0)\\[2pt]
=&(\phi_0+\eta,u_0,l_0,(\phi_0+\eta)^{2\iota})\quad  \text{for} \quad x\in\mathbb{R}^3,\\[2pt]
&(\phi^{\epsilon,\eta},u^{\epsilon,\eta},l^{\epsilon,\eta},h^{\epsilon,\eta})\rightarrow (\eta,0,\bar{l},\eta^{2\iota}) \quad \text{as} \hspace{2mm}|x|\rightarrow \infty \quad \rm for\quad t\geq 0,
\end{aligned}\right.\end{equation}
where $\epsilon$ and $\eta$ are  positive constants, 
\begin{equation}\label{2.11}
 \psi^{\epsilon,\eta}=\fr{a\delta}{\delta-1}\nabla h^{\epsilon,\eta},\quad n^{\epsilon,\eta}=(ah^{\epsilon,\eta})^b,\quad b=\fr{2-\delta-\gamma}{\delta-1}\leq 0,
\end{equation}
$v=(v^{(1)},v^{(2)},v^{(3)})^\top \in\mathbb{R}^3$ is a given vector, $g$ and $w$ are given real functions satisfying $w\geq 0$, $(v(0,x),g(0,x),w(0,x))=(u_0(x),h_0(x)=(\phi^\eta_0)^{2\iota}(x),l_0(x))$, and
\begin{equation}\label{4.1*}
\begin{split}
&g\in L^\infty\cap C([0,T]\times \mathbb{R}^3),\quad \nabla g\in C([0,T];L^q\cap D^{1,3}\cap D^2),\\[3pt]
& g_t\in C([0,T];H^2),\quad \nabla g_{tt}\in L^2([0,T];L^2),\\[3pt]
& v\in C([0,T];H^3)\cap L^2([0,T];H^4),\quad t^{\fr{1}{2}}v\in L^\infty([0,T];D^4),\\[3pt]
&v_t\in C([0,T];H^1)\cap L^2([0,T];D^2),\quad v_{tt}\in L^2([0,T];L^2),\\[3pt]
&  t^{\fr{1}{2}}v_t\in L^\infty([0,T];D^2)\cap L^2([0,T];D^3),\\[3pt]
& t^{\fr{1}{2}}v_{tt}\in L^\infty([0,T];L^2)\cap L^2([0,T];D^1_*),\quad w\in C([0,T];D^1_*\cap D^3),\\[3pt]
& w_t\in C([0,T];D^1_*\cap D^2),\quad \nabla w_{tt}\in  L^2([0,T];L^2).
\end{split}
\end{equation}
For the sake of clarity, we declare that the functions $(\phi_0,u_0,l_0)$ and the constant $\bar{l}=e^{\frac{\bar{S}}{c_v}}>0$ shown above in problem \eqref{ln} is exactly the ones in  \eqref{2.4}-\eqref{2.5}, and also define here  $$\psi_0=\frac{a\delta}{\delta-1}\nabla \phi^{2\iota}_0.$$

It follows from the classical theory \cites{KA,oar,amj}, at least when $\eta$
and $\epsilon$ are positive, that the following global well-posedness of \ef{ln} in $[0,T]\times \mathbb{R}^3$ holds.
\begin{lemma}\label{ls}
	Let $\ef{can1}$ hold, $\eta>0$ and $\epsilon>0$ and $(\phi_0,u_0,l_0,h_0)$ satisfy 
	\eqref{a}-\eqref{2.8*}. Then for any time $T>0$,  there exists a unique strong solution $(\phi^{\epsilon,\eta},u^{\epsilon,\eta},l^{\epsilon,\eta},h^{\epsilon,\eta})$ in $[0,T]\times \mathbb{R}^3$ to $\ef{ln}$ such that
	\begin{equation}\label{2.13}\begin{aligned}
	&\phi^{\epsilon,\eta}-\eta\in C([0,T];D^1_*\cap D^3), \quad  \phi^{\epsilon,\eta}_t\in C([0,T];H^2), \\[3pt]
	& h^{\epsilon,\eta}\in L^\infty\cap C([0,T]\times \mathbb{R}^3),\quad  \nabla h^{\epsilon,\eta}\in C([0,T];H^2),\\[3pt]
	& h^{\epsilon,\eta}_t\in C([0,T];H^2),\quad 	u^{\epsilon,\eta}\in C([0,T];H^3)\cap L^2([0,T];H^4),\\[3pt]
	& u^{\epsilon,\eta}_t\in C([0,T];H^1)\cap L^2([0,T];D^2),\quad  u^{\epsilon,\eta}_{tt}\in L^2([0,T];L^2),\\[3pt]
	& t^{\fr{1}{2}}u^{\epsilon,\eta}\in L^\infty([0,T];D^4), \quad t^{\fr{1}{2}}u^{\epsilon,\eta}_t\in L^\infty([0,T];D^2)\cap L^2([0,T];D^3),\\[3pt]
 & t^{\fr{1}{2}}u^{\epsilon,\eta}_{tt}\in L^\infty([0,T];L^2)\cap L^2([0,T];D^1_*),\quad  l^{\epsilon,\eta}-\bar{l}\in C([0,T];D^1_*\cap D^3),\\[3pt]
& l^{\epsilon,\eta}_t\in  C([0,T^*];D^{1}_*\cap D^2),\quad \nabla l^{\epsilon,\eta}_{tt}\in L^2([0,T^*];L^2).
	\end{aligned}\end{equation}
\end{lemma}
 
Now we are going to derive the uniform a priori estimates, independent of $(\epsilon,\eta)$, for the strong  solution $(\phi^{\epsilon,\eta},u^{\epsilon,\eta},l^{\epsilon,\eta},h^{\epsilon,\eta})$ to $\ef{ln}$ obtained  in Lemma \ref{ls}.

\subsection{A priori estimates independent of $(\epsilon,\eta)$}

For any fixed $\eta\in(0,1]$, since 
$$(\phi^\eta_0,u^\eta_0,l^\eta_0,h^\eta_0)=(\phi_0+\eta,u_0,l_0,(\phi_0+\eta)^{2\iota}),$$ $(\phi_0,u_0,l_0,h_0)$  satisfy 
	\eqref{a}-\eqref{2.8*}  and $\psi_0=\frac{a\delta}{\delta-1}\nabla \phi^{2\iota}_0$, there exists a constant $c_0>0$ independent of $\eta$ such that
\begin{equation}\label{2.14}\begin{aligned}
&2+\eta+\bar{l}+\|\phi^\eta_0-\eta\|_{D^1_*\cap D^3}+\|u^\eta_0\|_{3}+\|\nabla h^\eta_0\|_{L^q\cap D^{1,3}\cap D^2}+|\nabla (h^\eta_0)^{\frac{1}{2}}|_6
\\
&|(h^\eta_0)^{-1}|_\infty+|g^\eta_1|_2+|g^\eta_2|_2+|g^\eta_3|_2+|g^\eta_4|_2+\|l^\eta_0-\bar{l}\|_{D^1_*\cap D^3}+|(l^\eta_0)^{-1}|_\infty\leq c_0,
\end{aligned}\end{equation}
where 
\begin{equation*}
g^\eta_1=(\phi^\eta_0)^{\iota}\nabla u^\eta_0,\quad g^\eta_2=(\phi^\eta_0)^{2\iota}Lu^\eta_0,\quad 
g^\eta_3=(\phi^\eta_0)^{\iota}\nabla((\phi^\eta_0)^{2\iota}Lu^\eta_0),\quad g_4^\eta=(\phi^\eta_0)^{\iota}\nabla^2l^\eta_0.
\end{equation*}

\begin{remark}\label{r1}
First, it follows from  the definition of $g^\eta_2$, $\phi^\eta_0>\eta$ and the far field behavior as in $\ef{ln}_6$  that 
	\begin{equation}\label{incom}\left\{\begin{aligned}
	&L((\phi^\eta_0)^{2\iota}u^\eta_0)=g^\eta_2-\fr{\delta-1}{a\delta}G(\psi^\eta_0,u^\eta_0),\\
&(\phi^\eta_0)^{2\iota}u^\eta_0\longrightarrow 0 \ \ \text{as}\ \ |x|\longrightarrow \infty,\\
	\end{aligned}\right.\end{equation}
	where $\psi^\eta_0=\frac{a\delta}{\delta-1}\nabla (\phi^\eta_0)^{2\iota}=\frac{a\delta}{\delta-1}\nabla h^\eta_0$, and 
\begin{equation}\label{Gdingyi}
G=\alpha\psi^\eta_0\cdot\nabla u^\eta_0+\alpha \text{div}(u^\eta_0\otimes\psi^\eta_0)+(\alpha+\beta)(\psi^\eta_0\text{div} u^\eta_0+\psi^\eta_0\cdot\nabla u^\eta_0+u^\eta_0\cdot\nabla\psi^\eta_0).
\end{equation}
Then it follows from the standard elliptic theory and $\ef{2.14}$ that
\begin{equation}\label{incc}
\begin{split}
|(\phi^\eta_0)^{2\iota}u^\eta_0|_{D^2}\leq & C(|g^\eta_2|_2+|G(\psi^\eta_0,u^\eta_0)|_2)\leq C_1<\infty,\\
|(\phi^\eta_0)^{2\iota}\nabla^2u^\eta_0|_2\leq & C(|(\phi^\eta_0)^{2\iota}u^\eta_0|_{D^2}+|\nabla\psi^\eta_0|_3|u^\eta_0|_6+|\psi^\eta_0|_\infty|\nabla u^\eta_0|_2)\leq C_1,
\end{split}
\end{equation}
where $C_1>0$ is a generic  constant independent of $(\epsilon,\eta)$. Due to $\nabla^2\phi^{2\iota}_0\in L^3$ and $\ef{incc}$, it holds that 
\begin{equation}\label{incc1}
|(\phi^\eta_0)^\iota\nabla^2\phi^\eta_0|_2+|(\phi^\eta_0)^\iota\nabla(\psi^\eta_0\cdot Q(u^\eta_0))|_2\leq C_1,
\end{equation}
where one has used the fact that
\begin{equation*}
\begin{split}
|\phi_0^\iota\nabla^2\phi_0 |_2\leq & C_1(|\phi_0|_6|\phi_0|^{-\iota}_\infty|\nabla^2\phi^{2\iota}_0|_3+|\nabla \phi^\iota_0|_6|\nabla \phi_0|_3)\leq  C_1,\\
|(\phi^\eta_0)^\iota\nabla^2\phi^\eta_0|_2=& \Big|\phi_0^\iota\nabla^2\phi_0 \frac{\phi_0^{-\iota}}{(\phi_0+\eta)^{-\iota}}\Big|_2\leq  |\phi_0^\iota\nabla^2\phi_0 |_2\leq C_1.
\end{split}
\end{equation*}

Second, it follows from  the initial  compatibility condition 
$$\nabla((\phi^\eta_0)^{2\iota}Lu^\eta_0)=(\phi^\eta_0)^{-\iota}g^\eta_3\in L^2,$$ 
that formally,
\begin{equation}\label{incom1}\left\{\begin{aligned}
	&L((\phi^\eta_0)^{2\iota}u^\eta_0)=\triangle^{-1}\text{div}((\phi^\eta_0)^{-\iota}g^\eta_3)-\fr{\delta-1}{a\delta}G(\psi^\eta_0,u^\eta_0),\\
&(\phi^\eta_0)^{2\iota}u^\eta_0\longrightarrow 0 \ \ \text{as}\ \ |x|\longrightarrow \infty.\\
	\end{aligned}\right.\end{equation}
Thus the standard  elliptic theory yields
\begin{equation}\label{incc*}
\begin{split}
|(\phi^\eta_0)^{2\iota}u^\eta_0|_{D^3}\leq & C(|\phi^\eta_0)^{-\iota}g^\eta_3|_2+|G(\psi^\eta_0,u^\eta_0)|_{D^1})\leq C_1<\infty,\\
|(\phi^\eta_0)^{2\iota}\nabla^3u^\eta_0|_2\leq & C(|(\phi^\eta_0)^{2\iota}u^\eta_0|_{D^3}+|\nabla\psi^\eta_0|_3|\nabla u^\eta_0|_6\\
&+|\psi^\eta_0|_\infty|\nabla^2 u^\eta_0|_2+|\nabla^2\psi_0^\eta|_2|u_0^\eta|_\infty)\leq C_1.
\end{split}
\end{equation}
Actually, the  rigorous proof for \ef{incom1} can be obtained by a standard smoothing process of the initial data,  which is  omitted here.
	\end{remark}
Now let $T$ be a positive fixed constant, and  assume that there exist some time $T^*\in(0,T]$ and constants $c_i (i=1,\cdots,5)$ such that
\begin{equation}\label{2.15}
1<c_0\leq c_1\leq c_2\leq c_3\leq c_4\leq c_5,
\end{equation}
and
\begin{equation}\label{2.16}\begin{aligned}
\sup_{0\leq t\leq T^*}\|\nabla g(t)\|^2_{L^q\cap D^{1,3}\cap D^2}\leq c_1^2,\quad  \sup_{0\leq t\leq T^*}\|w(t)-\bar{l}\|_{D^1_*\cap D^3}^2\leq c^2_1,&\\
\inf_{[0,T_*]\times \mathbb{R}^3} w(t,x)\geq c^{-1}_1,\quad  \sup_{0\leq t\leq T^*}\|v(t)\|^2_{1}+\int^{T^*}_0(|v|^2_{D^2}+|v_t|^2_2)\text{d}t\leq c_2^2,&\\
\sup_{0\leq t\leq T^*}(|v|^2_{D^2}+|v_t|^2_2+|g\nabla^2v|^2_2)(t)+\int^{T^*}_0(|v|_{D^3}^2+|v_t|^2_{D^1_*})\text{d}t\leq c_3^2,&\\
\sup_{0\leq t\leq T^*}(|v|^2_{D^3}+|\sqrt{g}\nabla v_t|^2_2+|\nabla v_t|^2_2)(t)+\int^{T^*}_0(|v|^2_{D^4}+|v_t|^2_{D^2}+|v_{tt}|^2_2)\text{d}t\leq c_4^2,&\\
\sup_{0\leq t\leq T^*}(|g_t|^2_{D^1_*}+|g\nabla^2v|^2_{D^1_*})(t)+\int^{T^*}_0(|(g\nabla^2v)_t|^2_2+|g\nabla^2v|^2_{D^2})\text{d}t\leq c_4^2,&\\
\sup_{0\leq t\leq T^*}(|g_t|^2_\infty+ |w_t|^2_3+|w_t|^2_{D^1_*})(t)\leq c^2_4,&\\
\text{ess}\sup_{0\leq t\leq T^*}(t|v|^2_{D^4}+t|\nabla^2v_t|^2_2+t|g\nabla^2v_t|^2_2)(t)+\int^{T^*}_0|g_{tt}|^2_{D^1_*}\text{d}t\leq c^2_5,&\\ 
  \sup_{0\leq t\leq T^*}( |w_t|^2_\infty+|w_t(t)|^2_{D^2})+\int^{T^*}_0|w_{tt}|^2_{D^1_*}\text{d}t\leq c^2_5,&\\
\text{ess}\sup_{0\leq t\leq T^*}t|v_{tt}(t)|^2_2+\int^{T^*}_0t(|v_{tt}|^2_{D^1_*}+|\sqrt{g}v_{tt}|_{D^1_*}^2+|v_t|^2_{D^3})\text{d}t\leq c_5^2.&
\end{aligned}
\end{equation}
$T^*$ and $c_i(i=1,\cdots,5)$ will be determined later, and depend only on $c_0$ and the fixed constants $(A, R, c_v, \alpha,\beta,\gamma,\delta, T)$.
Hereinafter,  $ M(c)\geq 1$ will denote  a generic continuous  and increasing function on $[0,\infty)$, and $C\geq 1$ will denote  a generic positive constant. Both $M(c)$ and $C$  depend only on fixed constants $(A, R, c_v, \alpha, \beta, \gamma, \delta, T)$, and  may be different from  line to line. 
Moreover, in  the rest of \S 3.2, without causing ambiguity,
we simply  drop the superscript $\epsilon$ and $\eta$ in 
$(\phi^\eta_0,u^\eta_0,l^\eta_0,h^\eta_0,\psi^\eta_0)$, 
$(\phi^{\epsilon,\eta},u^{\epsilon,\eta},l^{\epsilon,\eta},h^{\epsilon,\eta},\psi^{\epsilon,\eta})$, and $(g_1^\eta,g_2^\eta,g_3^\eta,g_4^\eta)$.

\subsubsection{The a priori estimates for $\phi$.} 

In the rest of \S 3.2, let  $(\phi,u,l,h)$ be the unique classical solution to $\ef{ln}$ in $[0,T]\times\mathbb{R}^3$ obtained in Lemma \ref{ln}. 
\begin{lemma}\label{phiphi}
	
	\begin{equation}\label{phi1}\begin{aligned}
	&\|\phi(t)-\eta\|_{D^1_*\cap D^3}\leq Cc_0,\quad|\phi_t(t)|_2\leq Cc_0c_2,\quad|\phi_t(t)|_{D^1_*}\leq Cc_0c_3,\\
	&|\phi_t(t)|_{D^2}\leq Cc_0c_4,\quad|\phi_{tt}(t)|_2\leq Cc_4^3,\quad\int^t_0\|\phi_{tt}(s)\|^2_1 \text{d}s\leq Cc_0^2c_4^2,
	\end{aligned}\end{equation}
	for $0\leq t\leq T_1=\min\{T^*,(1+Cc_4)^{-6}\}$.
\end{lemma}
\begin{proof} First, it follows directly from  $\ef{ln}_1$  that, for $0\leq t\leq T_1$, 
\begin{equation}\label{phiinf}
|\phi|_\infty\leq |\phi_0|_\infty\exp\big({C\int^t_0|\text{div} v|_\infty \text{d}s}\big)\leq Cc_0.
\end{equation}

Second,
the standard energy estimates  for transport equations, $\ef{2.16}$ and $\ef{phiinf}$ yield that,  for $0\leq t\leq T_1$,
\begin{equation*}\begin{aligned}
\|\phi-\eta\|_{D^1_*\cap D^3}\leq &C(\|\phi_0-\eta\|_{D^1_*\cap D^3}+\eta\int^t_0\|\nabla v\|_3\text{d}s)\exp\Big(\int^t_0C\|v\|_4 \text{d}s\Big)
\leq Cc_0.
\end{aligned}\end{equation*}

This, together with $\ef{ln}_1$, yields
 that for $0\leq t\leq T_1$, 
\begin{equation}\label{phit}\left\{\begin{aligned}
&|\phi_t(t)|_2\leq C(|v|_3|\nabla\phi|_6+|\phi|_\infty|\nabla v|_2)\leq Cc_0c_2,\\
&|\phi_t(t)|_{D^1_*}\leq C(|v|_\infty|\nabla^2\phi|_2+|\nabla\phi|_6|\nabla v|_3+|\phi|_\infty|\nabla^2 v|_2)\leq Cc_0c_3,\\
&|\phi_t(t)|_{D^2}\leq C\|v\|_3(\|\nabla\phi\|_2+|\phi|_\infty) \leq Cc_0c_4.
\end{aligned}\right.\end{equation}

Similarly, it follows from 
\begin{equation*}
\phi_{tt}=-v_t\cdot\nabla\phi-v\cdot\nabla\phi_t-(\gamma-1)\phi_t\dv v-(\gamma-1)\phi \dv v_t
\end{equation*}
and $\ef{2.16}$ that for $0\leq t\leq T_1$,
\begin{equation*}\begin{aligned}
|\phi_{tt}|_2\leq & C(|v_t|_3|\nabla\phi|_6+|v|_\infty|\nabla\phi_t|_2+|\nabla v|_\infty|\phi_t|_2+|\phi|_\infty|\nabla v_t|_2)\leq Cc_4^3,\\
\int^t_0\|\phi_{tt}\|^2_1\text{d}s\leq & \int_0^t( \|v_t\cdot\nabla\phi\|_1+\|v\cdot\nabla\phi_t\|_1+\|\phi_t\dv v\|_1+\|\phi \dv v_t\|_1)^2\text{d}s \leq  Cc_0^2c_4^2.
\end{aligned}\end{equation*}

The proof of Lemma \ref{phiphi} is complete.
\end{proof}

\subsubsection{The a priori estimates for $\psi$.}
The following estimates for $\psi$ are needed  to deal with the degenerate elliptic operator.
\begin{lemma}\label{psi} 	\rm {For} $t\in [0, T_1]$ and  $q>3$, it holds that 
	\begin{equation}\label{2.24}\begin{aligned}
	&|\psi(t)|_\infty^2+\|\psi(t)\|^2_{L^q\cap D^{1,3}\cap D^2}\leq Cc_0^2,\hspace{2mm}|\psi_t(t)|_2\leq Cc_3^2,\\
	&|h_t(t)|^2_\infty\leq Cc_3^3c_4,\quad |\psi_t(t)|^2_{D^1_*}+\int^t_0(|\psi_{tt}|^2_2+|h_{tt}|^2_6)\text{d}s\leq Cc_4^4.
	\end{aligned}\end{equation}

\end{lemma}
\noindent
\begin{proof} It follows from  $\psi=\frac{a\delta}{\delta-1}\nabla h$  and   $\ef{ln}_4$ that 
\begin{equation}\label{psieq}
\psi_t+\sum_{k=1}^3 A_k(v) \partial_k\psi+B^*(v)\psi+a\delta ( g\nabla \text{div} v +\nabla g \text{div} v ) =0,
\end{equation}
with  $B^*(v)=(\nabla v)^\top$ and $A_k(v)$ defined in \eqref{2.3}.

First, multiplying $(\ref{psieq})$ by $q \psi |\psi|^{q-2}$ and integrating over $\mathbb{R}^3$ yield that 
\begin{equation}\label{psiq}
\begin{split}
\fr{d}{dt}|\psi|^q_q \leq & C(|\nabla v|_\infty|\psi|^q_q+|\dv v|_\infty|\nabla g |_q|\psi|_q^{q-1}+|g\nabla^2v|_{q}|\psi|_q^{q-1})\\
\leq & C(|\nabla v|_\infty|\psi|^q_q+|\dv v|_\infty|\nabla g |_q|\psi|_q^{q-1}+\|g\nabla^2v\|_{2}|\psi|_q^{q-1}).
\end{split}
\end{equation}

According to  \eqref{2.16}, one can obtain that 
\begin{equation*}
\int^t_0\|g\nabla^2v\|_2\text{d}s\leq t^{\fr{1}{2}}\big(\int^t_0\|g\nabla^2v\|^2_2\text{d}s\big)^{\fr{1}{2}}\leq c_4t^{\fr{1}{2}},
\end{equation*}
which, together  with \eqref{psiq} and  Gronwall's inequality,  yields that 
\begin{equation*}
|\psi(t)|_q \leq Cc_0 \quad \text{for} \quad  0\leq t\leq T_1.
\end{equation*}

Second, set $\varsigma=(\varsigma_1,\varsigma_2,\varsigma_3)^\top$ ($ |\varsigma|=1$ and $\varsigma_i=0,1$). Applying  $\partial_{x}^{\varsigma} $ to $\ef{psieq}$,
multiplying by $3|\partial_{x}^{\varsigma} \psi|\partial_{x}^{\varsigma} \psi$ and then integrating over $\mathbb{R}^3$, one can get
\begin{equation}\label{2.26}\begin{split}
\frac{d}{dt}|\partial_{x}^{\varsigma}  \psi|^3_3
\leq & \Big(\sum_{k=1}^{3}|\partial_{k}A_k(v)|_\infty+|B^*(v)|_\infty\Big)|\partial_{x}^{\varsigma}  \psi|^3_3+C|\Theta_\varsigma |_3|\partial_{x}^{\varsigma}  \psi|^2_3,
\end{split}
\end{equation}
where
\begin{equation*}
\begin{split}
\Theta_\varsigma=\partial_{x}^{\varsigma} (B^*\psi)-B^*\partial_{x}^{\varsigma}  \psi+\sum_{k=1}^{3}\big(\partial_{x}^{\varsigma} (A_k \partial_k \psi)-A_k \partial_k\partial_{x}^{\varsigma}  \psi\big)+ a\delta \partial_{x}^{\varsigma}\big(g\nabla \text{div} v+\nabla g \text{div} v\big).
\end{split}
\end{equation*}

On the other hand, for $|\varsigma|=2$ and $\varsigma_i=0,1,2$, applying  $\partial_{x}^{\varsigma} $ to $\ef{psieq}$, multiplying by $2\partial_{x}^{\varsigma} \psi$ and then integrating over $\mathbb{R}^3$ lead to 
\begin{equation}\label{2.26*}\begin{split}
\frac{d}{dt}|\partial_{x}^{\varsigma}  \psi|^2_2
\leq & \Big(\sum_{k=1}^{3}|\partial_{k}A_k(v)|_\infty+|B^*(v)|_\infty\Big)|\partial_{x}^{\varsigma}  \psi|^2_2+C|\Theta_\varsigma |_2|\partial_{x}^{\varsigma}  \psi|_2.
\end{split}
\end{equation}

For $|\varsigma|=1$, it is easy  to obtain
\begin{equation}\label{zhen2}
\begin{split}
|\Theta_\varsigma |_3\leq &C\big(|\nabla^2 v|_3(|\psi|_\infty+|\nabla g|_\infty)+ |\nabla v|_\infty( |\nabla\psi|_3+|\nabla^2 g|_3)
+|\nabla(g\nabla^2 v)|_{3}\big).
\end{split}
\end{equation}
Similarly, for $|\varsigma|=2$, one has
\begin{equation}\label{zhen2*}
\begin{split}
|\Theta_\varsigma |_2\leq& C\big(|\nabla v|_\infty(|\nabla^2\psi|_2+|\nabla^3 g|_2)+|\nabla^2 v|_6(|\nabla\psi|_3+|\nabla^2 g|_3)\big)\\
&+C|\nabla^3 v|_2 (|\psi|_\infty+|\nabla g|_\infty)+C|g\nabla \text{div}v|_{D^2}.
\end{split}
\end{equation}
It follows from $\ef{2.26}$-$\ef{zhen2*}$ and the Gagliardo-Nirenberg inequality
\begin{equation*}
|\psi|_\infty\leq C |\psi|^\xi_q|\nabla\psi|^{1-\xi}_6\leq C|\psi|_q^\xi|\nabla^2\psi|_2^{1-\xi} \quad \text{with} \quad \xi=\fr{q}{6+q},
\end{equation*}
that
\begin{equation*}
\frac{d}{dt}\|\psi(t)\|_{D^{1,3}\cap D^2}\leq Cc_4\|\psi(t)\|_{D^{1,3}\cap D^2}+C|g\nabla \text{div}v|_{D^2}+C c^2_4,
\end{equation*}
which, along with   Gronwall's inequality,  implies that for $0\leq t \leq T_1$,
\begin{equation}\label{2.26a}\begin{split}
\|\psi(t)\|_{D^{1,3}\cap D^2}\leq&  \Big(c_0+Cc^2_4t+C\int_0^t |g\nabla \text{div}v|_{D^2} \text{d}s\Big) \exp(Cc_4t)\leq Cc_0.
\end{split}
\end{equation}

Next, due to  $\ef{psieq}$, it holds that  for  $0\leq t \leq T_1$,
\begin{equation*}\left\{\begin{aligned}
&|\psi_t(t)|_2\leq C\big(|\nabla v|_2| \psi|_{D^{1,3}}+|\nabla v|_2|\psi|_{\infty}+|g\nabla^2 v|_2+|\nabla g|_\infty |\nabla v|_2\big)\leq Cc^2_3,\\[10pt]
&|\nabla \psi_t(t)|_{2}\leq C \big(\|v\|_{3}(\|\psi\|_{L^q\cap D^{1,3}\cap D^2}+\|\nabla g\|_{L^q\cap D^{1,3}\cap D^2})+|g\nabla^2 v|_{D^1_*}\big) \leq Cc^2_4.
\end{aligned}\right.
\end{equation*}
Similarly, via the relation
$$\psi_{tt}=-\nabla (v \cdot \psi)_t-a\delta \big(g\nabla \text{div} v+\nabla g \text{div} v\big)_t,$$
for $0\leq t \leq T_1$,  one gets
\begin{equation}\label{2.26c}\begin{split}
\int_0^t |\psi_{tt}|^2_{2} \text{d}s
\leq& C\int_0^t \Big(|v_t|^2_6|\nabla \psi|^2_3+|\nabla v|^2_\infty|\psi_t|^2_{2}+|v|^2_\infty|\nabla \psi_t|^2_2+|\psi|^2_\infty|\nabla v_t|^2_{2}\\
&+|(g \nabla \text{div} v)_t|^2_{2}+|\nabla g|^2_\infty|\nabla v_t|^2_2+|\nabla v|^2_\infty|\nabla g_t|^2_2\Big) \text{d}s
\leq Cc^4_4.
\end{split}
\end{equation}

Finally, it follows from Gagliardo-Nirenberg inequality and $\ef{2.16}$ that
\begin{equation}\label{2.26d}\begin{split}
|g\text{div} v|_\infty\leq & C|g\text{div} v|^{\frac{1}{2}}_{D^1}|g\text{div} v|^{\frac{1}{2}}_{D^2}
\leq  C\big(|\nabla g|_\infty|\nabla v|_2+|g\nabla^2 v|_2\big)^{\frac{1}{2}}\\
&\cdot \big(|\nabla^2 g|_2|\nabla v|_\infty+|\nabla g|_\infty|\nabla^2 v|_2+|g\nabla^2v|_{D^1_*}\big)^{\frac{1}{2}}
\leq  Cc^{\frac{3}{2}}_3c^{\frac{1}{2}}_4.
\end{split}
\end{equation}

Then, together   with $\ef{ln}_4$,  it yields  that for $0\leq t \leq T_1$,
\begin{equation*}\begin{split}
|h_t(t)|_\infty\leq C(|v|_\infty|\psi|_\infty+|g\text{div}v|_\infty)\leq Cc^{\frac{3}{2}}_3c^{\frac{1}{2}}_4,&\\
\int_0^t |h_{tt}|^2_6\text{d}s\leq C\int_0^t \big(|v|_\infty|\psi_t|_6+|v_t|_6|\psi|_\infty+|g_{t}|_\infty|\nabla v|_6+|g\nabla v_t|_6\big)^2 \text{d}s\leq  Cc^4_4,&
\end{split}
\end{equation*}
where  one has used the fact that
\begin{equation}\begin{split}\label{gnablavt}
|g\nabla v_t|_6\leq& C\big(|\nabla g|_\infty|\nabla v_t|_2+| g\nabla^2 v_t|_2)\\
\leq &C\big(|\nabla g|_\infty|\nabla v_t|_2+| (g\nabla^2 v)_t|_2+| g_t|_\infty|\nabla^2 v|_2).
\end{split}
\end{equation}

The proof of Lemma \ref{psi} is complete.
\end{proof}

\subsubsection{The a priori estimates for h-related auxiliary variables}  Set
$$ 
\varphi=h^{-1}\quad \text{and} \quad 
\quad n=(ah)^b=a^bh^{\frac{2-\delta-\gamma}{\delta-1}}.
$$
\begin{lemma}\label{varphi}\rm {For} $t\in [0, T_1]$ and  $q>3$, it holds that 
	\begin{equation}\label{2.34}\begin{aligned}
	h(t,x)>\fr{1}{2c_0},\quad\fr{2}{3}\eta^{-2\iota}<\varphi(t,x)<2|\varphi_0|_\infty\leq & 2c_0,\\
	|\nabla\sqrt{h}(t)|_6\leq Cc_0,\quad 
 	\|n(t)\|_{ L^\infty\cap D^{1,q}\cap D^{1,6} \cap D^{2,3} \cap D^3}\leq & M(c_0),\\
	|n_t(t)|_\infty\leq M(c_0)c_4^2,\quad |n_t(t)|_6\leq & M(c_0)c_3^2, \\
	|\nabla n_t(t)|_3\leq M(c_0)c_4^2,\quad |\nabla n_t(t)|_6\leq &M(c_0)c_4^2.
	\end{aligned}\end{equation}

\end{lemma}
\noindent
\begin{proof} \textbf{Step 1:} Estimates on $\varphi$.
Note that 
\begin{equation}\label{2.34a}
\varphi_t+v\cdot \nabla \varphi-(\delta-1)g\varphi^2\text{div} v=0.
\end{equation}
Let $X(t;x)$ be  the particle path defined by
\begin{equation}\begin{cases}
\label{2.34b}
\frac{d}{ds}X(t;x)=v(s,X(t; x)),\quad  0\leq t\leq T;\\[8pt]
X(0;x)=x, \quad \qquad \qquad \ \ \quad   x\in \mathbb{R}^3
\end{cases}
\end{equation}
Then 
\begin{equation}\label{2.34c}
\begin{split}
\displaystyle
\varphi(t,X(t;x))=\varphi_0(x)\Big(1+(1-\delta)\varphi_0(x)\int_0^t g\text{div}v (s,X(s;x))\text{d}s\Big)^{-1}.
\end{split}
\end{equation}
This, along with $\ef{2.26d}$, implies that 
\begin{equation}\label{2.34d}
\begin{split}
\displaystyle
\frac{2}{3}\eta^{-2\iota}<\varphi(t,x)<2|\varphi_0|_\infty\leq 2c_0\quad \text{for} \quad  [t,x]\in [0,T_1]\times \mathbb{R}^3.
\end{split}
\end{equation}

\noindent\textbf{Step 2:} Estimates on $\nabla \sqrt{h}$.
It follows from  $\ef{ln}_4$ that 
\begin{equation}\label{2.23b}
(\sqrt{h})_t+v\cdot\nabla\sqrt{h}+\fr{1}{2}(\delta-1)h^{-\fr{1}{2}}g\text{div}v=0,
\end{equation}
which implies 
\begin{equation}\label{2.23c}
(\nabla\sqrt{h})_t+\nabla(v\cdot\nabla\sqrt{h})+\fr{1}{2}(\delta-1)(\nabla h^{-\fr{1}{2}}g\text{div}v+h^{-\fr{1}{2}}\nabla(g\text{div}v))=0.
\end{equation}
Multiplying $\ef{2.23c}$ by $6|\nabla\sqrt{h}|^4\nabla\sqrt{h}$ and  integrating with respect to $x$ over $\mathbb{R}^3$ yield
\begin{equation}\label{2.23d}\begin{split}
\fr{d}{dt}|\nabla\sqrt{h}|^6_6 \leq& C|\nabla v|_\infty|\nabla\sqrt{h}|^6_6+C\big(|\varphi|^{\fr{1}{2}}_\infty(|g\nabla^2v|_6+|\nabla g|_\infty|\nabla v|_6)|\nabla\sqrt{h}|^5_6 \\
& +|\varphi|^{\fr{3}{2}}_\infty|g\text{div}v|_6|\psi|_\infty|\nabla\sqrt{h}|^5_6\big).
\end{split}
\end{equation}
Integrating $\ef{2.23d}$ with respect to $t$ and using  $\ef{2.16}$, \eqref{2.34d} and Lemma \ref{psi} lead to 
\begin{equation}\label{2.23e}
|\nabla\sqrt{h}(t)|_6\leq Cc_0\quad \text{for} \quad 0\leq t\leq T_1.
\end{equation}


\noindent\textbf{Step 3:} Estimates on $n$. Since 
$n=(ah)^b$, then 
\begin{equation}\label{neq}
n_t+v\cdot\nabla n+(2-\delta-\gamma)a^b h^{b-1}g\dv v=0.
\end{equation}

Then  it follows from Lemma \ref{psi}, \eqref{2.34d} and \eqref{2.23e} that for $0\leq t \leq T_1$,
\begin{equation*}
\begin{split}
&|n|_\infty\leq  a^b|\varphi|_\infty^{-b}\leq M(c_0),\quad |\nabla n|_q=a^b|bh^{b-1}\nabla h|_q\leq M(c_0),\\
&|\nabla n|_6=2a^b|h^{b-\fr{1}{2}}\nabla \sqrt{h}|_6\leq M(c_0),\\
&|\nabla^2n|_6\leq  C(|h^{b-1}\nabla^2h|_6+|h^{b-\frac{3}{2}}\nabla h \cdot \nabla \sqrt{h}|_6)\leq M(c_0),\\
&|\nabla^3n|_2\leq  C(|h^{b-1}\nabla^3 h|_2+|h^{b-\frac{3}{2}}\nabla^2 h \cdot \nabla \sqrt{h}|_2+|h^{b-\frac{3}{2}}| \nabla \sqrt{h}|^3|_2)\leq M(c_0),\\
&|\nabla^2 n|_3\leq  C(|h^{b-1}\nabla^2h|_3+|h^{b-1}| \nabla \sqrt{h}|^2|_3)\leq M(c_0),\\
&|n_t|_\infty\leq C(|v|_\infty|\nabla n|_\infty+|\varphi^{1-b}|_\infty|g \text{div} v|_\infty)\leq M(c_0)c_4^2,\\
&|n_t|_6\leq  C(|v|_\infty|\nabla n|_6+|\varphi^{1-b}|_\infty|g \text{div} v|_6)\leq M(c_0)c_3^2,\\
&|\nabla n_t|_{3}\leq C(|\nabla v\cdot \nabla n|_3+|v\cdot \nabla^2 n|_3+|h^{b-1}\nabla g \text{div}v|_3)\\
&\qquad \ \ \ \ +C(|h^{b-\frac{3}{2}} \nabla \sqrt{h} g \text{div}v|_3+|h^{b-1} g \nabla \text{div}v|_3)\leq  M(c_0)c_4^2,\\
&|\nabla n_t|_{6}\leq  C(|\nabla v\cdot \nabla n|_6+|v\cdot \nabla^2 n|_6+|h^{b-1}\nabla g \text{div}v|_6)\\
&\qquad \ \ \ \ +C(|h^{b-1} \nabla h g \text{div}v|_6+|h^{b-1} g \nabla \text{div}v|_6)\leq  M(c_0)c_4^2.
\end{split}
\end{equation*}

The proof of Lemma \ref{varphi} is complete.
\end{proof}

\subsubsection{The a priori estimates for $l$} We now turn to the estimates involving the entropy.
\begin{lemma}\label{l}Set $T_2=\min\{T_1,(1+Cc_4)^{-20-2\nu}\}$. Then for $t\in [0,T_2]$, 
	\begin{equation}\label{lt}\begin{split}
	c_0^{-1}\leq l(t,x)\leq M(c_0),\quad\|l-\bar{l}\|_{D^1_*\cap D^3}\leq & M(c_0),\\
	|l_t|_{\infty}\leq   M(c_0)c_3^{3+\nu}c_4,\quad 	|l_t|_3\leq   M(c_0)c_3^{4+\nu},\quad 
	|l_t|_{D^1_*}\leq  & M(c_0)c_3^{6+\nu}c_4^{\fr{1}{2}},\\
 |\nabla^2l_t|_2\leq 	 M(c_0)c_3^{8+\nu}c_4^{\fr{3}{2}},\quad 
 \int^t_0| l_{tt}|^2_6dt\leq & M(c_0)c_4^{6+2\nu},\\
 \int^t_0|\nabla l_{tt}|^2_2dt\leq  M(c_0)c_4^{10+2\nu},\quad
	|h^{\fr{1}{4}}\nabla l|_6\leq & M(c_0).
	\end{split}\end{equation}

\end{lemma}
\begin{proof} \textbf{Step 1:} It follows from the equation for entropy, $\ef{ln}_3$ that 
\begin{equation*}
\fr{dl(t,X(t,x))}{dt}=a_4w^\nu ng^2 H(v)=a_4w^\nu ng^2(2\alpha|Dv|^2+\beta|\text{div}v|^2)\geq 0,
\end{equation*}
which yields that 
\begin{equation}\label{linf}
l(t,X(t,x))\geq l_0(x)\geq c_0^{-1} \quad \text{for} \quad 0\leq t \leq T_1.
\end{equation}

\noindent\textbf{Step 2:} Set $\varsigma=(\varsigma_1,\varsigma_2,\varsigma_3)^\top$ with $\varsigma_i$ nonnegative integer and  $|\varsigma|=\varsigma_1+\varsigma_2+\varsigma_3$. Applying  $\partial_{x}^{\varsigma} $ to $\ef{ln}_3$, multiplying by $2\partial_{x}^{\varsigma} l$ and then integrating over $\mathbb{R}^3$, one has 
\begin{equation}\label{parl}\begin{split}
\frac{d}{dt}|\partial_{x}^{\varsigma} l  |^2_2
\leq & |\nabla v|_\infty|\partial_x^{\varsigma}l|_2^2+|\varLambda_\varsigma |_2|\partial_{x}^{\varsigma} l|_2,
\end{split}
\end{equation}
where
$$
\varLambda_\varsigma=\partial_{x}^{\varsigma} (a_4w^\nu ng^2 H(v))-(\partial_x^\varsigma(v\cdot\nabla l)-v\cdot\nabla\partial_x^\varsigma l),$$ 
which  will be estimated according to $|\varsigma|$.

For $|\varsigma|=1$, direct calculations show that 
\begin{equation}\label{nablal}
\begin{split}
|\varLambda_\varsigma|_2\leq &C|w|^\nu_\infty(|\nabla n|_6|g\nabla v|^2_6+|n|_\infty|g\nabla v|_\infty(|\nabla g|_\infty|\nabla v|_2+|g\nabla^2v|_2))\\
&+C(|w^{\nu-1}|_\infty|n|_\infty|g\nabla v|_6^2|\nabla w|_6+|\nabla v|_\infty|\nabla l|_2).
\end{split}
\end{equation}

Next,   for $|\varsigma|=2$, it holds that 
\begin{equation}\label{var2}
\begin{split}
|\varLambda_\varsigma|_2&\leq C(|w^\nu n\nabla^2(g^2H(v))|_2+|\nabla(w^\nu n)\nabla(g^2H(v))|_2\\
&+|\nabla^2(w^\nu n)g^2H(v)|_2+|\partial_x^\varsigma(v\cdot\nabla l)-v\cdot\nabla\partial_x^\varsigma l|_2)=C\sum^4_{i=1}A_i.
\end{split}
\end{equation}
By direct calculations, Sobolev inequalities  and  Lemma \ref{zhen1}, one can obtain 
\begin{equation}\label{a1}
\begin{split}
A_1\leq &C|w^\nu|_\infty|n|_\infty|\nabla^2(g^2H(v))|_2\\
\leq &C|w^\nu|_\infty|n|_\infty(|g\nabla v|_\infty|\nabla^2 g|_3|\nabla v|_6+|g\nabla v|_\infty|\nabla g|_\infty |\nabla^2 v|_2\\
&+|\nabla g|^2_\infty|\nabla v|_\infty|\nabla v|_2+|g\nabla v|_\infty|g\nabla^3 v|_2+|g\nabla^2v|^2_4),\\
A_2\leq& C|w^\nu|_\infty|\nabla n|_\infty|g\nabla v|_\infty(|\nabla g|_\infty|\nabla v|_2+|g\nabla^2 v|_2)\\
&+C|w^{\nu-1}|_\infty|n|_\infty|\nabla w|_\infty|g\nabla v|_\infty(|g\nabla ^2v|_2+|\nabla g|_\infty|\nabla v|_2),\\
A_3\leq& C|w^\nu|_\infty|\nabla^2 n|_6|g\nabla v|^2_6+C|w^{\nu-1}|_\infty|g\nabla v|^2_6(|n|_\infty|\nabla^2w|_6+|\nabla n|_\infty|\nabla w|_6)\\
&+C|w^{\nu-2}|_\infty|n|_\infty|\nabla w|^2_4|g\nabla v|^2_\infty,\\
A_4\leq& C(|\nabla^2v|_3+|\nabla v|_\infty)|\nabla^2l|_2.
\end{split}
\end{equation}

Similarly, for $|\varsigma|=3$, one has
\begin{equation}\label{var3}
\begin{split}
|\varLambda_\varsigma|_2\leq& C(|\nabla^3(w^\nu n)(g^2H(v))|_2+|w^\nu n\nabla^3(g^2H(v))|_2\\
&+|\nabla^2(w^\nu n)\nabla (g^2H(v))|_2
+|\nabla(w^\nu n)\nabla^2(g^2H(v))|_2\\
&+|\partial_x^\varsigma(v\cdot\nabla l)-v\cdot\nabla\partial_x^\varsigma l|_2=C\sum^5_{i=1}B_i,
\end{split}
\end{equation}
and each $B_i$ $(i=1,...,5)$ can be estimated as follows:
\begin{equation}\label{b1}
\begin{split}
B_1\leq& |g\nabla v|_\infty^2|\nabla^3(w^\nu n)|_2\\
\leq& C|g\nabla v|_\infty^2(|w^\nu|_\infty|\nabla^3n|_2+|w^{\nu-1}|_\infty|\nabla^2n|_6|\nabla w|_3\\
&+|w^{\nu-1}|_\infty|n|_\infty|\nabla^3w|_2
+|w^{\nu-1}|_\infty|\nabla n|_6|\nabla^2w|_3\\
&+ |w^{\nu-2}|_\infty|\nabla n|_6|\nabla w|^2_6+ |w^{\nu-2}|_\infty|n|_\infty|\nabla w|_6|\nabla^2 w|_3\\
&+|w^{\nu-3}|_\infty|n|_\infty|\nabla w|^3_6),\\
B_2\leq& |w^\nu|_\infty |n|_\infty|\nabla^3(g^2H(v))|_2\\
\leq &C|w^\nu|_\infty |n|_\infty\big(|g\nabla v|_\infty|g\nabla^4v|_2+|g\nabla^3v|_6|g\nabla^2v|_3\\
&+|\nabla g|^2_\infty|\nabla^2 v|_2|\nabla v|_\infty+|g\nabla v|_\infty|\nabla^2 g|_6|\nabla^2v|_3\\
&+|g\nabla v|_\infty|\nabla ^3g|_2|\nabla v|_\infty+|\nabla g|_\infty|\nabla^2g|_3|\nabla v|_6|\nabla v|_\infty\\
&+|g\nabla^3v|_2|\nabla g|_\infty|\nabla v|_\infty+|g\nabla^2v|_6|\nabla g|_\infty|\nabla^2 v|_3\big),\\
B_3\leq& |\nabla^2(w^\nu n)|_3 |\nabla(g^2H(v))|_6\\
\leq &C\big(|w^\nu|_\infty |\nabla^2n|_3+|n|_\infty(|w^{\nu-1}|_\infty|\nabla^2w|_3+|w^{\nu-2}|_\infty|\nabla w|^2_6)\\
&+|w^{\nu-1}|_\infty|\nabla n|_\infty|\nabla w|_3\big)\cdot(|g\nabla v|_\infty|\nabla g|_\infty|\nabla v|_6+|g\nabla v|_\infty|g\nabla^2v|_6),\\
B_4\leq& |\nabla(w^\nu n)|_\infty|\nabla^2(g^2H(v))|_2\\
\leq& C(|w^{\nu-1}|_\infty|\nabla w|_\infty|n|_\infty+|w^\nu|_\infty|\nabla n|_\infty)\cdot(|g\nabla v|_\infty|\nabla^2g|_3|\nabla v|_6\\
&+|\nabla g|^2_\infty|\nabla v|^2_4+|g\nabla v|_\infty|g\nabla^3v|_2+|g\nabla^2v|_4^2+|g\nabla v|_\infty|\nabla g|_\infty|\nabla^2v|_2),\\
B_5\leq&  C(|\nabla^3v|_2+|\nabla^2v|_3+|\nabla v|_\infty)\|\nabla l\|_2.
\end{split}
\end{equation}

$\ef{2.16}$ implies that
\begin{equation}\label{wguji}
\begin{split}
|w^{\nu}|_\infty\leq Cc^\nu_1,\quad |w^{\nu-1}|_\infty\leq Cc^{\nu+1}_1,\quad |w^{\nu-2}|_\infty\leq & Cc^{\nu+2}_1,\\
|w^{\nu-3}|_\infty\leq Cc^{\nu+3}_1,\quad |\nabla w|_\infty+\|\nabla w\|_2\leq & Cc_1,
\end{split}
\end{equation}
which, along with \eqref{parl}-\eqref{b1},   $\ef{2.16}$ and \eqref{2.34}, yields that 
\begin{equation*}
\begin{split}
\frac{d}{dt}\|\nabla l  \|_2
\leq & C\| v\|_3\|\nabla l  \|_2+M(c_0)(c_4^{\nu+10}+c_4^{\nu+2}|g\nabla^4v|_2).
\end{split}
\end{equation*}
It then follows from  Gronwall's inequality and $\ef{2.16}$ again that 
\begin{equation}\label{2.42h}
\|\nabla l\|_2\leq (\|\nabla l_0\|_2+M(c_0)c_4^{\nu+10}(t+t^{\frac{1}{2}}))\exp(Cc_4t)\leq M(c_0)\ \  \text{for} \ 0\leq t\leq T_2.
\end{equation}

Second, according to  $\ef{ln}_3$, for  $0\leq t\leq T_2$, it holds that
\begin{equation}\label{2.42j}
\begin{split}
& |l_t|_\infty\leq C(|v|_\infty|\nabla l|_\infty+|w^\nu|_\infty|n|_\infty|g\nabla v|_\infty^2)\leq M(c_0)c_3^{3+\nu}c_4, \\
& |l_t|_3\leq C(|v|_\infty|\nabla l|_3+|w^\nu|_\infty|n|_\infty|g\nabla v|^2_6)\leq M(c_0)c_3^{4+\nu}.
\end{split}
\end{equation}
It follows from  \eqref{nablal} and $\ef{ln}_3$ that for $|\varsigma|=1$,
\begin{equation}\label{2.42k}
|\partial_x^\varsigma l_t|_2 \leq  C(|v|_\infty|\nabla^2 l|_2+|\varLambda_\varsigma|_2)\leq M(c_0)c_3^{6+\nu}c_4^{\fr{1}{2}}.
\end{equation}

 Similarly,  for $|\varsigma|=2$, from \eqref{var2}-\eqref{a1}, one can obtain
\begin{equation}\label{2.42l}
|\partial_x^\varsigma l_t|_2 \leq C(
|v|_\infty|\nabla^3 l|_2+
|\varLambda_\varsigma|_2)\leq M(c_0)c_3^{8+\nu}c_4^{\fr{3}{2}}.
\end{equation}



On the other hand, since
\begin{equation}\label{ltt}
l_{tt}=-(v\cdot \nabla l)_t+a_4(w^\nu ng^2H(v))_t,
\end{equation}
one gets
\begin{equation*}\label{}
\begin{split}
|l_{tt}|_6 \leq& C(|v_t|_6|\nabla l|_\infty+|v|_\infty|\nabla l_t|_6+|w^\nu|_\infty|g\nabla v |_\infty(|n_t|_\infty|g\nabla v |_6\\
&+|n|_\infty|g_t|_\infty|\nabla v|_6+ |n|_\infty|g\nabla v_t|_6)+|w^{\nu-1}|_\infty|n|_\infty|w_t|_6|g\nabla v|^2_\infty),
\end{split}
\end{equation*}
which implies that 
\begin{equation}\label{2.42n}
\int^t_0|l_{tt}|^2_6\text{d}s\leq M(c_0)c_4^{6+2\nu} \quad \text{for} \quad  0\leq t\leq T_2.
\end{equation}

\noindent Due to \eqref{ltt}, one has that for $|\varsigma|=1$, 
\begin{equation}\label{naltt}
\partial_x^\varsigma l_{tt}=-\partial_x^\varsigma \big((v\cdot \nabla l)_t\big)+\partial_x^\varsigma (a_4w^\nu ng^2H(v))_t.
\end{equation}

It follows  from Lemma \ref{varphi}, \eqref{2.42h}-\eqref{2.42j}, \eqref{2.42k}-\eqref{2.42l} and \eqref{2.16} that
\begin{equation*}
\begin{split}
|\partial_x^\varsigma \big((v\cdot \nabla l)_t\big)|_2 \leq & C(|v_t|_6|\nabla^2l|_3+|\nabla v_t|_2|\nabla l|_\infty+|\nabla v|_\infty|\nabla l_t|_2+|v|_\infty|\nabla^2 l_t|_2)\\
\leq & M(c_0)|\nabla v_t|_2+ M(c_0)c_4^{11+\nu},\\
|\partial_x^\varsigma (w^\nu ng^2H(v))_t|_2 \leq & C(|\partial_x^\varsigma ((w^\nu)_t ng^2H(v))|_2+|\partial_x^\varsigma (w^\nu n_tg^2H(v))|_2\\
&+|\partial_x^\varsigma (w^\nu ngg_tH(v))|_2+|\partial_x^\varsigma (w^\nu ng^2H(v)_t)|_2)
=C\sum^4_{i=1}K_i,
\end{split}
\end{equation*}
and each $K_i$ $(i=1,...,4)$ can be estimated as follows:
\begin{equation*}
\begin{split}
K_1\leq& C \big(|n|_\infty|g\nabla v|_\infty^2(|w^{\nu-1}|_\infty|\nabla w_t|_2+|w^{\nu-2}|_\infty|w_t|_3|\nabla w|_6)\\
&+|w^{\nu-1}|_\infty|w_t|_3|\nabla n|_6|g\nabla v|_\infty^2+|w^{\nu-1}|_\infty|n|_\infty|\nabla g|_\infty|w_t|_3|g\nabla v|_\infty|\nabla v|_6\\
&+|w^{\nu-1}|_\infty|n|_\infty|w_t|_3|g\nabla v|_\infty|g\nabla^2v|_6\big),\\
K_2\leq& C\big(|w^{\nu-1}|_\infty|n_t|_\infty|g\nabla v|_\infty^2|\nabla w|_2+|w^{\nu}|_\infty|g\nabla v|_\infty|g\nabla v|_6|\nabla n_t|_3\\
&+|w^{\nu}|_\infty|n_t|_\infty|\nabla g|_\infty|g\nabla v|_\infty|\nabla v|_2+|w^{\nu}|_\infty|n_t|_\infty|g\nabla v|_\infty|g\nabla^2 v|_2\big),\\
K_3\leq& C\big(|w^{\nu-1}|_\infty|n|_\infty|g\nabla v|_\infty|g_t|_\infty|\nabla w|_3|\nabla v|_6+|w^{\nu}|_\infty|g\nabla v|_\infty|g_t|_\infty|\nabla n|_6|\nabla v|_3\\
&+|w^{\nu}|_\infty|n|_\infty(|g_t|_\infty|\nabla g|_\infty|\nabla v|_4^2+|g\nabla v|_\infty|\nabla v|_\infty|\nabla g_t|_2+|g\nabla v|_\infty|g_t|_\infty|\nabla^2v|_2)\big),\\
K_4\leq& C\big(|w^{\nu-1}|_\infty|n|_\infty|g\nabla v|_\infty|g\nabla   v_t|_6|\nabla w|_3+|w^{\nu}|_\infty|g\nabla v|_6|g\nabla v_t|_6|\nabla n|_6\\
&+|w^{\nu}|_\infty|n|_\infty(|g\nabla v|_\infty|\nabla g|_\infty|\nabla v_t|_2+|g\nabla v_t|_6|g\nabla^2 v|_3+|g\nabla v|_\infty|g\nabla^2v_t|_2)\big),
\end{split}
\end{equation*}
which yields that 
\begin{equation*}
|\partial_x^\varsigma (w^\nu ng^2H(v))_t|_2 \leq M(c_0)c_4^{8+\nu}+M(c_0)c_4^{4+\nu}|g\nabla^2v_t|_2+M(c_0)c_4^{5+\nu}|\nabla v_t|_2.
\end{equation*}
These and  \eqref{naltt} yield
\begin{equation}\label{2.42n}
\int^t_0|\nabla l_{tt}|^2_2\text{d}s\leq M(c_0)c_4^{10+2\nu} \quad \text{for} \quad 0\leq t\leq T_2.
\end{equation}

\noindent\textbf{Step 3:} The  estimate on  $|h^{\fr{1}{4}}\nabla l|_6$.
Applying $h^{\fr{1}{4}}\nabla$ to $\ef{ln}_3$ yields

\begin{equation}\label{nlt1}
\begin{split}
(h^{\fr{1}{4}}\nabla l)_t-(h^{\fr{1}{4}})_t\nabla l+h^{\fr{1}{4}}\nabla(v\cdot\nabla l)=a_4h^{\fr{1}{4}}\nabla(w^\nu ng^2H(v)).
\end{split}
\end{equation}
Denoting $h^{\fr{1}{4}}\nabla l=z$, multiplying $6|z|^4z$ on both side of $\ef{nlt1}$, integrating over $\mathbb{R}^3$, and integration by part, one has
\begin{equation}\label{hnl*}
\begin{split}
\fr{d}{dt}|z|_6^6\leq& C\int\big|\big((h^{\fr{1}{4}})_t\nabla l - h^{\fr{1}{4}}\nabla(v\cdot\nabla l)+a_4h^{\fr{1}{4}}\nabla(w^\nu ng^2H(v)) \big)\cdot |z|^4z\big|\\
\leq & C(|h_t|_\infty|\varphi|_\infty+ |\nabla v|_\infty+|v|_\infty|\psi|_\infty|\varphi|_\infty) |z|^6_6+J_*,
\end{split}
\end{equation}
where $J_*=C\int |h^{\fr{1}{4}}\nabla(w^\nu ng^2H(v)) \cdot |z|^4z|$.

Note  that $n=(ah)^b$ and $\fr{1}{4}+b\leq 0$ due to \ef{can1}.
One can get
\begin{equation}\label{J3}
\begin{split}
J_*\leq& C|\varphi|_\infty^{-\fr{1}{4}-b}(|w^{\nu-1}|_\infty|\nabla w|_6|g\nabla v|^2_\infty+|w^\nu|_\infty|g\nabla v|_\infty|\nabla g|_6|\nabla v|_\infty\\
&+|w^\nu|_\infty|g\nabla v|_\infty|g\nabla^2v|_6+|\varphi|_\infty|\psi|_\infty|w^\nu|_\infty|g\nabla v|_\infty|g\nabla v|_6)|z|^5_6.
\end{split}
\end{equation}
Note that \ef{2.14} implies 
\begin{equation}\label{nbl0}
\begin{split}
|h_0^{\fr{1}{4}}\nabla l_0|_6&\leq C|\nabla(\phi_0^{\fr{\iota}{2}}\nabla l_0)|_2\leq C(|\phi_0^{-\fr{\iota}{2}}|_\infty|\phi_0^\iota\nabla^2 l_0|_2+|\nabla\phi_0^{\fr{\iota}{2}}\cdot\nabla l_0|_2)\leq M(c_0).
\end{split}
\end{equation}
It follows from \ef{hnl*}-\ef{nbl0} and Gronwall's inequality that
\begin{equation}\label{hnl}
|h^{\fr{1}{4}}\nabla l|_6\leq M(c_0)\quad \text{for} \quad 0\leq t\leq T_2.
\end{equation}

The proof of Lemma \ref{l} is complete.
\end{proof}
\subsubsection{The equivalence of  $g$ and $h$ in a short time}
\begin{lemma}\label{gh}
It holds that
\begin{equation}\label{g/h}
\tilde{C}^{-1}\leq gh^{-1}\leq \tilde{C}
\end{equation}
for $0\leq t\leq T_2$, where $\tilde{C}$ is a suitable constant.
\end{lemma}
\begin{proof}
Set $gh^{-1}=y(t,x)$. Then a simple computation shows 
\begin{equation}\label{ghe}
y_t+yh^{-1}h_t=g_t\varphi; \quad y(0,x)=1.
\end{equation}
Thus 
\begin{equation}\label{y}
y(t,x)=\exp\big(-\int^t_0h_s h^{-1}\text{d}s\big)\big(1+\int^t_0g_s\varphi\exp\big(\int^s_0 h_\tau h^{-1}\text{d}\tau\big)\text{d}s\big),
\end{equation}
which, along with  Lemmas \ref{psi}-\ref{varphi} and \ef{2.16}, yields \eqref{g/h}.

The proof of Lemma \ref{gh} is complete.
\end{proof} 
\subsubsection{The a priori estimates for $u$}
Based on the estimates of $\phi$, $h$ and $l$ obtained above, now we are ready to derive the lower order energy estimates for $u$.
\begin{lemma}\label{lu}
\rm {For} $t\in [0, T_2]$, it holds that 
	\begin{equation}\label{2.61}\begin{aligned}
	|\sqrt{h}\nabla u(t)|^2_2+\|u(t)\|_1^2+\int^t_0(\|\nabla u\|_1^2+|u_t|^2_2)\text{d}s\leq & M(c_0),\\
	(|u|^2_{D^2}+|h\nabla^2u|^2_2+|u_t|^2_2)(t)+\int^t_0(|u|_{D^3}^2+|u_t|^2_{D^1_*})\text{d}s\leq & M(c_0)c_2^{3}c_3.
	\end{aligned}\end{equation}
\end{lemma}
\begin{proof}
    \textbf{Step 1:} Estimate on $|u|_2$.
It follows from  $\ef{ln}_2$ that 
\begin{equation}\label{2.61a}
l^{-\nu} (u_t+v\cdot \nabla v+a_1\phi\nabla l+l\nabla\phi)+a_2\sqrt {h^2+\epsilon^2}Lu=a_2gl^{-\nu}\nabla l^\nu\cdot Q(v)+a_3\psi\cdot Q(v).
\end{equation}
Multiplying $\ef{2.61a}$ by $u$ and integrating over $\mathbb{R}^3$, one can obtain by integration by parts, \text{Gagliardo-Nirenberg} inequality, \text{H\"older's} inequality and \text{Young's} inequality that
\begin{equation}\label{2.61b}
\begin{split}
&\fr{1}{2}\fr{d}{dt}|l^{-\fr{\nu}{2}}u|_2^2+a_2\alpha|(h^2+\epsilon^2)^{\fr{1}{4}}\nabla u|^2_2+a_2(\alpha+\beta)|(h^2+\epsilon^2)^{\fr{1}{4}}\text{div} u|^2_2 \\
 =&-\int l^{-{\nu}}(v\cdot\nabla v+a_1\phi\nabla l+l\nabla\phi
-a_2\nabla l^\nu \cdot gQ(v)-a_3l^{{\nu}}\psi\cdot Q(v))\cdot u\\
&+\fr{1}{2}\int(l^{-\nu})_t|u|^2-a_2\int\nabla\sqrt{h^2+\epsilon^2}\cdot Q(u)\cdot u\\
\leq& C\big(|l^{-\fr{\nu}{2}}|_\infty(|v|_\infty|\nabla v|_2+|\nabla l|_2|\phi|_\infty+|l|_\infty|\nabla\phi|_2)
+|l^{\frac{\nu}{2}-1}|_\infty|g\nabla v|_\infty|\nabla l|_2\\
&+|l^{\fr{3}{2}\nu}|_\infty|\psi|_\infty|\nabla v|_2\big)|l^{-\fr{\nu}{2}}u|_2+C|l^{-1}|_\infty|l_t|_\infty| l^{-\frac{\nu}{2}}u|^2_2\\
&+C|\psi|_\infty|l^{\fr{\nu}{2}}|_\infty|\varphi|^{\fr{1}{2}}_\infty|\sqrt{h}\nabla u|_2|l^{-\fr{\nu}{2}}u|_2\\
\leq& M(c_0)c_4^{4+\nu}|l^{-\fr{\nu}{2}}u|^2_2+M(c_0)c_4^4+\fr{1}{2}a_2\alpha|\sqrt{h}\nabla u|^2_2,\\
\end{split}
\end{equation}
which, along with $\text{Gronwall's} $ inequalty and Lemma \ref{l}, yields  that for $0\leq t\leq T_2$,
\begin{equation}\label{ul2}
\begin{split}
&|u|_2^2+|l^{-\fr{\nu}{2}}u|_2^2+\int^t_0|\sqrt{h}\nabla u|^2_2\text{d}s\\
\leq &  M(c_0) (|u_0|_2^2+c_4^4t)\exp{(M(c_0)c_4^{4+\nu}t)}\leq M(c_0).
\end{split}
\end{equation}
\textbf{Step 2:} Estimate on $|\nabla u|_2$. Multiplying $\ef{2.61a}$ by $u_t$ , integrating over $\mathbb{R}^3$ and integration by parts, one gets by  the \text{Gagliardo-Nirenberg} inequality, \text{H\"older's}
inequality, $\text{Young's}$ inequality and Lemmas \ref{phiphi}-\ref{l} that
\begin{equation*}
\begin{split}
& \fr{1}{2}\fr{d}{dt}(a_2\alpha|(h^2+\epsilon^2)^{\fr{1}{4}}\nabla u|^2_2+a_2(\alpha+\beta)|(h^2+\epsilon^2)^{\fr{1}{4}}\text{div} u|^2_2)+|l^{-\fr{\nu}{2}}u_t|^2_2 \\
 =&-\int l^{-\nu}\big(v\cdot\nabla v+a_1\phi\nabla l+l\nabla\phi
-a_2g\nabla l^\nu\cdot Q(v)-a_3 l^{\nu}\psi\cdot Q(v)\big)\cdot u_t\\
&+\fr{1}{2}\int a_2\fr{h}{\sqrt{h^2+\epsilon^2}}h_t(\alpha|\nabla u|^2+(\alpha+\beta)|\text{div}u|^2)\\
&-\int a_2\nabla\sqrt{h^2+\epsilon^2}\cdot Q(u)\cdot u_t\\
\leq& C|l^{-\fr{\nu}{2}}|_\infty(|v|_\infty|\nabla v|_2+|\nabla l|_2|\phi|_\infty+|l|_\infty|\nabla\phi|_2+|g\nabla v|_\infty|l^{\nu-1}|_\infty|\nabla l|_2\\
&+|\psi|_\infty|l^\nu|_\infty|\nabla v|_2)|l^{-\fr{\nu}{2}}u_t|_2+C|h_t|_\infty|\varphi|_\infty|\sqrt{h}\nabla u|^2_2\\
&+C|l^{\fr{\nu}{2}}|_\infty|\psi|_\infty|l^{-\fr{\nu}{2}}u_t|_2|\varphi|^{\fr{1}{2}}_\infty|\sqrt{h}\nabla u|_2\\
\leq& M(c_0)c_4^4|\sqrt{h}\nabla u|_2^2+M(c_0)c_4^4+\fr{1}{2}|l^{-\fr{\nu}{2}}u_t|^2_2,
\end{split}
\end{equation*}
which, along with \text{Gronwall's} inequality and $\ef{2.14}$,  implies that for $0\leq t\leq T_2$,
\begin{equation}\label{2.61f}
\begin{split}
&|\sqrt{h}\nabla u|^2_2+|\nabla u|^2_2+\int^t_0\big(|l^{-\fr{\nu}{2}}u_t|^2_2+|u_t|^2_2\big)\text{d}s\\
\leq &  M(c_0)(1+c_4^4t)\exp{( M(c_0)c_4^4t)}\leq  M(c_0).
\end{split}
\end{equation}

By the definitions of the \text{Lam$\acute{\text{e}}$} operator $L$ and $\psi$, it holds that
\begin{equation}\label{2.61h}
\begin{split}
&a_2L(\sqrt{h^2+\epsilon^2}u) =a_2\sqrt{h^2+\epsilon^2}Lu-a_2G(\nabla\sqrt{h^2+\epsilon^2},u) \\  
=&l^{-\nu}\mathcal{H}-a_2G(\nabla\sqrt{h^2+\epsilon^2},u),
\end{split}
\end{equation}
where 
\begin{equation}\label{Hdingyi}
\mathcal{H}=-u_t-v\cdot\nabla v-l\nabla\phi-a_1\phi\nabla l+a_2g\nabla l^\nu \cdot Q(v)+a_3 l^\nu \psi\cdot Q(v).
\end{equation}
Next, for  giving  the $L^2$ estimate of $\nabla^2 u$, we  consider the $L^2$ estimates of $$(\mathcal{H},\tilde{G}=G(\nabla\sqrt{h^2+\epsilon^2},u)).$$  It follows from \eqref{Gdingyi}, \eqref{2.16}, \eqref{ul2}-\eqref{2.61f}, \eqref{Hdingyi} and  Lemmas \ref{phiphi}-\ref{l} that 
\begin{equation}\label{2.61illl}
\begin{split}
|\mathcal{H}|_2\leq & C(|u_t|_2+|v|_6|\nabla v|_3+|l|_\infty|\nabla \phi|_2+|\phi|_\infty|\nabla l|_2\\
&+|g\nabla l^\nu \cdot Q(v)|_2+|l^\nu |_\infty|\psi|_\infty|\nabla v|_2)\\
\leq & M(c_0)(|u_t|_2+c_2^{\fr{3}{2}}c_3^{\fr{1}{2}}),\\
|\tilde{G}|_2
\leq & |\nabla\sqrt{h^2+\epsilon^2}|_\infty |\nabla u|_2+|\nabla^2 \sqrt{h^2+\epsilon^2}|_3 |u|_6)\leq M(c_0),
\end{split}
\end{equation}
where one has used the facts that 
\begin{equation}\label{hl2}
\begin{split}
|\nabla v|_3\leq & |\nabla v|^{\fr{1}{2}}_2|\nabla v|^{\fr{1}{2}}_6 \leq  C|\nabla v|^{\fr{1}{2}}_2|\nabla^2 v|^{\fr{1}{2}}_2\leq Cc_2^{\fr{3}{2}}c_3^{\fr{1}{2}},\\
|\nabla l^\nu \cdot  gQ(v)|_2=&|(\nabla l^\nu \cdot  gQ(v))(0,x)+\int^t_0(\nabla l^\nu \cdot  gQ(v))_s\text{d}s|_2\\
\leq& |\nabla l_0^\nu\cdot h_0Q(u_0)|_2+t^{\fr{1}{2}}\Big(\int^t_0|(\nabla l^\nu \cdot gQ(v))_s|_2^2\text{d}s\Big)^{\fr{1}{2}}\\
\leq& |l_0^{\nu-1}|_\infty|h_0^{\fr{1}{4}}\nabla l_0|_6|h_0^{\fr{3}{4}}\nabla u_0|_3\\
&+Ct^{\fr{1}{2}}\Big(\int^t_0\big(|g_t|^2_\infty|\nabla l^\nu|^2_\infty|\nabla v|^2_2+|\nabla l^\nu|^2_3|g\nabla v_t|^2_6\big)\text{d}s\Big)^{\fr{1}{2}}\\
&+Ct^{\fr{1}{2}}\Big(\int_0^t |g\nabla v|^2_\infty\big(|\nabla l_t|^2_2|l^{\nu-1}|^2_\infty+|(l^{\nu-1})_t|^2_\infty|\nabla l|_2^2\big)\text{d}s\Big)^{\fr{1}{2}}\\
\leq &|l_0^{\nu-1}|_\infty|h_0^{\fr{1}{4}}\nabla l_0|_6|\sqrt{h_0}\nabla u_0|^{\fr{1}{2}}_2|h_0\nabla u_0|^{\fr{1}{2}}_6\\
&+M(c_0)c_4^{8.5+\nu}(t+t^{\frac{1}{2}})
\leq M(c_0),\\
|\nabla\sqrt{h^2+\epsilon^2}|_\infty \leq &  C|\nabla h|_\infty\leq C|\psi|_\infty\leq Cc_0,\\
|\nabla^2\sqrt{h^2+\epsilon^2}|_3 \leq &  C(|\nabla \sqrt{h}|^2_6+|\nabla^2 h |_3)\leq   C(|\nabla \sqrt{h}|^2_6+|\nabla \psi|_3)\leq Cc^2_0.
\end{split}
\end{equation}

Then it follows from \eqref{ul2}-\eqref{2.61illl},  the classical theory for elliptic equations and  Lemmas \ref{psi}-\ref{l}  that 
\begin{equation}\label{2.61i}
\begin{split}
|\sqrt{h^2+\epsilon^2}u|_{D^2}\leq& C(|l^{-\nu}\mathcal{H}|_2+|G(\nabla\sqrt{h^2+\epsilon^2},u)|_2)\\
\leq& C(|l^{-\nu}|_\infty|\mathcal{H}|_2+|G(\nabla\sqrt{h^2+\epsilon^2},u)|_2)\\
\leq & M(c_0)(|u_t|_2+c_2^{\fr{3}{2}}c_3^{\fr{1}{2}}),\\
|\sqrt{h^2+\epsilon^2}\nabla^2u|_2\leq& C( |\sqrt{h^2+\epsilon^2}u|_{D^2} +|\nabla\psi|_3|u|_6+|\psi|_\infty|\nabla u|_2  \\
&+ M(c_0)|\psi|^2_\infty|u|_2|\varphi|_\infty) 
\leq  C |\sqrt{h^2+\epsilon^2}u|_{D^2} +M(c_0).
\end{split}
\end{equation}

Finally, it follows from $\ef{2.61f}$, $\ef{2.61i}$ and Lemma \ref{varphi}  that 
\begin{equation*}
\int^t_0(|h\nabla^2u|^2_2+|\nabla^2u|^2_2)\text{d}s\leq M(c_0) \quad \text{for}\quad 0\leq t\leq T_2.
\end{equation*}

\textbf{Step 3}: Estimate on $|u|_{D^2}$. Applying $\partial_t$ to $\ef{ln}_2$ yields
\begin{equation}\label{2.61j}\begin{split}
 u_{tt}+a_2l^{\nu}\sqrt{h^2+\epsilon^2}Lu_t=&-(v\cdot \nabla v)_t-(l\nabla\phi)_t-a_1(\phi\nabla l)_t\\
&-a_2(l^\nu)_t\sqrt{h^2+\epsilon^2}Lu -a_2\fr{h}{\sqrt{h^2+\epsilon^2}}l^\nu h_tLu\\
&+(a_2g\nabla l^\nu\cdot Q(v)+a_3l^\nu\psi \cdot Q(v))_t.
\end{split}
\end{equation}
Multiplying $\ef{2.61j}$ by $l^{-\nu}u_t$, integrating over $\mathbb{R}^3$ and integration by part lead to 
\begin{equation}\label{2.61l}\begin{split}
&  \fr{1}{2}\fr{d}{dt}|l^{-\fr{\nu}{2}}u_{t}|^2_2+a_2\alpha|(h^2+\epsilon^2)^{\fr{1}{4}}\nabla u_t|_2^2+a_2(\alpha+\beta)|(h^2+\epsilon^2)^{\fr{1}{4}}\text{div} u_t|_2^2\\
=&\int l^{-\nu}\Big(-(v\cdot \nabla v)_t-(l \nabla \phi)_t-a_1( \phi\nabla l)_t-a_2(l^\nu)_t\sqrt{h^2+\epsilon^2}Lu\\
&+(a_2g\nabla l^\nu\cdot Q(v)+a_3l^\nu\psi\cdot Q(v))_t\Big)\cdot u_t-\int a_2\fr{h}{\sqrt{h^2+\epsilon^2}} h_tLu\cdot u_t\\
&-\int a_2\nabla\sqrt{h^2+\epsilon^2}\cdot Q(u_t)\cdot u_t+\frac{1}{2}\int (l^{-\nu})_t|u_t|^2\\
\leq& C\Big(|l^{-\frac{\nu}{2}}|_\infty(|v|_\infty|\nabla v_t|_2+|v_t|_2|\nabla v|_\infty+|l_t|_\infty|\nabla\phi|_2+|\nabla l_t|_2|\phi|_\infty\\
&+|\phi_t|_\infty|\nabla l|_2+|l|_\infty|\nabla\phi_t|_2)
+|l^{\frac{\nu}{2}-1}|_\infty|l_t|_\infty|\sqrt{h^2+\epsilon^2}\nabla^2u|_2
\\
&+|l^{\frac{\nu}{2}-2}|_\infty|g\nabla v|_\infty|l_t|_\infty|\nabla l|_2+|l^{\frac{\nu}{2}-1}|_\infty(|g_t|_\infty|\nabla v|_\infty|\nabla l|_2+|g\nabla v|_\infty|\nabla l_t|_2)\\
&+|l^{\frac{\nu}{2}}|_\infty(|\psi_t|_2|\nabla v|_\infty+|\psi|_\infty
|\nabla v_t|_2)+|l^{\frac{\nu}{2}-1}|_\infty|\psi|_\infty|l_t|_\infty|\nabla v|_2\Big)|l^{-\frac{\nu}{2}}u_t|_2\\
&+C|l^{-1}|_\infty|\sqrt{g}\nabla v_t|_2|\nabla l|_3(|\sqrt{h}\nabla u_t|_2+|l^{\frac{\nu}{2}}|_\infty|\varphi|^{\frac{1}{2}}_\infty|\psi|_\infty|l^{-{\frac{\nu}{2}}}u_t|_2)\\
&+C|l^{-1}|_\infty|l_t|_\infty|l^{-\frac{\nu}{2}}u_t|^2_2\\
&+C|l^{\frac{\nu}{2}}|_\infty(|h_t|_\infty|\nabla^2u|_2+|\varphi|^{\fr{1}{2}}_\infty|\psi|_\infty|\sqrt{h}\nabla u_t|_2)|l^{-\frac{\nu}{2}}u_t|_2,
\end{split}
\end{equation}
where one has used 
\begin{equation}\label{2.61k1}\begin{split}
&\int |l^{-\nu}g\nabla l^\nu\cdot Q(v_t)\cdot u_t|= \nu\int \Big|l^{-1}\frac{\sqrt{g}}{\sqrt{h}}\nabla l\cdot\sqrt{g}Q(v_t) \cdot \sqrt{h}u_t\Big|\\
\leq & C|l^{-1}|_\infty|\sqrt{g}\nabla v_t|_2|\nabla l|_3|\sqrt{h}u_t|_6 \\
\leq & C|l^{-1}|_\infty|\sqrt{g}\nabla v_t|_2|\nabla l|_3(|\sqrt{h}\nabla u_t|_2+|l^{\frac{\nu}{2}}|_\infty|\varphi|^{\frac{1}{2}}_\infty|\psi|_\infty|l^{-{\frac{\nu}{2}}}u_t|_2).
\end{split}
\end{equation}
Integrating $\ef{2.61l}$ over $(\tau,t)$ $(\tau\in(0,t))$,  one can get by using \ef{2.16}, Lemmas \ref{phiphi}-\ref{l}  and \text{Young's} inequality that
\begin{equation}\label{2.61n}
\begin{split}
&\fr{1}{2}|l^{-\fr{\nu}{2}}u_t(t)|^2_2+\fr{a_2\alpha}{2}\int^t_\tau|\sqrt{h}\nabla u_t|^2_2\text{d}s   \\
 \leq & \fr{1}{2}|l^{-\fr{\nu}{2}}u_t(\tau)|^2_2+M(c_0)c_4^{8+2\nu}\int^t_0|l^{-\fr{\nu}{2}}u_t|^2\text{d}s+M(c_0)c_4^9t+M(c_0).
\end{split}
\end{equation}
Due to $\ef{ln}_2$, it holds that
\begin{equation}\label{2.61o}
\begin{split}
|u_t(\tau)|_2\leq& C(|v|_\infty|\nabla v|_2+|\phi|_\infty|\nabla l|_2+|\nabla\phi|_2|l|_\infty+|l|^\nu_\infty|(h+\epsilon)Lu|_2
\\
& +|l^{\nu-1}|_\infty|g\nabla v|_\infty|\nabla l|_2+|\psi|_\infty|l^\nu|_\infty|\nabla v|_2)(\tau).
\end{split}
\end{equation}
 It follows from this, $\ef{4.1*}$, $\ef{2.14}$, $\ef{incc}$ and Lemma \ref{ls} that
\begin{equation*}\label{2.611p}\begin{split}
\lim\sup_{\tau\rightarrow 0}|u_t(\tau)|_2\leq& C(|u_0|_\infty|\nabla u_0|_2+|\phi_0|_\infty|\nabla l_0|_2+|\nabla\phi_0|_2|l
_0|_\infty+|\psi_0|_\infty|l_0^\nu|_\infty|\nabla u_0|_2 \\
&+|l_0^\nu|_\infty(|g_2|_2+|Lu_0|_2)+|l_0^{\nu-1}|_\infty|\phi_0^{2\iota}\nabla u_0|_\infty|\nabla l_0|_2)\leq M(c_0).
\end{split}
\end{equation*}
Letting $\tau\rightarrow 0$ in \ef{2.61n} and using \text{Gronwall's} inequality give that for $0\leq t\leq T_2$,
\begin{equation}\label{2.61q}
\begin{split}
&|u_t(t)|_2^2+\int^t_0\big(|\sqrt{h}\nabla u_t(s)|^2_2+|\nabla u_t(s)|^2_2\big)\text{d}s\\
\leq & (M(c_0)c_4^9t+M(c_0))\exp{(M(c_0)c_4^{8+2\nu}t)}\leq M(c_0),
\end{split}
\end{equation}
which, along with  $\ef{2.61i}$, yields that  for $0\leq t\leq T_2$,
\begin{equation}\label{2.61s}
|\sqrt{h^2+\epsilon^2}u(t)|_{D^2}+ |h\nabla^2 u(t)|_2+ |u(t)|_{D^2}\leq M(c_0)c_2^{\fr{3}{2}}c_3^{\fr{1}{2}}.
\end{equation}

Next, for  giving  the $L^2$ estimate of $\nabla^3 u$, we  consider the $L^2$ estimates of 
$$(\nabla \mathcal{H},\nabla \tilde{G}=\nabla G(\nabla\sqrt{h^2+\epsilon^2},u)).$$
It follows from \eqref{Gdingyi}, \eqref{2.16}, \eqref{ul2}-\eqref{2.61f}, \eqref{Hdingyi}, \eqref{hl2},  \eqref{2.61s}  and  Lemmas \ref{phiphi}-\ref{l} that
\begin{equation}\label{2.61t}\begin{split}
 |\mathcal{H}|_{D^1_*}
\leq & C(|u_t|_{D^1_*}+|v|_\infty|\nabla^2 v|_2+|\nabla v|_6|\nabla v|_3+|l|_\infty|\nabla^2 \phi|_2\\
&+|\nabla \phi|_2|\nabla l|_\infty+| \phi|_\infty|\nabla^2 l|_2+|\nabla g|_\infty |\nabla l^\nu|_\infty |\nabla v|_2\\
&+ |\nabla^2 l^\nu|_3|g\nabla v|_6+ |\nabla l^\nu|_\infty |g\nabla^2 v|_2\\
&+|\nabla l^\nu |_\infty|\psi|_\infty|\nabla v|_2+|l^\nu |_\infty|\nabla \psi|_3|\nabla v|_6\\
&+ |l^\nu |_\infty|\psi|_\infty|\nabla^2 v|_2)\leq M(c_0)(|u_t|_{D^1_*}+c_3^2),\\
 |\tilde{G}|_{D^1_*}
\leq &C( |\nabla\sqrt{h^2+\epsilon^2}|_\infty |\nabla^2 u|_2+|\nabla^2 \sqrt{h^2+\epsilon^2}|_3 |\nabla u|_6\\
&+|\nabla^3 \sqrt{h^2+\epsilon^2}|_2 | u|_\infty)\leq  M(c_0)c^2_3,
\end{split}
\end{equation}
where one has used the fact that 
\begin{equation}\label{2.61rrnnn}
\begin{split}
|\nabla^3\sqrt{h^2+\epsilon^2}|_2 \leq &  M(c_0)(|\nabla \sqrt{h}|^3_6+|\nabla^2 h |_3|\nabla \sqrt{h}|_6+|\nabla^3 h|_2)\\
\leq &   M(c_0)(|\nabla \sqrt{h}|^3_6+|\nabla \psi|_3|\nabla \sqrt{h}|_6+|\nabla^2 \psi|_2)\leq M(c_0).
\end{split}
\end{equation}

It follows from \eqref{ul2}-\eqref{2.61h}, \eqref{2.61illl},  \eqref{2.61s}-\eqref{2.61t},  the classical theory for elliptic equations and  Lemmas \ref{psi}-\ref{l} that
\begin{equation}\label{2.61tkkk}\begin{split}
|\sqrt{h^2+\epsilon^2}u(t)|_{D^3}\leq& C|l^{-\nu}\mathcal{H}|_{D^1_*}+ C|G(\nabla\sqrt{h^2+\epsilon^2},u)|_{D^1_*}\\
\leq & C(|l^{-\nu}|_\infty|\mathcal{H}|_{D^1_*}+|\nabla l^{-\nu}|_\infty|\mathcal{H}|_{2}\\ &+|G(\nabla\sqrt{h^2+\epsilon^2},u)|_{D^1_*})\\
\leq& M(c_0)(|u_t|_{D^1_*}+c_3^2),\\
|\sqrt{h^2+\epsilon^2}\nabla^3u(t)|_2\leq& C( |\sqrt{h^2+\epsilon^2}u(t)|_{D^3} +\|\psi\|_{L^\infty \cap D^{1,3}\cap D^2}\|u\|_2)\\
&+C(1+\|\psi\|_{L^\infty \cap D^{1,3}\cap D^2}^3\| u\|_1)(1+|\varphi|^2_\infty)\\
\leq& M(c_0)(|\sqrt{h^2+\epsilon^2}u(t)|_{D^3}+c_3^2).
\end{split}
\end{equation}

Finally, it follows from  $\ef{2.61q}$, \eqref{2.61tkkk} and Lemma \ref{varphi} that
\begin{equation}\label{2.61u}
\int^t_0(|h\nabla^3u|^2_2+|h\nabla^2u|_{D^1_*}^2+|u|^2_{D^3})\text{d}s\leq M  (c_0) \quad \text{for} \quad 0\leq t\leq T_2.
\end{equation}

The proof of Lemma \ref{lu} is complete.

\end{proof}
Next, we estimate the higher order derivatives of the velocity $u$.
\begin{lemma}\label{hu}
\rm {For} $t\in [0, T_2]$, it holds that 
	\begin{equation}\label{2.62}\begin{aligned}
	(|\sqrt{h}\nabla u_t|^2_2+|u_t|^2_{D^1_*}+|u|^2_{D^3}+|h\nabla^2 u|^2_{D^1_*})(t)
	\leq M(c_0)c_3^4,&\\
	\int^t_0(|u_{tt}|^2_2+|u_t|^2_{D^2})\text{d}s\leq M(c_0),&\\
	\int^t_0(|h\nabla^2u_t|^2_2+|u|^2_{D^4}+|h\nabla^2u|_{D^2}^2+|(h\nabla^2u)_t|^2_2)\text{d}s\leq M(c_0).&
	\end{aligned}\end{equation}

\end{lemma}
\begin{proof}
Multiplying \ef{2.61j} by $l^{-\nu}u_{tt}$ and integrating over $\mathbb{R}^3$ lead to 
\begin{equation}\label{2.62a}
\begin{split}
& \fr{1}{2}\fr{d}{dt}(a_2\alpha|(h^2+\epsilon^2)^{\fr{1}{4}}\nabla u_t|^2_2+a_2(\alpha+\beta)|(h^2+\epsilon^2)^{\fr{1}{4}}\text{div}u_t|^2_2)+|l^{-\fr{\nu}{2}}u_{tt}|_2^2
=\sum^4_{i=1}I_i,
\end{split}
\end{equation}
where  $I_i$, $i=1,2,3,4$,  are given and  estimated as follows:
\begin{equation}\label{i11}
\begin{split}
I_1=&\int l^{-\nu}\Big(-(v\cdot \nabla v)_t-(l\nabla \phi)_t-a_1(\phi\nabla l)_t\\
&-a_2(l^\nu)_t\sqrt{h^2+\epsilon^2}Lu-a_2l^\nu\fr{h}{\sqrt{h^2+\epsilon^2}} h_tLu\Big)\cdot u_{tt}\\
\leq& C|l^{-\fr{\nu}{2}}|_\infty\Big(|v|_\infty|\nabla v_t|_2+|v_t|_2|\nabla v|_\infty+|l_t|_\infty|\nabla\phi|_2+|\nabla l_t|_2|\phi|_\infty\\
&+|\phi_t|_\infty|\nabla l|_2+|l|_\infty|\nabla\phi_t|_2\\
&+|l^{\nu-1}|_\infty|l_t|_\infty|\sqrt{h^2+\epsilon^2}\nabla^2u|_2+|l^{\nu}|_\infty|h_t|_\infty|\nabla^2u|_2\Big)|l^{-\fr{\nu}{2}}u_{tt}|_2,\\
I_2=&\int l^{-\nu} \big(a_2g\nabla l^\nu\cdot Q(v)+a_3l^\nu\psi\cdot Q(v)\big)_t\cdot u_{tt}\\
\leq &C|l^{-\fr{\nu}{2}}|_\infty\big(|(\nabla l^\nu)_t|_2|g\nabla v|_\infty+|g_t|_\infty|\nabla l^\nu|_3|\nabla v|_6\\
&+|h^{\fr{1}{4}}\nabla l^\nu|_6|h^{-\fr{1}{4}}g^{\fr{1}{4}}|_\infty|g^{\fr{3}{4}}\nabla v_t|_3+|l^\nu|_\infty|\psi|_\infty|\nabla v_t|_2\\
&+|l^\nu|_\infty|\psi_t|_2|\nabla v|_\infty+|(l^\nu)_t|_\infty|\psi|_\infty|\nabla v|_2\big)|l^{-\fr{\nu}{2}}u_{tt}|_2,\\
I_3+I_4=&-\int a_2\nabla\sqrt{h^2+\epsilon^2}Q(u_t)\cdot u_{tt}\\
&
+\frac{1}{2}\int a_2\fr{h}{\sqrt{h^2+\epsilon^2}}h_t(\alpha|\nabla u_t|^2+(\alpha+\beta)|\text{div} u_t|^2)\\
\leq &C(|l^{\fr{\nu}{2}}|_\infty|\varphi|^{\fr{1}{2}}_\infty|\psi|_\infty|\sqrt{h}\nabla u_t|_2|l^{-\fr{\nu}{2}}u_{tt}|_2
+|h_t|_\infty|\sqrt{h}\nabla u_t|^2_2|\varphi|_\infty).
\end{split}
\end{equation}
\noindent Integrating $\ef{2.62a}$ over $(\tau,t)$ and combining $\ef{i11}$ yield that for $0\leq t\leq T_2$,
\begin{equation}\label{2.62b}
\begin{split}
&|\sqrt{h}\nabla u_t(t)|^2_2+\int^t_\tau|l^{-\fr{\nu}{2}}u_{tt}|^2_2\text{d}s\\
\leq& C|(h^2+\epsilon^2)^{\fr{1}{4}}\nabla u_t(\tau)|^2_2+M(c_0)c_4^4\int^t_0|\sqrt{h}\nabla u_t|^2_2\text{d}s +M(c_0)(c_4^{17+2\nu}t+1),
\end{split}
\end{equation}
where one has used the fact \eqref{g/h} and 
\begin{equation}\label{key}
\begin{split}
|g^{\fr{3}{4}}\nabla v_t|_3\leq C|\sqrt{g}\nabla v_t|_2^{\fr{1}{2}}|g\nabla v_t|_6^{\fr{1}{2}}.
\end{split}
\end{equation}

Due to  $\ef{ln}_2$, one gets
\begin{equation}\label{2.62c}
\begin{split}
|\sqrt{h}\nabla u_t(\tau)|_2  \leq& (|\sqrt{h}\nabla(v\cdot \nabla v+a_1\phi\nabla l+l\nabla\phi+a_2\sqrt {h^2+\epsilon^2}l^\nu Lu\\
& -a_2\nabla l^\nu\cdot gQ(v)-a_3 l^\nu\psi\cdot Q(v))|_2)(\tau). \\
\end{split}
\end{equation}
It follows from  \ef{4.1*}, $\ef{2.14}$, Lemma 3.1 and Remark 3.1 that 
\begin{equation}\label{2.62d}
\begin{split}
&\lim\sup_{\tau\rightarrow 0}|\sqrt{h}\nabla u_t(\tau)|_2\\
\leq & C(|\sqrt{h_0}\nabla(u_0\cdot\nabla u_0)|_2+|\sqrt{h_0}l_0\nabla^2\phi_0|_2+|\sqrt{h_0}\nabla^2l_0\phi_0|_2\\
&+|\sqrt{h_0}\nabla l_0\cdot \nabla\phi_0|_2 
 +|\sqrt{h_0}\nabla(\sqrt{h_0^2+\epsilon^2}l_0^\nu Lu_0)|_2\\
 &+|\sqrt{h_0}\nabla(\nabla l_0^\nu \cdot  h_0Q(u_0))|_2+|\sqrt{h_0}\nabla(l_0^\nu\psi_0 \cdot  Q(u_0))|_2)\\
 \leq& C(|\phi_0^\iota u_0|_6|\nabla^2u_0|_3+|\nabla u_0|_\infty|\phi_0^\iota\nabla u_0|_2+|l_0|_\infty|\phi_0^\iota\nabla^2\phi_0|_2\\
 &+|\nabla^2l_0\phi_0^{\iota+1}|_2
 +|\nabla l_0|_3|\phi_0^\iota\nabla \phi_0|_6+|l_0^\nu|_\infty(|\nabla\psi_0|_3|\phi_0^\iota\nabla u_0|_6\\
 &+|\psi_0|_\infty|\phi^\iota_0\nabla^2u_0|_2)+|l^{\nu-1}_0|_\infty|\psi_0|_\infty|\phi^\iota_0\nabla u_0|_6|\nabla l_0|_3\\
 &+|\sqrt{h_0}\nabla(\sqrt{h_0^2+\epsilon^2}l_0^\nu Lu_0)|_2+|\sqrt{h_0}\nabla(\nabla l_0^\nu \cdot h_0Q(u_0))|_2).
\end{split}
\end{equation}
\ef{2.14} and \ef{incc1} imply
\begin{equation}\label{2.62e}
\begin{split}
&|\phi_0^{\iota}\nabla^2l_0|_2=|g_4|_2\leq c_0,\quad 
|\phi_0^\iota\nabla^2\phi_0|_2\leq M(c_0).
\end{split}
\end{equation}
Note that
\begin{equation}\label{2.62ennn}
\begin{split}
&|\sqrt{h_0}\nabla(\sqrt{h_0^2+\epsilon^2}l_0^\nu Lu_0)|_2
\leq  |\sqrt{h_0}l_0^\nu\nabla(\sqrt{h_0^2+\epsilon^2} Lu_0)|_2\\
&\qquad \qquad +|\sqrt{h_0}\sqrt{h_0^2+\epsilon^2} Lu_0\otimes \nabla l_0^\nu|_2=|J_1|_2+|J_2|_2,
\end{split}
\end{equation}
which will be estimated as follows.
For $J_1$, it holds that
\begin{equation*}
  \begin{split}
 J_1&=\sqrt{h_0}l_0^\nu \Big(\sqrt{h_0^2+\epsilon^2}\nabla Lu_0+\fr{h_0}{\sqrt{h_0^2+\epsilon^2}}Lu_0\otimes\nabla h_0 \Big) \\
    &=l_0^\nu\Big(\fr{h_0}{\sqrt{h_0^2+\epsilon^2}}g_3+\epsilon^2\nabla Lu_0\fr{\sqrt{h_0}}{\sqrt{h_0^2+\epsilon^2}}\Big),
 \end{split}
\end{equation*}
which yields 
\begin{equation}\label{2.62i}
 |J_1|_2\leq C|l_0^\nu|_\infty(|g_3|_2+\epsilon^2|\varphi_0|^{\fr{1}{2}}_\infty|\nabla^3u_0|_2).
\end{equation}
For $J_2$, one gets from \ef{2.14} that 
\begin{equation}\label{2.62j}\begin{split}
|J_2|_2&\leq C(|h_0^{\fr{3}{2}}l_0^{\nu-1}Lu_0\otimes \nabla l_0|_2+\epsilon|h_0^{\fr{1}{2}}l_0^{\nu-1}Lu_0\otimes \nabla l_0|_2)\\
 &\leq C|l_0^{\nu-1}|_\infty|h_0^{\fr{3}{2}}Lu_0|_6|\nabla l_0|_3+C|g_2|_2|\phi_0^{-\iota}|_\infty|l_0^{\nu-1}|_\infty|\nabla l_0|_\infty\\
&\leq C|l_0^{\nu-1}|_\infty|\nabla (h_0^{\fr{3}{2}}Lu_0)|_2|\nabla l_0|_3+C|g_2|_2|\phi_0^{-\iota}|_\infty|l_0^{\nu-1}|_\infty|\nabla l_0|_\infty\leq M(c_0).
\end{split}
\end{equation}
On the other hand, it follows  from  \ef{2.14} and Remark 3.1 that 
\begin{equation*}
\begin{split}
&|\sqrt{h_0}\nabla(\nabla l_0^\nu \cdot  h_0Q(u_0))|_2\\
\leq &C(|\sqrt{h_0}\nabla^2l_0^\nu|_2| h_0
\nabla u_0|_\infty+|\sqrt{h_0}\nabla l_0^\nu|_6(|h_0\nabla^2 u_0|_3+|\psi_0|_\infty|\nabla u_0|_3))\\
\leq & C(|\sqrt{h_0}\nabla^2l_0^\nu|_2| h_0
\nabla u_0|_6^{\fr{1}{2}}|\nabla(h_0
\nabla u_0)|_6^{\fr{1}{2}}\\
&+(|\sqrt{h_0}\nabla^2 l_0^\nu|_2+|\varphi_0^{\fr{1}{2}}\psi_0|_\infty|\nabla l_0^\nu|_2)(|h_0\nabla^2 u_0|_3
+|\psi_0|_\infty|\nabla u_0|_3))\leq M(c_0).
\end{split}
\end{equation*}
Hence, combining all the estimates with  $\ef{2.62d}$ yields 
\begin{equation}\label{2.62k}
\lim\sup_{\tau\rightarrow 0}|\sqrt{h}\nabla u_t(\tau)|_2\leq M(c_0),
\end{equation}
which implies also 
\begin{equation*}
\lim\sup_{\tau\rightarrow 0}|\sqrt{\epsilon}\nabla u_t(\tau)|_2\leq \lim\sup_{\tau\rightarrow 0}\sqrt{\epsilon}|\varphi|_\infty^{\fr{1}{2}}|\sqrt{h}\nabla u_t(\tau)|_2\leq M(c_0).
\end{equation*}
Letting $\tau\rightarrow 0$, one gets  from \eqref{2.62b} and  Gronwall's inequality that for $0\leq t\leq T_2$,
\begin{equation}\label{2.62m}\begin{split}
&|\sqrt h\nabla u_t(t)|_2^2+|\nabla u_t(t)|^2_2+\int^t_0|u_{tt}(s)|^2_2\text{d}s\\
\leq &  M(c_0)(1+c_4^{17+2\nu}t)\exp( M(c_0)c_4^4t)\leq  M(c_0),
\end{split}
\end{equation}
which, along with   $\ef{2.61tkkk}$, yields  
\begin{equation}\label{2.63}
|\sqrt{h^2+\epsilon^2}u|_{D^3}+|\sqrt{h^2+\epsilon^2}\nabla^3u|_2+|h\nabla^2u|_{D^1}+|\nabla^3u|_2\leq M(c_0)c^2_3.
\end{equation}
Note that  $\ef{2.61j}$ gives
\begin{equation}\label{2.64}\begin{split}
a_2L(\sqrt{h^2+\epsilon^2}u_t)&=a_2\sqrt{h^2+\epsilon^2}Lu_t-a_2G(\nabla\sqrt{h^2+\epsilon^2},u_t)\\
&=l^{-\nu}\mathcal{G}-a_2G(\nabla\sqrt{h^2+\epsilon^2},u_t),
\end{split}
\end{equation}
with 
\begin{equation}\label{huagdingyi}\begin{split}
\mathcal{G}=&-u_{tt}-(v\cdot \nabla v)_t-(l\nabla \phi)_t-a_1(\phi\nabla l)_t-a_2(l^\nu)_t\sqrt{h^2+\epsilon^2}Lu\\
&-a_2\fr{h}{\sqrt{h^2+\epsilon^2}} h_t l^{\nu}Lu+(a_2g\nabla l^\nu\cdot Q(v)+a_3l^\nu \psi \cdot Q(v))_t.
\end{split}
\end{equation}
Next, for  giving  the $L^2$ estimates of $(\nabla^2 u_t,\nabla^4 u)$, we  consider the $L^2$ estimates of $$(\mathcal{G},\widehat{G}=G(\nabla\sqrt{h^2+\epsilon^2},u_t),\nabla^2 \mathcal{H}).$$
In fact, it follows from \eqref{Gdingyi}, \eqref{2.16}, \eqref{Hdingyi}, \eqref{hl2}, \eqref{2.62m}-\eqref{2.63}, \eqref{huagdingyi} and  Lemmas \ref{phiphi}-\ref{lu} that 

\begin{equation}\label{2.61ccvv}
\begin{split}
|\mathcal{G}|_2\leq & C(|u_{tt}|_2+\|v\|_2|\nabla v_t|_2+\|l\|_{L^\infty\cap D^1\cap D^2}\| \phi_t\|_1+\|\phi\|_2\|l_t\|_{L^3\cap D^1}\\
&+|(l^\nu)_t |_\infty|\sqrt{h^2+\epsilon^2} Lu|_2 +|l^\nu |_\infty |h_t|_\infty|\nabla^2u|_2+|g_t|_\infty|\nabla l^\nu|_2
|\nabla v|_\infty\\
&+|g\nabla v|_\infty|\nabla(l^\nu)_t|_2+|l^{\nu-1}|_\infty|g^{\fr{1}{4}}\nabla l|_6|g^{\fr{3}{4}}\nabla v_t|_3\\
&+|(l^\nu)_t|_3|\psi|_\infty|\nabla v|_6+|l^\nu|_\infty|\psi_t|_2|\nabla v|_\infty+|l^\nu|_\infty|\psi|_\infty|\nabla v_t|_2\\
\leq & M(c_0)(|u_{tt}|_2+c_4^{8.5+\nu}+c_4^{\fr{1}{2}}|g\nabla v_t|^{\fr{1}{2}}_{D^1_*}),\\
|\mathcal{H}|_{D^2}\leq & C(|u_t|_{D^2}+\|v\|_2\|\nabla v\|_2+\|l\|_{L^\infty\cap D^1\cap D^3}\|\nabla \phi\|_2\\
&+\|\nabla l^\nu\|_2 (\|g\nabla v\|_{L^\infty\cap D^1\cap D^2}+\|\nabla g\|_{L^\infty\cap D^2 }\|\nabla v\|_2)\\
& +\|l^\nu \|_{L^\infty\cap D^1\cap D^2}\|\psi\|_{L^q\cap D^{1,3}\cap D^2}\|\nabla v\|_2)\\
\leq & M(c_0)(|u_t|_{D^2}+c_4^2),\\
|\widehat{G}|_2
\leq & C(|\nabla\sqrt{h^2+\epsilon^2}|_\infty |\nabla u_t|_2+|\nabla^2 \sqrt{h^2+\epsilon^2}|_3 |u_t|_6)\leq M(c_0),
\end{split}
\end{equation}
where one has used  \eqref{g/h}  and 
\begin{equation*}
\begin{split}
|g^{\fr{3}{4}}\nabla v_t|_3
\leq  C|\sqrt{g}\nabla v_t|_2^{\fr{1}{2}}|g\nabla v_t|_6^{\fr{1}{2}}
\leq Cc_4^{\fr{1}{2}}|g\nabla v_t|^{\fr{1}{2}}_{D^1_*}.
\end{split}
\end{equation*}
It follows from \eqref{2.61h}, \eqref{2.61illl}, \eqref{2.61t}, \eqref{2.62m}-\eqref{2.64},  \eqref{2.61ccvv}, the classical theory for elliptic equations and  Lemmas \ref{phiphi}-\ref{lu} that
\begin{equation}\label{2.65}\begin{split}
\displaystyle
|\sqrt{h^2+\epsilon^2}u_t|_{D^2}\leq &  C|l^{-\nu}\mathcal{G}|_2+C|G(\nabla\sqrt{h^2+\epsilon^2},u_t)|_2\\
\displaystyle
\leq& M(c_0)(|u_{tt}|_2+c_4^{\fr{1}{2}}|g\nabla v_t|^{\fr{1}{2}}_{D^1_*}+c_4^{8.5+\nu}),\\[2pt]
\displaystyle
|\sqrt{h^2+\epsilon^2}\nabla^2u_t|_2\leq & C(|\sqrt{h^2+\epsilon^2}u_t|_{D^2}+|\nabla u_t|_2(|\psi|_\infty+|\nabla\psi|_3)\\
&+|\psi|^2_\infty|u_t|_2|\varphi|_\infty)\\
\displaystyle
\leq & M(c_0)(|u_{tt}|_2+c_4^{\fr{1}{2}}|g\nabla v_t|^{\fr{1}{2}}_{D^1_*}+c_4^{8.5+\nu}),\\[2pt]
\displaystyle
|(h\nabla^2u)_t|_2\leq & C(|h\nabla^2u_t|_2+|h_t|_\infty|\nabla^2 u|_2)\\
\displaystyle
\leq & M(c_0)(|u_{tt}|_2+c_4^{\fr{1}{2}}|g\nabla v_t|^{\fr{1}{2}}_{D^1_*}+c_4^{8.5+\nu}),\\[2pt]
\displaystyle
|u|_{D^4}\leq& C|(h^2+\epsilon^2)^{-\fr{1}{2}}l^{-\nu}\mathcal{H}|_{D^2}\\
\leq &  M(c_0)(|u_{t}|_{D^2}
+c_4^{2})\\
\leq & M(c_0)(|u_{tt}|_2+c_4^{\fr{1}{2}}|g\nabla v_t|^{\fr{1}{2}}_{D^1_*}+c_4^{8.5+\nu}).
\end{split}
\end{equation}

 $\ef{ln}_2$ yield that  for multi-index $\xi\in\mathbb{Z}^3_+$ with $|\xi|=2$,  
\begin{equation}\label{2.67}
\begin{split}
&a_2L(\sqrt{h^2+\epsilon^2}\nabla^\xi u)=a_2\sqrt{h^2+\epsilon^2}\nabla^\xi Lu-a_2G(\nabla\sqrt{h^2+\epsilon^2},\nabla^\xi u)\\
=& \sqrt{h^2+\epsilon^2}\nabla^\xi\big[\big(\sqrt{h^2+\epsilon^2})^{-1}l^{-\nu}\mathcal{H}\big]
-a_2G(\nabla\sqrt{h^2+\epsilon^2},\nabla^\xi u),\\
\end{split}
\end{equation}
which, along with \eqref{2.61illl}, \eqref{2.61t}, \eqref{2.62m}-\eqref{2.63}, \eqref{2.61ccvv}-\eqref{2.65}, the classical theory for elliptic equations and  Lemmas \ref{phiphi}-\ref{lu},  implies that 
\begin{equation}\label{2.68}
\begin{split}
|\sqrt{h^2+\epsilon^2}\nabla^2u(t)|_{D^2}
\leq& C|\sqrt{h^2+\epsilon^2}\nabla^\xi\big[\big(\sqrt{h^2+\epsilon^2})^{-1}l^{-\nu}\mathcal{H}\big]|_{2}\\
&+C(|\psi|_\infty|u|_{D^3}+|\nabla\psi|_3|\nabla^2u|_6+|\nabla^2u|_2|\psi|^2_\infty|\varphi|_\infty)\\
\leq & M(c_0)(|u_{t}|_{D^2}+c_4^{2})\\
\leq & M(c_0)(|u_{tt}|_2+c_4^{\fr{1}{2}}|g\nabla v_t|^{\fr{1}{2}}_{D^1_*}+c_4^{8.5+\nu}).
\end{split}
\end{equation}

At last, it follows from \eqref{2.16}, \eqref{2.62m}, \eqref{2.65}, \eqref{2.68} and Lemma \ref{varphi} that  
\begin{equation}\label{2.69}
\int^{T_2}_0(|h\nabla^2u_t|^2_2+|u_t|^2_{D^2}+|u|^2_{D^4}+|h\nabla^2u|_{D^2}^2+|(h\nabla^2u)_t|^2_2)\text{d}t\leq  M(c_0).
\end{equation}

The proof of Lemma \ref{hu} is complete.
\end{proof}
\smallskip
\smallskip
\smallskip
Finally, we derive the time weighted estimates for the velocity $u$.
\begin{lemma}\label{hu2}Set $T_3=\min\{T_2,(1+Cc_5)^{-20-2\nu}\}$. Then for $t\in [0,T_3]$, 
	\begin{equation}\label{2.70}
	\begin{split}
	t|u_t(t)|^2_{D^2}+t|h\nabla^2u_t(t)|^2_2+t|u_{tt}(t)|^2_2+t|u(t)|^2_{D^4}(t) & \leq  M(c_0)c^{19+2\nu}_4,\\
	\int^t_0s(|u_{tt}|_{D^1_*}^2+|u_t|^2_{D^3}+|\sqrt{h} u_{tt}|_{D^1_*}^2)\text{d}s & \leq  M(c_0)c^{19+2\nu}_4.
	\end{split}
	\end{equation}
\end{lemma}
\begin{proof} Differentiating  $\ef{2.61j}$ with respect to $t$ yields 
\begin{equation}\label{2.71}
\begin{split}
&u_{ttt}+a_2\sqrt{h^2+\epsilon^2}l^\nu Lu_{tt}\\
=& -2v_t\cdot\nabla v_t-v_{tt}\cdot\nabla v-v\cdot\nabla v_{tt}-2a_1\phi_t\nabla l_t-a_1\phi\nabla l_{tt}-a_1\phi_{tt}\nabla l\\
&-2l_t\nabla\phi_t-l_{tt}\nabla\phi-l\nabla\phi_{tt}
+a_3l^\nu\psi_{tt}\cdot Q(v)+a_3(l^\nu)_{tt}\psi\cdot  Q(v)
\\
&+a_3l^\nu\psi \cdot Q(v_{tt})+2a_3(l^\nu)_t\psi_t\cdot Q(v)+2a_3l^\nu \psi_t\cdot Q(v_t)+2a_3(l^\nu)_t\psi\cdot Q(v_t)\\
&+a_2g(\nabla l^\nu)_{tt}\cdot Q(v)+a_2\nabla l^\nu\cdot (gQ(v))_{tt}\\
&+2a_2(\nabla l^\nu)_t\cdot(gQ(v))_t
-a_2(l^\nu)_{tt}\sqrt{h^2+\epsilon^2}L u\\
&-2a_2(l^\nu)_{t}\sqrt{h^2+\epsilon^2}L u_t-a_2l^\nu\frac{h}{\sqrt{h^2+\epsilon^2}}h_{tt} Lu\\
&-a_2l^\nu \frac{\epsilon^2h^2_t}{(h^2+\epsilon^2)^{\frac{3}{2}}}Lu-2a_2\frac{h}{\sqrt{h^2+\epsilon^2}}(h_{t} l^\nu Lu_t+(l^\nu)_th_{t} Lu).
\end{split}
\end{equation}
Multiplying $\ef{2.71}$ by $l^{-\nu}u_{tt}$ and integrating over $\mathbb{R}^3$ give
\begin{equation}\label{2.72}
\begin{split}
&\fr{1}{2}\fr{d}{dt}|l^{-\fr{\nu}{2}}u_{tt}|^2_2+a_2\alpha|(h^2+\epsilon^2)^{\fr{1}{4}}\nabla u_{tt}|^2_2+a_2(\alpha+\beta)|(h^2+\epsilon^2)^{\fr{1}{4}}\text{div}u_{tt}|^2_2=\sum^{4}_{i=1}H_i,
\end{split}
\end{equation}
where  $H_i$, $i=1,2,3,4$, are given and estimated as follows.
\begin{equation}\label{h1}
\begin{split}
H_1=&\int l^{-\nu}\big(-2v_t\cdot\nabla v_t-v_{tt}\cdot\nabla v-v \cdot\nabla v_{tt}-2a_1\phi_t\nabla l_t-a_1\phi\nabla l_{tt}\\
&-a_1\phi_{tt}\nabla l
-2l_t\nabla\phi_t-l_{tt}\nabla\phi-l\nabla\phi_{tt}\big)\cdot u_{tt}\\
\leq& C|l^{-\fr{\nu}{2}}|_\infty\big(|\nabla v_t|_6|v_t|_3+|\nabla v|_\infty|v_{tt}|_2+|v|_\infty|\nabla v_{tt}|_2\\
&+
|\phi_{tt}|_2|\nabla l|_\infty+|\phi_t|_6|\nabla l_t|_3+|\phi|_\infty|\nabla l_{tt}|_2\\
&+|l_t|_\infty|\nabla\phi_t|_2+|l_{tt}|_6|\nabla\phi|_3+|l|_\infty|\nabla\phi_{tt}|_2)|l^{-\fr{\nu}{2}}u_{tt}|_2,\\
H_2=&\int l^{-\nu} (a_3l^\nu \psi_{tt}\cdot Q(v)+a_3(l^\nu)_{tt}\psi\cdot Q(v)+a_3l^\nu\psi \cdot Q(v_{tt})\\
&+2a_3(l^\nu)_t\psi_t\cdot Q(v)
+2a_3l^\nu\psi_t\cdot Q(v_t)+2a_3(l^\nu)_t\psi\cdot Q(v_t))\cdot u_{tt}\\
\leq& C|l^{-\fr{\nu}{2}}|_\infty(|l^\nu|_\infty|\psi_{tt}|_2|\nabla v|_\infty+|(l^\nu)_{tt}|_6|\psi|_\infty|\nabla v|_3
\\
&+|l^\nu|_\infty|\psi|_\infty|\nabla v_{tt}|_2+|\psi_t|_3|(l^\nu)_t|_\infty|\nabla v|_6
\\
&+|\psi|_\infty|(l^\nu)_t|_\infty|\nabla v_t|_2
+|l^\nu|_\infty|\psi_t|_3|\nabla v_t|_6
)|l^{-\fr{\nu}{2}}u_{tt}|_2,
\end{split}
\end{equation}
\begin{equation}\label{h1mmm}
\begin{split}
H_3=&\int l^{-\nu}(a_2g(\nabla l^\nu)_{tt}\cdot Q(v)
+a_2\nabla l^\nu \cdot  (g Q(v))_{tt}\\
&+2a_2(\nabla l^\nu)_t\cdot(g Q(v))_t)\cdot u_{tt}\\
\leq& C|l^{-\fr{\nu}{2}}|_\infty(|(\nabla l^\nu)_{tt}|_2|g\nabla v|_\infty+|(\nabla l^\nu)_t|_3|(g\nabla v)_t|_6)|l^{-\fr{\nu}{2}}u_{tt}|_2\\
&+M(c_0)\int\big(|\nabla l^\nu\cdot g\nabla v_{tt}\cdot u_{tt}|
+|\nabla l^\nu\cdot g_{tt}\nabla v\cdot u_{tt}|\\
&+|\nabla l^\nu\cdot g_t\nabla v_{t}\cdot u_{tt}|\big)\\
\leq &C|l^{-\fr{\nu}{2}}|_\infty(|(\nabla l^\nu)_{tt}|_2|g\nabla v|_\infty+|(\nabla l^\nu)_t|_3|(g\nabla v)_t|_6)|l^{-\fr{\nu}{2}}u_{tt}|_2\\
&+M(c_0)(|\nabla l|_3|g_{tt}|_6|\nabla v|_\infty|l^{-\fr{\nu}{2}}u_{tt}|_2+|\nabla l|_3|g_t|_\infty|\nabla v_t|_6|l^{-\fr{\nu}{2}}u_{tt}|_2)\\
&+M(c_0)\int|\nabla l\cdot g\nabla v_{tt}\cdot u_{tt}|.
\end{split}
\end{equation}
In order  to estimate $\int|\nabla l\cdot g\nabla v_{tt}\cdot u_{tt}|$, one uses Lemma \ref{l}-\ref{gh} to get 
\begin{equation}\label{key1}
\begin{split}
&\int|\nabla l\cdot g\nabla v_{tt}\cdot u_{tt}|=\int|h^{\fr{1}{4}}\nabla l\sqrt{g}\cdot \nabla v_{tt}\cdot \fr{\sqrt{g}}{\sqrt{h}}h^{\fr{1}{4}}u_{tt}|\\
\leq & |h^{\fr{1}{4}}\nabla l|_6|\fr{\sqrt{g}}{\sqrt{h}}|_\infty|\sqrt{g}\nabla v_{tt}|_2|h^{\fr{1}{4}}u_{tt}|_3
\leq M(c_0)|\sqrt{g}\nabla v_{tt}|_2|h^{\fr{1}{4}}u_{tt}|_3.
\end{split}
\end{equation}
Note also that
\begin{equation}\label{key2}
|h^{\fr{1}{4}}u_{tt}|_3\leq |u_{tt}|_2^{\fr{1}{2}}|\sqrt{h}u_{tt}|_6^{\fr{1}{2}}.
\end{equation}
It then follows from the Sobolev and  \text{H\"older} inequalities, and Lemmas \ref{psi}-\ref{varphi} that 
\begin{equation}\label{key3}
\begin{split}
&\int|\nabla l\cdot g\nabla v_{tt}\cdot u_{tt}|\leq M(c_0)|\sqrt{g}\nabla v_{tt}|_2|\sqrt{h}u_{tt}|_6^{\fr{1}{2}}|l^{-\fr{\nu}{2}}u_{tt}|_2^{\fr{1}{2}}\\
\leq & M(c_0)|\sqrt{g}\nabla v_{tt}|_2(|\sqrt{h}\nabla u_{tt}|^{\fr{1}{2}}_2+|\varphi|_\infty^{\fr{1}{4}}|\psi|^{\fr{1}{2}}_\infty|l^{-\fr{\nu}{2}}u_{tt}|^{\fr{1}{2}}_2)|l^{-\fr{\nu}{2}}u_{tt}|^{\fr{1}{2}}_2\\
\leq & M(c_0)(|\sqrt{g}\nabla v_{tt}|_2|\sqrt{h}\nabla u_{tt}|^{\fr{1}{2}}_2|l^{-\fr{\nu}{2}}u_{tt}|^{\fr{1}{2}}_2+|\sqrt{g}\nabla v_{tt}|_2|\varphi|_\infty^{\fr{1}{4}}|\psi|^{\fr{1}{2}}_\infty|l^{-\fr{\nu}{2}}u_{tt}|_2)\\
\leq & \fr{a_2\alpha}{4}|\sqrt{h}\nabla u_{tt}|_2^2+M(c_0)|\sqrt{g}\nabla v_{tt}|_2^{\fr{4}{3}}|l^{-\fr{\nu}{2}}u_{tt}|_2^{\fr{2}{3}}+ M(c_0)|\sqrt{g}\nabla v_{tt}|_2|l^{-\fr{\nu}{2}}u_{tt}|_2.
\end{split}
\end{equation}
Hence, 
\begin{equation}\label{h3}
\begin{split}
H_3\leq& M(c_0)\Big(\big(|\nabla l_{tt}|_2+|l_t|_3|\nabla l_t|_6+|l_{tt}|_6|\nabla l|_3+|l_t|^2_\infty|\nabla l|_2)|g\nabla v|_\infty\\
&+(|\nabla l_t|_3+|l_t|_\infty|\nabla l|_3)|(g\nabla v)_t|_6\Big)|l^{-\fr{\nu}{2}}u_{tt}|_2+\fr{a_2\alpha}{4}|\sqrt{h}\nabla u_{tt}|_2^2\\
&+M(c_0)(|\nabla l|_3|g_{tt}|_6|\nabla v|_\infty|l^{-\fr{\nu}{2}}u_{tt}|_2+|\nabla l|_3|g_t|_\infty|\nabla v_t|_6|l^{-\fr{\nu}{2}}u_{tt}|_2)\\
&+M(c_0)(|\sqrt{g}\nabla v_{tt}|_2^{\fr{4}{3}}|l^{-\fr{\nu}{2}}u_{tt}|_2^{\fr{2}{3}}+ |\sqrt{g}\nabla v_{tt}|_2|l^{-\fr{\nu}{2}}u_{tt}|_2),
\end{split}
\end{equation}
where one has used 
\begin{equation*}\begin{split}
|(\nabla l^\nu)_{tt}|_2\leq& M(c_0)(|\nabla l_{tt}|_2+|l_t|_3|\nabla l_t|_6+|l_{tt}|_6|\nabla l|_3+|l_t|^2_\infty|\nabla l|_2),\\
|(\nabla l^\nu)_t|_3\leq& M(c_0)(|\nabla l_t|_3+|l_t|_\infty|\nabla l|_3).
\end{split}\end{equation*}

Similarly, one can also obtain that 
\begin{equation}\label{h4}
\begin{split}
H_4=&-\int l^{-\nu}\big(a_2\frac{h l^\nu}{\sqrt{h^2+\epsilon^2}}\nabla h\cdot Q(u_{tt})
+a_2(l^\nu)_{tt}\sqrt{h^2+\epsilon^2}L u\\
&+2a_2(l^\nu)_{t}\sqrt{h^2+\epsilon^2}L u_t
+a_2l^\nu\frac{h}{\sqrt{h^2+\epsilon^2}}h_{tt} Lu+a_2l^\nu\frac{\epsilon^2h^2_t}{(h^2+\epsilon^2)^{\frac{3}{2}}} Lu\\
&+2a_2\frac{h}{\sqrt{h^2+\epsilon^2}}(h_{t} l^\nu Lu_t+(l^\nu)_th_{t} Lu)\big) \cdot u_{tt}+\fr{1}{2}\int (l^{-\nu})_t|u_{tt}|^2\\
\leq& C|l^{-\fr{\nu}{2}}|_\infty
\big(|l^\nu|_\infty|\psi|_\infty|\sqrt{h}\nabla u_{tt}|_2|\varphi|^{\fr{1}{2}}_\infty
+|(l^\nu)_{tt}|_6|\sqrt{h^2+\epsilon^2}\nabla^2u|_3\\[1pt]
&+|(l^\nu)_t|_\infty|\sqrt{h^2+\epsilon^2}\nabla^2u_t|_2
+|l^\nu|_\infty|h_{tt}|_6|\nabla^2u|_3\\[1pt]
&+|l^\nu|_\infty|h_t|^2_\infty|\varphi|^3_\infty|\nabla^2 u|_2
+|l^\nu|_\infty|h_t|_\infty|\varphi|_\infty|h\nabla^2 u_{t}|_2\\[1pt]
&+|(l^\nu)_t|_\infty|h_t|_\infty|\nabla^2 u|_2\big)|l^{-\fr{\nu}{2}}u_{tt}|_2
+ M(c_0)|l_t|_\infty|l^{-\fr{\nu}{2}}u_{tt}|_2^2.
\end{split}
\end{equation}
Multiplying $\ef{2.72}$ by $t$ and integrating over $(\tau,t)$, one can obtain from above estimates on $H_i$, \eqref{2.16} and Lemmas \ref{phiphi}-\ref{hu} that 
\begin{equation}\label{2.73}
\begin{split}
&t|l^{-\fr{\nu}{2}}u_{tt}(t)|^2_2+\fr{a_2\alpha}{4}\int^t_\tau s|\sqrt{h}\nabla u_{tt}|^2_2\text{d}s\\
\leq & \tau|l^{-\fr{\nu}{2}}u_{tt}(\tau)|^2_2+M(c_0)c^{19+2\nu}_4(1+t)+M(c_0)c^{12+2\nu}_5\int^t_\tau s|l^{-\fr{\nu}{2}}u_{tt}|^2_2\text{d}s.
\end{split}
\end{equation}
It follows from $\ef{2.62m}$ and Lemma \ref{bjr} that  there exists a sequence $s_k$ such that
\begin{equation*}
s_k\longrightarrow 0, \quad \text{and} \quad s_k|u_{tt}(s_k,x)|^2_2\longrightarrow 0,\quad \text{as}\quad k\longrightarrow \infty.
\end{equation*}
Taking $\tau=s_k$ and letting $k\rightarrow \infty$ in $\ef{2.73}$, one can get  by  Gronwall's inequality that
\begin{equation}\label{2.74}
\begin{split}
t|u_{tt}(t)|^2_2+&\int^t_0s|\sqrt{h}\nabla u_{tt}|^2_2 \text{d}s+\int^t_0s|\nabla u_{tt}|^2_2\text{d}s\leq M(c_0)c^{19+2\nu}_4,
\end{split}
\end{equation}
for $0\leq t\leq T_3=\min\{T_2,(1+Cc_5)^{-20-2\nu}\}$.

It follows from $\ef{2.65}$  and $\ef{2.74}$  that
\begin{equation}\label{2.75}
t^{\fr{1}{2}}|\nabla^2u_t(t)|_2+t^{\fr{1}{2}}|h\nabla^2u_t(t)|_2+ t^{\fr{1}{2}}|\nabla^4u(t)|_2\leq  M(c_0)c^{9.5+\nu}_4.
\end{equation}
Next, for  giving  the $L^2$ estimate of $\nabla^3 u_t$, we  consider the $L^2$ estimates of 
$$(\nabla \mathcal{G},\nabla \widehat{G}=\nabla G(\nabla\sqrt{h^2+\epsilon^2},u_t)).$$
It follows from  \eqref{Gdingyi},  \eqref{hl2}, \eqref{2.61rrnnn}, \eqref{huagdingyi} and Lemmas \ref{phiphi}-\ref{hu} that 
\begin{equation}\label{2.61ccvvmmmm}
\begin{split}
|\mathcal{G}|_{D^1_*}\leq & C(|u_{tt}|_{D^1_*}+\|\nabla v\|_2|\nabla v_t|_2+|v|_\infty|\nabla^2 v_t|_2\\
&+ \|l\|_{L^\infty\cap D^1\cap D^3}\| \phi_t\|_2+\|l_t\|_{L^3\cap D^1_*\cap D^2}\|\phi\|_3\\
&+\|l^{\nu-1} \|_{1,\infty}\|l_t\|_{L^\infty\cap D^2}(\|\sqrt{h^2+\epsilon^2} Lu\|_1+|\psi|_\infty|\nabla^2 u|_2)\\
&+(1+|\psi|_\infty)(1+|\varphi|_\infty)\|h_t\|_{L^\infty\cap D^2}\|l^{\nu} \|_{1,\infty} \|\nabla^2 u\|_1\\
&+\|g_t\|_{L^\infty\cap D^1}\|\nabla l^\nu\|_2\|\nabla v\|_2+\|\nabla l^\nu\|_2(|\nabla g|_\infty|\nabla v_t|_2+|g\nabla^2 v_t|_2)\\
&+(|g\nabla v|_\infty+|\nabla g|_\infty\|\nabla v\|_2+\|g\nabla^2 v\|_1)\|l_t\|_{{ D^1_*\cap  D^2}}\|l^{\nu-1} \|_{ L^\infty\cap D^1\cap  D^3}\\
&+\|l^{\nu-1} \|_{1,\infty}\|l_t\|_{L^\infty \cap  D^1_*}\|\psi\|_{L^\infty \cap  D^{1,3}}\|\nabla v\|_2\\
&+\|l^{\nu} \|_{1,\infty}\|\psi_t\|_1\|\nabla v\|_2+\|l^{\nu} \|_{1,\infty}\|\psi\|_{L^\infty \cap  D^{1,3}}\|\nabla v_t\|_1)\\
\leq &M(c_0)(|\nabla u_{tt}|_2+c_4|v_t|_{D^2}+|g\nabla^2 v_t|_2+c_4^{11.5+\nu}),\\
|\widehat{G}|_{D^1_*}\leq &C( |\nabla\sqrt{h^2+\epsilon^2}|_\infty |\nabla^2 u_t|_2+|\nabla^2 \sqrt{h^2+\epsilon^2}|_3 |\nabla u_t|_6\\
&+|\nabla^3 \sqrt{h^2+\epsilon^2}|_2 | u_t|_\infty)\leq  M(c_0)(|u_t|_{D^2}+c^2_4).
\end{split}
\end{equation}
Hence $\ef{2.64}$, \eqref{2.61ccvvmmmm}, the classical theory for elliptic equations and   Lemmas \ref{phiphi}-\ref{hu} yield that  for $0\leq t\leq T_3$,
\begin{equation*}
\begin{split}
|\sqrt{h^2+\epsilon^2}u_t|_{D^3}\leq&  C|l^{-\nu}\mathcal{G}|_{D^1_*}+C|G(\nabla\sqrt{h^2+\epsilon^2},u_t)|_{D^1_*}\\
\leq &M(c_0)(|\nabla u_{tt}|_2+c_4(|u_t|_{D^2}+|v_t|_{D^2})+|g\nabla^2 v_t|_2+c_4^{11.5+\nu}),\\
|\sqrt{h^2+\epsilon^2}\nabla^3u_t(t)|_2\leq& C(|\sqrt{h^2+\epsilon^2}u_t|_{D^3}+|u_t|_\infty|\nabla^2\psi|_2+|\nabla u_t|_6|\nabla \psi|_3\\
&+|\nabla^2u_t|_2|\psi|_\infty
+|\nabla u_t|_2\|\psi\|^2_{L^\infty \cap D^{1,3}\cap D^2}|\varphi|_\infty+|u_t|_2|\psi|_\infty^3|\varphi|^2_2)\\
\leq & C|\sqrt{h^2+\epsilon^2}u_t|_{D^3}+M(c_0)(|u_t|_{D^2}+c_4^{2}),
\end{split}
\end{equation*}
which, along with $\ef{2.62}$, $\ef{2.74}$-$\ef{2.75}$ and Lemma \ref{varphi}, implies that 
\begin{equation*}\label{2.78}
\int^{T_3}_0s(|\sqrt{h^2+\epsilon^2}u_t|_{D^3}^2+|h\nabla^3u_t|_2^2+|\nabla^3u_t|^2_2)\text{d}s\leq M(c_0)c^{19+2\nu}_4.
\end{equation*}

The proof of Lemma \ref{hu2} is complete.
\end{proof}
It then  follows from Lemmas \ref{phiphi}-\ref{hu2} that for 
$0\leq t\leq T_3=\min\{T^*,(1+Cc_5)^{-20-2\nu}\}$,
\begin{equation*}\begin{aligned}
\|(\phi-\eta)(t)\|^2_{D^1_*\cap D^3}+\|\phi_t(t)\|^2_2+
|\phi_{tt}(t)|^2_2+\int^t_0\|\phi_{tt}\|^2_1\text{d}s\leq &Cc_4^6,\\
\|\psi(t)\|^2_{L^q\cap D^{1,3}\cap D^2}\leq M(c_0),\hspace{2mm}|\psi_t(t)|_2\leq Cc_3^2,\quad 
|h_t(t)|_\infty^2\leq & Cc_3^2c_4,\\ 
h(t,x)>\fr{1}{2c_0}, \quad
|\psi_t(t)|^2_{D^1_*}+\int^t_0(|\psi_{tt}|^2_2+|h_{tt}|^2_6)\text{d}s\leq & Cc_4^4,\\ 
\fr{2}{3}\eta^{-2\iota}<\varphi,
\ \  	|h^{\fr{1}{4}}\nabla l(t)|_6\leq  M(c_0),\quad   c_0^{-1}\leq|l(t)|_\infty\leq & M(c_0),\\
\|(l-\bar{l})(t)\|^2_{D^1_*\cap D^3}\leq M(c_0),\quad 
|l_t(t)|^2_{\infty}\leq  M(c_0)c_4^{8+2\nu},\quad |l_t(t)|^2_3\leq  & M(c_0)c_3^{8+2\nu},\\
|l_t(t)|^2_{D^1_*}\leq M(c_0)c_3^{12+2\nu}c_4,\ \ |\nabla^2l_t(t)|_2^2+\int^t_0|\nabla l_{tt}|^2_2\text{d}s\leq & M(c_0)c_4^{19+2\nu},
\end{aligned}\end{equation*}
\begin{equation*}\begin{aligned}
|\sqrt{h}\nabla u(t)|^2_2+\|u(t)\|^2_1+\int^t_0\big(\|\nabla u\|^2_1+|u_t|^2_2\big)\text{d}s\leq & M(c_0),\\
(|u|_{D^2}^2+|h\nabla^2u|^2_2+|u_t|^2_2)(t)+\int^t_0(|u|^2_{D^3}+|h\nabla^2u|_{D^1_*}^2+|u_t|^2_{D^1_*})\text{d}s\leq & M(c_0)c_2^3c_3,\\
(|u_t|^2_{D^1_*}+|\sqrt{h}\nabla u_t|^2_2+|u|^2_{D^3}+|h\nabla^2u|^2_{D^1_*})(t)+\int^t_0|u_t|^2_{D^2}\text{d}s\leq &  M(c_0)c^4_3,\\
\int^t_0(|u_{tt}|^2_2+|u|^2_{D^4}+|h\nabla u|^2_{D^1_*}+|h\nabla^2u|^2_{D^2}+|(h\nabla^2u)_t|^2_2)\text{d}s\leq & M(c_0),\\
	t|u_t(t)|^2_{D^2}+t|h\nabla^2u_t(t)|^2_2+t|u_{tt}(t)|^2_2+t|u(t)|^2_{D^4}(t)  \leq & M(c_0)c^{19+2\nu}_4,\\
	\int^t_0s(|u_{tt}|_{D^1_*}^2+|u_t|^2_{D^3}+|\sqrt{h} u_{tt}|_{D^1_*}^2)\text{d}s  \leq & M(c_0)c^{19+2\nu}_4.
\end{aligned}\end{equation*}
Therefore, defining the time
\begin{equation}\label{timedefinition}
T^*=
\min\{T,(1+CM(c_0)^{466+125\nu+8\nu^2})^{-20-2\nu}\},
\end{equation}
and constants
\begin{equation}
\begin{split}
c_1^2=& c_2^2=M(c_0)^2,\quad c_3^2=M(c_0)^{8},\\
c_4^2=& M(c_0)^{98+16\nu},\quad c_5^2=M(c_0)^{932+250\nu+16\nu^2},
\end{split}
\end{equation}
one can arrive at  the following desirable estimates  for $0\leq t\leq T^*$:
\begin{equation}\begin{aligned}\label{key1kk}
\|(\phi-\eta)(t)\|^2_{D^1_*\cap D^3}+\|\phi_t(t)\|^2_2+
|\phi_{tt}(t)|^2_2+\int^t_0\|\phi_{tt}\|^2_1\text{d}s\leq & c_5^2,\\
 \|\psi(t)\|^2_{L^q\cap D^{1,3}\cap D^2}+	|h^{\fr{1}{4}}\nabla l(t)|^2_6\leq c_1^2,\quad|h_t(t)|_\infty^2+|\psi_t(t)|_2\leq & c_4^2,\\
|l_t(t)|^2_{\infty}+|\nabla^2l_t(t)|_2^2+ |\psi_t(t)|^2_{D^1_*}+\int^t_0(|\nabla l_{tt}|^2_2+|\psi_{tt}|^2_2+|h_{tt}|^2_6)\text{d}s\leq & c_5^2,\\
 \inf_{[0,T_*]\times \mathbb{R}^3} l(t,x)\geq c^{-1}_0,\quad  h(t,x)>  \fr{1}{2c_0},  \quad \fr{2}{3}\eta^{-2\iota}<&\varphi,\\
\|(l-\bar{l})(t)\|^2_{D^1_*\cap D^3}\leq  c_1^2,\quad 
|l_t(t)|^2_{3}+|l_t(t)|^2_{D^1_*}\leq & c^2_4,\\
 |\sqrt{h}\nabla u(t)|^2_2+\|u(t)\|^2_1+\int^t_0\big(\|\nabla u\|^2_1+|u_t|^2_2\big)\text{d}s\leq &c_2^2,\\
(|u|_{D^2}^2+|h\nabla^2u|^2_2+|u_t|^2_2)(t)+\int^t_0(|u|^2_{D^3}+|h\nabla^2u|_{D^1_*}^2+|u_t|^2_{D^1_*})\text{d}s\leq & c_3^2,\\
(|u_t|^2_{D^1_*}+|\sqrt{h}\nabla u_t|^2_2+|u|^2_{D^3}+|h\nabla u|^2_{D^1_*}+|h\nabla^2u|^2_{D^1_*})(t)+\int^t_0|u_t|^2_{D^2}\text{d}s\leq & c_4^2,\\
\int^t_0(|u_{tt}|^2_2+|u|^2_{D^4}+|h\nabla^2u|^2_{D^2}+|(h\nabla^2u)_t|^2_2)\text{d}s\leq  c_4^2,\quad t|h\nabla^2u_t(t)|^2_2\leq & c_5^2,\\
t(|u_t|^2_{D^2}+|u_{tt}|^2_2+|u|^2_{D^4})(t)+
\int^t_0s(|u_{tt}|_{D^1_*}^2+|u_t|^2_{D^3}+|\sqrt{h} u_{tt}|_{D^1_*}^2)\text{d}s\leq & c_5^2.
\end{aligned}\end{equation}

\subsection{Passing to the limit $\epsilon \rightarrow 0$}

With the help of the $(\epsilon,\eta)$-independent estimates established in \ef{key1kk}, one can now obtain  the local well-posedness of   the linearized problem \ef{ln} with  $\epsilon=0$ and $\phi^\eta_0\geq \eta$.
\begin{lemma}\label{epsilon0}
	Let $\ef{can1}$ hold and $\eta>0$. Assume $(\phi_0,u_0,l_0,h_0)$ satisfy
	\eqref{a}-\eqref{2.8*}, and there exists a positive constant $c_0$ independent of $\eta$ such that \ef{2.14} holds. Then there exists a time $T^*>0$ independent of $\eta$, and a unique strong  solution $
	(\phi^\eta,u^\eta,l^\eta, h^\eta)$
	  in $[0,T^*]\times\mathbb{R}^3$ to \ef{ln} with  $\epsilon=0$ satisfying \ef{2.13} with $T$ replaced by $T^*$. Moreover, the uniform estimates (independent of $\eta$) \ef{key1kk} hold.
\end{lemma}
\begin{proof} The well-posedness of \ef{ln} with  $\epsilon=0$ can be proved as follows: 

\noindent\textbf{Step 1:} Existence. First, it follows from Lemmas \ref{ls}--\ref{hu2} that  for every $\epsilon>0$ and $\eta>0$,   there exist a time $T^*>0$ independent of $(\epsilon,\eta)$, and   a unique strong solution $(\phi^{\epsilon,\eta}, u^{\epsilon,\eta}, l^{\epsilon,\eta}, h^{\epsilon,\eta})(t,x)$ in $[0,T^*]\times \mathbb{R}^3$  to the linearized problem (\ref{ln}) satisfying  the estimates in  \ef{key1kk}, which are independent of  $(\epsilon,\eta)$.

Second,   by  using of the characteristic method and the  standard energy estimates for $(\ref{ln})_4$,  one can show that for $0\leq t\leq T^*$,
\begin{equation}\label{related}
|h^{\epsilon,\eta}(t)|_{\infty}+|\nabla h^{\epsilon,\eta}(t)|_2+ |h^{\epsilon,\eta}_t(t)|_{2}\leq C(A, R, c_v, \eta, \alpha, \beta, \gamma, \delta, T^*, \phi_0,u_0).
\end{equation}

Thus, it follows from  \ef{key1kk}, (\ref{related})  and   Lemma \ref{aubin} (see \cite{jm} )  that for any $R> 0$,  there exists a subsequence of solutions (still denoted by) $(\phi^{\epsilon,\eta}, u^{\epsilon,\eta},l^{\epsilon,\eta},  h^{\epsilon,\eta} )$, which  converges  to  a limit $(\phi^\eta,  u^\eta, l^\eta, h^\eta) $ as $\epsilon \rightarrow 0$ in the following  strong sense:
\begin{equation}\label{ert1}\begin{split}
&(\phi^{\epsilon,\eta}, u^{\epsilon,\eta}, l^{\epsilon,\eta}, h^{\epsilon,\eta}) \rightarrow (\phi^\eta,  u^\eta, l^\eta, h^\eta ) \ \ \text{ in } \ C([0,T^*];H^2(B_R))\ \ \text{as}\ \ \epsilon \rightarrow 0,
\end{split}
\end{equation}
where $B_R$ is a ball centered at origin with radius $R$, and 
in the following  weak  or  weak* sense:
\begin{equation}\label{ruojixian}
\begin{split}
(\phi^{\epsilon,\eta}-\eta, u^{\epsilon,\eta})\rightharpoonup  (\phi^\eta-\eta, u^\eta) \quad &\text{weakly* \ in } \ L^\infty([0,T^*];H^3),\\
(\phi^{\epsilon,\eta}_t,\psi^{\epsilon,\eta},h^{\epsilon,\eta}_t)\rightharpoonup ( \phi^\eta_t,\psi^\eta,h^{\eta}_t) \quad &\text{weakly* \ in } \ L^\infty([0,T^*];H^2),\\
u^{\epsilon,\eta}_t \rightharpoonup   u^\eta_t \quad &\text{weakly* \ in } \ L^\infty([0,T^*];H^1),\\
 t^{\frac{1}{2}}(\nabla^2u_t^{\epsilon,\eta},\nabla^4u^{\epsilon,\eta}) \rightharpoonup  t^{\frac{1}{2}}(\nabla^2u_t^{\eta}, \nabla^4u^{\eta}) \quad &\text{weakly* \ in } \ L^\infty([0,T^*];L^2),\\
(\phi^{\epsilon,\eta}_{tt}, t^{\frac{1}{2}}u^{\epsilon,\eta}_{tt}) \rightharpoonup ( \phi^{\eta}_{tt}, t^{\frac{1}{2}} u^{\eta}_{tt}) \quad &\text{weakly* \ in } \ L^\infty([0,T^*];L^2),\\
(h^{\epsilon,\eta}, h^{\epsilon,\eta}_t) \rightharpoonup  (h^{\eta}, h^{\eta}_t) \quad &\text{weakly* \ in } \ L^\infty([0,T^*];L^\infty),\\
\nabla u^{\epsilon,\eta} \rightharpoonup  \nabla u^{\eta} \quad &\text{weakly \ \  in } \ \ L^2([0,T^*];H^3),\\
u^{\epsilon,\eta}_{t} \rightharpoonup  u^{\eta}_{t}\quad &\text{weakly \ \  in } \ \ L^2([0,T^*];H^2),\\
\phi^{\epsilon,\eta}_{tt} \rightharpoonup    \phi^{\eta}_{tt} \quad &\text{weakly \ \  in } \ \ L^2([0,T^*];H^1),\\
(\psi^{\epsilon,\eta}_{tt},u^{\epsilon,\eta}_{tt}) \rightharpoonup  (\psi^{\eta}_{tt},u^{\eta}_{tt})\quad &\text{weakly \ \  in } \ \ L^2([0,T^*];L^2),\\
h^{\epsilon,\eta}_{tt} \rightharpoonup h^{\eta}_{tt} \quad &\text{weakly \ \  in } \ \ L^2([0,T^*];L^6),\\
t^{\frac{1}{2}}(\nabla u^{\epsilon,\eta}_{tt},\nabla^3u^{\epsilon,\eta}_t) \rightharpoonup  t^{\frac{1}{2}}(\nabla u^{\eta}_{tt},\nabla^3u^{\eta}_t) \quad &\text{weakly \ \  in } \ \ L^2([0,T^*];L^2),\\
l^{\epsilon,\eta}-\bar{l}\rightharpoonup l^\eta-\bar{l}\quad &\text{weakly* \ in } \ L^\infty([0,T^*];D^1_*\cap D^3),\\
l^{\epsilon,\eta}_t\rightharpoonup l^\eta_t\quad &\text{weakly* \ in } \ L^\infty([0,T^*];D^1_*\cap D^2),\\
l^{\epsilon,\eta}_{tt}\rightharpoonup l^\eta_{tt}\quad &\text{weakly \ in } \ L^2([0,T^*];D^1_*).
\end{split}
\end{equation}
Furthermore, it follows from  the lower semi-continuity of weak or weak* convergence  that $(\phi^\eta,  u^\eta, l^\eta, h^\eta) $  satisfies also the corresponding  estimates in \ef{key1kk} and (\ref{related}) except those weighted estimates on $u^\eta$.

Now,  the uniform estimates on $(\phi^\eta,  u^\eta, l^\eta, h^\eta) $ obtained above  and  the    convergences in (\ref{ert1})-(\ref{ruojixian}) imply  that
\begin{equation}\label{ruojixian2}
\begin{split}
\sqrt{h^{\epsilon,\eta}} \nabla u^{\epsilon,\eta} \rightharpoonup  \sqrt{h^{\eta}}  \nabla u^{\eta} \quad &\text{weakly* \ in } \ L^\infty([0,T^*];L^2),\\
\sqrt{h^{\epsilon,\eta}} \nabla u^{\epsilon,\eta}_t\rightharpoonup   \sqrt{h^{\eta}}  \nabla u^{\eta}_t\quad &\text{weakly* \ in } \ L^\infty([0,T^*];L^2),\\
h^{\epsilon,\eta}\nabla^2 u^{\epsilon,\eta} \rightharpoonup  h^{\eta}\nabla^2 u^{\eta} \quad &\text{weakly* \ in } \ L^\infty([0,T^*];H^1),\\
 (h^{\epsilon,\eta}\nabla^2 u^{\epsilon,\eta})_t \rightharpoonup  (h^{\eta}\nabla^2 u^{\eta})_t \quad &\text{weakly \ \  in } \ \ L^2([0,T^*];L^2),\\
h^{\epsilon,\eta}\nabla^2 u^{\epsilon,\eta} \rightharpoonup  h^{\eta}\nabla^2 u^{\eta}\quad &\text{weakly \ \  in } \ \ L^2([0,T^*]; D^1_*\cap D^2),\\
\end{split}
\end{equation}
which, along with the lower semi-continuity of weak or weak* convergence again, implies that $(\phi^\eta,  u^\eta, l^\eta, h^\eta) $  satisfies also the uniform weighted  estimates on  $ u^{\eta}$.

Now we  are ready to   show that $(\phi^\eta,  u^\eta, l^\eta, h^\eta) $ is a weak solution in the sense of distributions  to  \ef{ln} with  $\epsilon=0$.
First, multiplying $(\ref{ln})_2$ by  any given  $X(t,x)=(X^{(1)},X^{(2)},X^{(3)})^\top\in C^\infty_c ([0,T^*)\times \mathbb{R}^3)$ on both sides, and integrating over $[0,t)\times \mathbb{R}^3$ for $t\in (0,T^*]$, one has
\begin{equation}\label{zhenzheng1}
\begin{split}
&\int_0^t \int_{\mathbb{R}^3}  \Big(u^{\epsilon,\eta}\cdot X_t - (v\cdot \nabla) v \cdot X -a_1\phi^{\epsilon,\eta}\nabla l^{\epsilon,\eta}\cdot X-l^{\epsilon,\eta}\nabla \phi^{\epsilon,\eta}\cdot X\Big)\text{d}x\text{d}s\\
=&\int u^{\epsilon,\eta}(t,x)\cdot X(t,x)+a_2\int_0^t \int_{\mathbb{R}^3} l^{\epsilon,\eta}\sqrt{(h^{\epsilon,\eta})^2+\epsilon^2} Lu^{\epsilon,\eta}\cdot X \text{d}x\text{d}s\\
&-\int_0^t \int_{\mathbb{R}^3}\Big(a_2g\nabla (l^{\epsilon,\eta})^\nu\cdot Q(v)\cdot X+a_3(l^{\epsilon,\eta})^\nu\psi^{\epsilon,\eta} \cdot Q(v) \cdot X\Big)\text{d}x\text{d}s.
\end{split}
\end{equation}
It follows from   the  uniform estimates obtained above and  (\ref{ert1})-(\ref{ruojixian2}) that  one can take limit  $\epsilon \rightarrow 0$ in (\ref{zhenzheng1}) to get 
\begin{equation}\label{zhenzhengxx}
\begin{split}
&\int_0^t \int_{\mathbb{R}^3} \Big(u^{\eta}\cdot X_t - (v\cdot \nabla) v \cdot X -a_1\phi^{\eta}\nabla l^{\eta}\cdot X-l^{\eta}\nabla \phi^{\eta}\cdot X \Big)\text{d}x\text{d}s\\
=&\int u^{\eta}(t,x)\cdot X(t,x)+a_2\int_0^t \int_{\mathbb{R}^3} l^{\eta}h^{\eta} Lu^{\eta}\cdot X \text{d}x\text{d}s\\
&-\int_0^t \int_{\mathbb{R}^3}\Big(a_2g\nabla (l^{\eta})^\nu\cdot Q(v)\cdot X+a_3(l^{\eta})^\nu\psi^{\eta} \cdot Q(v) \cdot X\Big)\text{d}x\text{d}s.
\end{split}
\end{equation}
Similarly,  one can  show that $(\phi^\eta,l^\eta, h^\eta )$ satisfies also the equations  $(\ref{ln})_1$, $(\ref{ln})_3$-$(\ref{ln})_4$ and the initial data in the sense of distributions.
So it is clear that $(\phi^\eta,  u^\eta, l^\eta, h^\eta) $ is a weak solution in the sense of distributions  to the linearized problem \ef{ln} with  $\epsilon=0$ satisfying  the  following regularities
\begin{equation*}\label{zheng}
\begin{split}
&\phi^\eta-\eta\in L^\infty([0,T^*]; H^3), \quad h^\eta\in L^\infty([0,T^*]\times \mathbb{R}^3),\\
&(\nabla h^\eta, h^\eta_t)\in L^\infty([0,T^*];H^2), \  \ \ u^\eta\in L^\infty([0,T^*]; H^3)\cap L^2([0,T^*]; H^4), \\
 &  u^\eta_t \in L^\infty([0,T^*]; H^1)\cap L^2([0,T^*]; D^2),\ \   u^\eta_{tt}\in L^2([0,T^*];L^2),\\
  &  t^{\frac{1}{2}}u^\eta\in L^\infty([0,T^*];D^4),\quad   t^{\frac{1}{2}}u^\eta_t\in L^\infty([0,T^*];D^2)\cap L^2([0,T^*]; D^3),\\
&   t^{\frac{1}{2}}u^\eta_{tt}\in L^\infty([0,T^*];L^2)\cap L^2([0,T^*];D^1_*),\\
&l^\eta-\bar{l}\in L^\infty([0,T^*]; D^1_*\cap D^3),\quad l^\eta_t\in L^\infty([0,T^*]; D^1_*\cap D^2).
\end{split}
\end{equation*}
Therefore, this weak solution $(\phi^\eta,  u^\eta, l^\eta, h^\eta) $  is actually
a strong one.

\noindent\textbf{Step 2:} Uniqueness and time continuity.  Since the estimate $h^\eta>\frac{1}{2c_0}$ holds,  the uniqueness and the time continuity of the strong solution obtained above can be obtained by the  same arguments as in Lemma \ref{ls}, so details are omitted.

Thus the proof of Lemma  \ref{epsilon0} is complete.

\end{proof}

\subsection{Nonlinear approximation solutions away from vacuum}

In this subsection, we will prove the local well-posedness of the classical solution to the following Cauchy problem under the assumption that $\phi^\eta_0\geq \eta$.
\begin{equation}\label{nl}\left\{\begin{aligned}
&\phi^{\eta}_t+u^{\eta}\cdot\nabla\phi^{\eta}+(\gamma-1)\phi^{\eta} \text{div}u^{\eta}=0,\\[2pt]
&u^{\eta}_t+u^{\eta}\cdot \nabla u^{\eta}+a_1\phi^{\eta}\nabla l^{\eta}+l^{\eta}\nabla\phi^{\eta}+a_2(l^\eta)^\nu h^{\eta} Lu^{\eta}\\[2pt]
=& a_2 h^{\eta} \nabla (l^\eta)^\nu  \cdot Q(u^{\eta})+a_3(l^\eta)^\nu  \psi^{\eta} \cdot Q(u^{\eta}),\\[3pt]
&l^{\eta}_t+u^{\eta}\cdot\nabla l^{\eta}=a_4(l^\eta)^\nu  n^{\eta}(\phi^\eta)^{4\iota} H(u^{\eta}),\\[3pt]
&h^{\eta}_t+u^\eta\cdot\nabla h^\eta+(\delta-1) (\phi^\eta)^{2\iota}\text{div} u^{\eta}=0,\\[3pt]
&(\phi^{\eta},u^{\eta},l^{\eta},h^{\eta})|_{t=0}=(\phi^\eta_0,u^\eta_0,l^\eta_0,h_0^\eta)\\[3pt]
=&(\phi_0+\eta,u_0,l_0,(\phi_0+\eta)^{2\iota})\quad\quad x\in\mathbb{R}^3,\\[3pt]
&(\phi^{\eta},u^{\eta},l^{\eta},h^{\eta})\rightarrow (\eta,0,\bar{l},\eta^{2\iota}) \quad \text{as} \hspace{2mm}|x|\rightarrow \infty \quad \rm for\quad t\geq 0,
\end{aligned}\right.\end{equation}
where $\psi^\eta=\fr{a\delta}{\delta-1}\nabla h^\eta$ and $n^{\eta}=(ah^{\eta})^b$.
\begin{theorem}\label{nltm}
Let $\ef{can1}$ hold and $\eta>0$. Assume that $(\phi_0,u_0,l_0,h_0)$ satisfy
	\eqref{a}-\eqref{2.8*}, and there exists a positive constant $c_0$ independent of $\eta$ such that \ef{2.14} holds. Then there exists a time $T_*>0$ independent of $\eta$, and a unique strong solution 
	$$
	(\phi^\eta, u^\eta,l^\eta, h^\eta=\phi^{2\iota})
	$$
	  in $[0,T_*]\times\mathbb{R}^3$ to \ef{nl} satisfying \ef{2.13}, where $T_*$ is independent of  $\eta$.	
	 Moreover, the uniform estimates (independent of $\eta$) \ef{key1kk} hold for $(\phi^\eta,u^\eta,l^\eta,h^\eta)$ with $T^*$ replaced by $T_*$. 
\end{theorem}
The proof is given  by an iteration scheme described below. \\

 Let $(\phi^0,u^0,l^0,h^0)$ be the solution to the following Cauchy problem
\begin{equation}\label{xyz}\left\{\begin{aligned}
\displaystyle
&U_t+u_0\cdot \nabla U=0,\ \ \text{in} \ \ (0,\infty)\times\mathbb{R}^3, \\[3pt]
&Y_t-W\triangle Y=0,\ \ \text{in} \ \ (0,\infty)\times\mathbb{R}^3, \\[3pt]
\displaystyle
&Z_t+u_0\cdot \nabla Z=0,\ \ \text{in} \ \ (0,\infty)\times\mathbb{R}^3, \\[3pt]
\displaystyle
&W_t+u_0\cdot \nabla W=0,\ \ \text{in} \ \ (0,\infty)\times\mathbb{R}^3, \\[3pt]
\displaystyle
&(U,Y,Z,W)|_{t=0}=(\phi^\eta_0,u^\eta_0,l^\eta_0,h^\eta_0)\\[3pt]
=& (\phi_0+\eta,u_0,l_0,(\phi_0+\eta)^{2\iota})  \ \ \text{in} \ \ \mathbb{R}^3, \\[3pt]
\displaystyle
&(U,Y,Z,W)\rightarrow (\eta,0,\bar{l},\eta^{2\iota}) \quad \text{as} \hspace{2mm}|x|\rightarrow \infty \quad \rm for\quad t\geq 0.
\end{aligned}\right.\end{equation}
Choose a time $\bar{T}\in (0,T^*]$ small enough such that
\begin{equation}\begin{aligned}\label{it0}
\sup_{0\leq t\leq \bar{T}}\|\nabla h^0(t)\|^2_{L^q\cap D^{1,3}\cap D^2}\leq c^2_1,\ \ \sup_{0\leq t\leq \bar{T}}\|(l^0-\bar{l})(t)\|^2_{D^1_*\cap D^3}\leq c_1^2,&\\
\inf_{[0,T_*]\times \mathbb{R}^3} l^0(t,x)\geq c^{-1}_0,\quad 
\sup_{0\leq t\leq \bar{T}}\|u^0(t)\|^2_1+\int^{\bar{T}}_0\big(|u^0|^2_{D^2}+|u^0_t|^2_2\big)\text{d}t\leq c_2^2,&\\
\sup_{0\leq t\leq \bar{T}}(|u^0|_{D^2}^2+|h^0\nabla^2u^0|^2_2+|u^0_t|^2_2)(t)+\int^{\bar{T}}_0(|u^0|^2_{D^3}+|u^0_t|^2_{D^1_*})\text{d}t\leq c_3^2,&\\
\sup_{0\leq t\leq \bar{T}}(|u^0_t|^2_{D^1_*}+|h^0_t|^2_{D^1_*}+|u^0|^2_{D^3})(t)
+\int^{\bar{T}}_0(|u^0_t|^2_{D^2}+|u^0|^2_{D^4}+|u^0_{tt}|^2_2)\text{d}t\leq c_4^2,&\\
\sup_{0\leq t\leq \bar{T}}|h^0_t(t)|^2_\infty+\int^{\bar{T}}_0(|(h^0\nabla^2u^0)_t|^2_2+|h^0\nabla^2u^0|^2_{D^2})\text{d}t\leq c_4^2,&\\
\sup_{0\leq t\leq \bar{T}}\big(|\sqrt{h^0}\nabla u^0_t|^2_2+|h^0\nabla^2u^0|^2_{D^1_*}+|l^0_t|^2_{3}+|l^0_t|^2_{D^1_*}\big)(t)\leq c^2_4,&\\
\sup_{0\leq t\leq \bar{T}}\big(|l^0_t|^2_{\infty}+ |\nabla^2l^0_t|_2^2)(t)+\int^{\bar{T}}_0|\nabla l^0_{tt}|^2_2\text{d}t\leq c_5^2,&\\
\text{ess}\sup_{0\leq t\leq T^*}(t|u^0_t(t)|^2_{D^2}+t|h^0\nabla^2u^0_t(t)|^2_2+t|u^0_{tt}(t)|^2_2+t|u^0(t)|^2_{D^4})\leq  c_5^2,\\
	\int^t_0s(|u^0_{tt}|_{D^1_*}^2+|u^0_t|^2_{D^3}+|\sqrt{h^0} u^0_{tt}|_{D^1_*}^2)\text{d}s\leq  c_5^2.&
\end{aligned}\end{equation}

\begin{proof} \textbf{Step 1:} Existence. One starts with the initial  iteration $(v,w,g)=(u^0,l^0, h^0)$, and  can obtain a classical solution $(\phi^1,u^1,l^1,h^1)$ to the problem  \ef{ln} with $\epsilon=0$. Inductively, one constructs approximate sequences $(\phi^{k+1},u^{k+1},l^{k+1},h^{k+1})$ as follows:
given $(u^k,l^k, h^k)$ for $k\geq 1$, define  $(\phi^{k+1},u^{k+1},l^{k+1},h^{k+1})$ by solving the following problem:
\begin{equation}\label{k+1}\left\{\begin{aligned}
\displaystyle
&\phi^{k+1}_t+u^k\cdot\nabla\phi^{k+1}+(\gamma-1)\phi^{k+1}\text{div}u^k=0,\\[2pt]
\displaystyle
&(l^{k+1})^{-\nu}(u_t^{k+1}+u^k\cdot\nabla u^k+a_1\phi^{k+1}\nabla l^{k+1}+l^{k+1}\nabla\phi^{k+1})\\[2pt]
\displaystyle
&+a_2h^{k+1}Lu^{k+1}
=a_2(l^{k+1})^{-\nu}h^k\nabla(l^{k+1})^{\nu}\cdot Q(u^k)+a_3\psi^{k+1}\cdot Q(u^k),\\[2pt]
\displaystyle
&l^{k+1}_t+u^k\cdot\nabla l^{k+1}=a_4(l^k)^\nu n^{k+1}(h^k)^2H(u^k),\\[2pt]
\displaystyle
&h^{k+1}_t+u^k\cdot \nabla h^{k+1}+(\delta-1)h^k\text{div}u^k=0,\\[2pt]
\displaystyle
&(\phi^{k+1},u^{k+1},l^{k+1},h^{k+1})|_{t=0}=(\phi^\eta_0,u^\eta_0,l^\eta_0,h^\eta_0)\\[2pt]
=& (\phi_0+\eta,u_0,l_0,(\phi_0+\eta)^{2\iota}) \ \ \text{in} \ \ \mathbb{R}^3, \\[2pt]
\displaystyle
&(\phi^{k+1},u^{k+1},l^{k+1},h^{k+1})\longrightarrow (\eta,0,\bar{l},\eta^{2\iota}) \quad \text{as} \hspace{2mm}|x|\rightarrow \infty \quad \rm for\quad t\geq 0,
\end{aligned}\right.\end{equation}
where $\psi^{k+1}=\fr{a\delta}{\delta-1}\nabla h^{k+1}$ and $ n^{k+1}=(a h^{k+1})^b$.
It follows from Lemma \ref{epsilon0} and  mathematical induction that, by replacing $(v,w, g)$ with $(u^k,l^k, h^k)$,  one  can solve $\ef{k+1}$ locally in time, and the solution  $(\phi^{k+1},u^{k+1},l^{k+1},h^{k+1})$ $(k=0,1,2,\cdots)$ satisfies  the uniform estimates \ef{key1kk}. Moreover, $\psi^{k+1}$ satisfies
\begin{equation}\label{k+2}
\psi^{k+1}_t+\nabla(u^k\cdot \psi^{k+1})+(\delta-1)\psi^k\text{div}u^k+a\delta h^k\nabla\text{div}u^k=0.
\end{equation}

To show the strong convergence of $(\phi^k,u^k,l^k,\psi^k)$, we set
\begin{equation*}
\begin{split}
&\bar{\phi}^{k+1}=\phi^{k+1}-\phi^k,\ \ \bar{u}^{k+1}=u^{k+1}-u^k,\ \ \bar{l}^{k+1}=l^{k+1}-l^k,\ \ \\
& \bar{\psi}^{k+1}=\psi^{k+1}-\psi^k,\ \  \bar{h}^{k+1}=h^{k+1}-h^k,\ \  \bar{n}^{k+1}=n^{k+1}-n^k.
\end{split}
\end{equation*}
Then \eqref{k+1} and \ef{k+2} yield 
\begin{equation}\label{k+3}\left\{\begin{aligned}
\displaystyle
&\bar{\phi}^{k+1}_t+u^k\cdot\nabla\bar{\phi}^{k+1}+\bar{u}^k\cdot\nabla\phi^k+(\gamma-1)(\bar{\phi}^{k+1}\text{div}u^k+\phi^k\text{div}\bar{u}^k)=0,\\[4pt]
\displaystyle
&(l^{k+1})^{-\nu}(\bar{u}_t^{k+1}+u^k\cdot\nabla \bar{u}^k+\bar{u}^k\cdot\nabla u^{k-1}+a_1\bar{\phi}^{k+1}\nabla l^{k+1}+a_1\phi^k\nabla\bar{l}^{k+1}\\[2pt]
&+\bar{l}^{k+1}\nabla\phi^{k+1}+l^k\nabla\bar{\phi}^{k+1})
+a_2h^{k+1}L\bar{u}^{k+1}+a_2\bar{h}^{k+1}Lu^{k}\\[2pt]
\displaystyle
&=-\big((l^{k+1})^{-\nu}-(l^k)^{-\nu}\big)(u_t^{k}+u^{k-1}\cdot\nabla u^{k-1}+a_1\phi^{k}\nabla l^{k}+l^{k}\nabla\phi^{k})\\[2pt]
\displaystyle
&+a_2(l^{k+1})^{-\nu}\bigg(h^k\big(\nabla(l^{k+1})^{\nu}-\nabla(l^{k})^{\nu}\big)\cdot Q(u^k)
+h^k\nabla(l^k)^\nu\cdot Q(\bar{u}^k)\\[2pt]
\displaystyle
&+\bar{h}^k\nabla (l^k)^\nu\cdot Q(u^{k-1})\bigg)
+a_3\bar{\psi}^{k+1}\cdot Q(u^k)+a_3\psi^{k}\cdot Q(\bar{u}^k)\\[2pt]
\displaystyle
&+a_2\big((l^{k+1})^{-\nu}-(l^k)^{-\nu})\big)\big(h^{k-1}\nabla(l^k)^\nu\cdot Q(u^{k-1})\big),\\[2pt]
\displaystyle
&\bar{l}^{k+1}_t+u^k\cdot\nabla \bar{l}^{k+1}+ \bar{u}^{k}\cdot\nabla l^k=a_4(l^k)^\nu\big(n^{k+1}(h^k)^2(H(u^k)-H(u^{k-1}))\\[2pt]
&+\bar{h}^k(h^k+h^{k-1})n^{k+1}H(u^{k-1})+(h^{k-1})^2\bar{n}^{k+1}H(u^{k-1})\big)\\
&+a_4\big((l^k)^\nu-(l^{k-1})^\nu\big)n^{k}(h^{k-1})^2H(u^{k-1}),\\[2pt]
\displaystyle
&\bar{\psi}^{k+1}_t+\nabla(u^k\cdot \bar{\psi}^{k+1}+\bar{u}^k\cdot\psi^k)+(\delta-1)(\bar{\psi}^k\text{div}u^k+\psi^{k-1}\text{div}\bar{u}^k)\\[2pt]
&+a\delta(h^k\nabla\text{div}\bar{u}^k+\bar{h}^k\nabla\text{div}u^{k-1})=0,\\[4pt]
\displaystyle
&(\bar{\phi}^{k+1},\bar{u}^{k+1},\bar{l}^{k+1},\bar{\psi}^{k+1})|_{t=0}=(0,0,0,0)\ \  \text{in}\ \ \mathbb{R}^3,
\\[4pt]
\displaystyle
&(\bar{\phi}^{k+1},\bar{u}^{k+1},\bar{l}^{k+1},\bar{\psi}^{k+1})\longrightarrow (0,0,0,0) \quad \text{as} \ \ |x|\rightarrow \infty \quad \rm for\quad t\geq 0.
\end{aligned}\right.\end{equation}

We now give some necessary estimates on 
$(\bar{\phi}^{k+1},\bar{u}^{k+1},\bar{l}^{k+1},\bar{\psi}^{k+1},\bar{h}^{k+1})$ to be used later. We start with $\bar{h}^{k+1}$, for which one need the following lemma whose proof is given in Remark  {\rm\ref{cancelproof}} later. 
\begin{lemma}\label{cancel}
\begin{equation*}\label{hbar}
\begin{split}\bar{h}^{k+1},\ \bar{\phi}^{k+1},\ \bar{\psi}^{k+1} \in L^\infty([0,\bar{T}];H^2(\mathbb{R}^3))  \quad \text{for} \quad k=1,2,....
\end{split}
\end{equation*}
\end{lemma}
\begin{remark}This lemma  is helpful to deal with some singular terms \textcolor{red}{of} forms of  $\infty-\infty$.\end{remark}
Assume that Lemma \ref{cancel} holds at this moment, one can deduce the uniform estimate for $\bar{h}^{k+1}$ as follows. It follows from $\ef{k+1}_4$ that 
\begin{equation}\label{hbar}
\begin{split}
&\bar{h}^{k+1}_t+u^k\cdot\nabla \bar{h}^{k+1}+\bar{u}^k\cdot\nabla h^k+(\delta-1)(h^k\text{div}{u}^k-{h}^{k-1}\text{div}u^{k-1})=0.
\end{split}
\end{equation}
Multiplying \ef{hbar} by $5|\bar{h}^{k+1}|^4\bar{h}^{k+1}$
and using integration by parts yield
\begin{equation}\label{hbarkey}
\begin{split}
\fr{d}{dt}|\bar{h}^{k+1}(t,x)|^6_6\leq & C(|\nabla u^k|_\infty|\bar{h}^{k+1}|^6_6+|\bar{u}^k|_6|\psi^k|_\infty|\bar{h}^{k+1}|^5_6\\
&+(|h^k\text{div}{u}^k|_6+|h^{k-1}\text{div}{u}^{k-1}|_6)|\bar{h}^{k+1}|^5_6),
\end{split}
\end{equation} 
 with a generic constant $C$  independent of $\eta$ and $k$, which implies 
\begin{equation}\label{hbar6}
|\bar{h}^{k+1}(t,x)|_6\leq C\quad \text{for}\quad 0<t\leq \bar{T} \quad \text{and} \quad  k=0,1,2,....
\end{equation}



Second, 
multiplying $\ef{k+3}_4$ by $2\bar{\psi}^{k+1}$ and integrating over $\mathbb{R}^3$ lead to
\begin{equation}\label{psibarmmm}
\begin{split}
\fr{d}{dt}|\bar{\psi}^{k+1}|^2_2\leq& C|\nabla u^k|_\infty|\bar{\psi}^{k+1}|^2_2+C(|\bar{u}^k|_6|\nabla\psi^k|_3+|\psi^k|_\infty|\nabla \bar{u}^{k}|_2\\
&+|\bar{\psi}^k\text{div} u^{k}|_2+|h^k\nabla^2\bar{u}^k|_2\\
&+|\psi^{k-1}|_\infty|\nabla \bar{u}^{k}|_2+|\bar{h}^{k}|_6|\nabla^2u^{k-1}|_3)|\bar{\psi}^{k+1}|_2,
\end{split}
\end{equation}
which, along with  the fact 
\begin{equation}\label{barpsi2}
\begin{split}
|\bar{\psi}^k\text{div} u^{k}|_2=&|\psi^k\text{div} u^{k}-\psi^{k-1}\text{div} u^{k}|_2 
\leq C(|\psi^k|_\infty+|\psi^{k-1}|_\infty)|\text{div} u^{k}|_2, \\
|h^k\nabla^2\bar{u}^k|_2=&|h^k\nabla^2u^k-h^{k-1}\nabla^2u^{k-1}-\bar{h}^k\nabla^2 u^{k-1}|_2 \\
\leq &C\big(|h^k\nabla^2u^k|_2+|h^{k-1}\nabla^2u^{k-1}|_2+|\bar{h}^k|_6|\nabla^2 u^{k-1}|_3\big),
\end{split}
\end{equation}
the uniform estimates \eqref{key1kk}  for 
$(\phi^{k},u^{k},l^{k},h^{k})$ $(k=1,2,...)$,  \eqref{hbar6} and    Gronwall's inequality   that 
\begin{equation}\label{barpsi6}
|\bar{\psi}^{k+1}(t,x)|_2+|\bar\psi^{k+1}(t,x)|_6\leq C\quad \text{for}\quad  0<t\leq \bar{T} \quad \text{and} \quad  k=0,1,2,....
\end{equation}
Moreover, \eqref{psibarmmm}, Young's inequality and the uniform estimates \eqref{key1kk}  for 
$(\phi^{k},u^{k},\\ l^{k},h^{k})$ $(k=1,2,...)$ also imply that 
\begin{equation}\label{psibar}
\begin{split}
\fr{d}{dt}|\bar{\psi}^{k+1}|^2_2
\leq & C\sigma^{-1}|\bar{\psi}^{k+1}|^2_2+\sigma(|\sqrt{h^k}\nabla\bar{u}^k|^2_2+|\bar{\psi}^{k}|_2^2+|h^k\nabla^2\bar{u}^k|^2_2),
\end{split}
\end{equation}
where $\sigma\in (0,1)$ is a constant to be determined later.

Next,
multiplying $\ef{k+3}_1$ by $2\bar{\phi}^{k+1}$ and integrating over $\mathbb{R}^3$ give 
\begin{equation}\label{phibar1}
\fr{d}{dt}|\bar{\phi}^{k+1}|^2_2\leq C(|\nabla u^k|_\infty|\bar{\phi}^{k+1}|_2+|\bar{u}^{k}|_6|\nabla\phi^k|_3+|\nabla\bar{u}^{k}|_2|\phi^k|_\infty)|\bar{\phi}^{k+1}|_2.
\end{equation}
Applying $\partial^\varsigma_x$ $(|\varsigma|=1)$ to $\ef{k+3}_1$, multiplying by $2\partial^\varsigma_x\bar{\phi}^{k+1}$ and  integrating over $\mathbb{R}^3$, one gets
\begin{equation*}
\begin{split}
\fr{d}{dt}|\partial^\varsigma_x\bar{\phi}^{k+1}|^2_2\leq& C(|\nabla u^k|_\infty|\nabla\bar{\phi}^{k+1}|_2+|\nabla \phi^k|_\infty|\nabla\bar{u}^k|_2+|\bar{u}^k|_6|\nabla^2\phi^k|_3)|\nabla\bar{\phi}^{k+1}|_2\\
&+C(|\nabla^2u^k|_3|\bar{\phi}^{k+1}|_6+|\phi^k|_\infty|\nabla\text{div}\bar{u}^k|_2)|\nabla\bar{\phi}^{k+1}|_2.
\end{split}
\end{equation*}
Hence, it holds that for $t\in [0,\bar{T}]$, 
\begin{equation}\label{phibar2}
\fr{d}{dt}\|\bar{\phi}^{k+1}\|^2_1\leq C\sigma^{-1}\|\bar{\phi}^{k+1}\|^2_1+\sigma(|\sqrt{h^k}\nabla\bar{u}^k|^2_2+|h^k\nabla^2\bar{u}^k|^2_2),
\end{equation}
where  one has used that 
\begin{equation*}
|\nabla\bar{u}^k|_2\leq C|\sqrt{h^k}\nabla\bar{u}^k|_2,\ \ |\nabla^2\bar{u}^k|_2\leq C|h^k\nabla^2\bar{u}^k|_2.
\end{equation*}

On the other hand, it follows from the definition of $n^{k}=(ah^{k})^b$ and $h^k\geq Cc_0^{-1}$ that 
\begin{equation}\label{nbar}
\begin{split}
a^{-b}|\bar{n}^{k+1}|_6&=|(h^{k+1})^b-(h^{k})^b|_6\leq C|\bar{\psi}^{k+1}|_2,\\
a^{-b}\nabla\bar{n}^{k+1}&
=b((h^{k+1})^{b-1}-(h^{k})^{b-1})\psi^{k+1}+b(h^{k})^{b-1}\bar{\psi}^{k+1}.
\end{split}
\end{equation}
Applying derivative $\partial_x^\varsigma$ $(|\varsigma|=1)$ to $\ef{k+3}_3$, multiplying by $2\partial_x^\varsigma\bar{l}^{k+1}$ and integrating over $\mathbb{R}^3$ lead to
\begin{equation}\label{barl}
\begin{split}
&\fr{d}{dt}|\partial_x^\varsigma \bar{l}^{k+1}|^2_2
=-\int\partial_x^\varsigma(u^k\cdot\nabla \bar{l}^{k+1}+ \bar{u}^{k}\cdot\nabla l^k) 2\partial_x^\varsigma\bar{l}^{k+1}\\
&+a_4\int\partial_x^\varsigma\bigg((l^k)^\nu\big(n^{k+1}(h^k)^2(H(u^k)-H(u^{k-1}))\\
&+\bar{h}^k(h^k+h^{k-1})n^{k+1}H(u^{k-1})+(h^{k-1})^2\bar{n}^{k+1}H(u^{k-1})\big) \bigg) 2\partial_x^\varsigma\bar{l}^{k+1}\\
&+a_4\int\partial_x^\varsigma\bigg(\big((l^k)^\nu-(l^{k-1})^\nu\big)n^{k}(h^{k-1})^2H(u^{k-1})\bigg) 2\partial_x^\varsigma\bar{l}^{k+1}
\equiv\sum^3_{i=1}N_i.
\end{split}
\end{equation}
$N_1$, $N_2$ and $N_3$ can be estimated by \ef{key1kk}, \eqref{hbar6}, \eqref{barpsi6} and  \eqref{nbar} as follows: 
\begin{equation*}\label{i1}
\begin{split}
|N_1|
\leq & C|\nabla\bar{l}^{k+1}|^2_2+C|\sqrt{h^k}\nabla\bar{u}^k|_2|\nabla\bar{l}^{k+1}|_2,\\
|N_2|
\leq& C(|h^k\nabla\bar{u}^{k}|_6+|h^k\nabla^2\bar{u}^{k}|_2)|\nabla\bar{l}^{k+1}|_2\\
&+C(|\bar{\psi}^{k}|_2+|\bar{n}^{k+1}|_6+|\bar{\psi}^{k+1}|_2)|\nabla\bar{l}^{k+1}|_2\\
\leq& C(|\sqrt{h^k}\nabla\bar{u}^{k}|_2+|h^k\nabla^2\bar{u}^{k}|_2)|\nabla\bar{l}^{k+1}|_2
+C(|\bar{\psi}^{k}|_2+|\bar{\psi}^{k+1}|_2)|\nabla\bar{l}^{k+1}|_2,\\
|N_3|
\leq& C|\nabla\bar{l}^k|_2|\nabla\bar{l}^{k+1}|_2.
\end{split}
\end{equation*}
These, along  with $\ef{nbar}$-$\ef{barl}$,  yield that 
\begin{equation}\label{nbarl}
\begin{split}
\fr{d}{dt}|\nabla\bar{l}^{k+1}|^2_2\leq & C\sigma^{-1}|\nabla\bar{l}^{k+1}|^2_2+\sigma(|\sqrt{h^k}\nabla\bar{u}^k|^2_2+|h^k\nabla^2\bar{u}^{k}|^2_2)\\
&+\sigma(|\bar{\psi}^{k}|^2_2+|\nabla\bar{l}^k|_2^2+|\bar{\psi}^{k+1}|^2_2).
\end{split}
\end{equation}

We now estimate  $\bar{u}^{k+1}$. Multiplying $\ef{k+3}_2$ by $2\bar{u}^{k+1}$ and integrating over $\mathbb{R}^3$ yield that
\begin{equation}\label{ubar}
\begin{split}
&\fr{d}{dt}|(l^{k+1})^{-\fr{\nu}{2}}\bar{u}^{k+1}|^2_2+a_2\alpha|\sqrt{h^{k+1}}\nabla\bar{u}^{k+1}|^2_2+a_2(\alpha+\beta)|\sqrt{h^{k+1}}\text{div}\bar{u}^{k+1}|^2_2\\
\leq& C\sigma^{-1}|(l^{k+1})^{-\fr{\nu}{2}}\bar{u}^{k+1}|^2_2+\sigma(|\sqrt{h^{k}}\nabla\bar{u}^{k}|^2_2
+|h^k\nabla^2\bar{u}^k|^2_2+|\bar{\psi}^{k}|^2_2)\\
&+C(\|\bar{\phi}^{k+1}\|^2_1+|\bar{\psi}^{k+1}|^2_2+|\nabla \bar{l}^{k+1}|^2_2).
\end{split}
\end{equation}
Multiplying $\ef{k+3}_2$ by $2\bar{u}_t^{k+1}$ and integrating over $\mathbb{R}^3$ yield that
\begin{equation}\label{utbar}
\begin{split}
2|(l^{k+1})^{-\fr{\nu}{2}}\bar{u}^{k+1}_t|^2_2+\fr{d}{dt}(a_2\alpha|\sqrt{h^{k+1}}\nabla\bar{u}^{k+1}|^2_2&\\
+a_2(\alpha+\beta)|\sqrt{h^{k+1}}\text{div}\bar{u}^{k+1}|^2_2)
\equiv & \sum^8_{i=1}O_i,
\end{split}
\end{equation}
where, $O_i$, $i=1,2,\cdots,8$, are defined and estimated  as follows:
\begin{equation}\label{j7}
\begin{split}
O_1=&-2\int(l^{k+1})^{-\nu}(u^k\cdot\nabla\bar{u}^k+\bar{u}^k\cdot\nabla u^{k-1}+a_1\bar{\phi}^{k+1}\nabla l^{k+1}\\
&+a_1\phi^k\nabla\bar{l}^{k+1}+\bar{l}^{k+1}\nabla\phi^{k+1}+l^k\nabla\bar{\phi}^{k+1})\cdot\bar{u}^{k+1}_t\\
\leq& C(|\sqrt{h^k}\nabla \bar{u}^k|_2+\|\bar{\phi}^{k+1}\|_1+|\nabla\bar{l}^{k+1}|_2)|(l^{k+1})^{-\fr{\nu}{2}}\bar{u}^{k+1}_t|_2,\\
O_2=&-2\fr{\delta-1}{a\delta}a_2\int\psi^{k+1}\cdot Q(\bar{u}^{k+1})\cdot \bar{u}^{k+1}_t\\
\leq &  C|\sqrt{h^{k+1}}\nabla\bar{u}^{k+1}|_2|(l^{k+1})^{-\fr{\nu}{2}}\bar{u}^{k+1}_t|_2,\\
O_3=&-2a_2\int\bar{h}^{k+1}Lu^k\cdot \bar{u}^{k+1}_t\leq C|\bar{\psi}^{k+1}|_2|(l^{k+1})^{-\fr{\nu}{2}}\bar{u}^{k+1}_t|_2,\\
O_4=& -2\int\big((l^{k+1})^{-\nu}-(l^{k})^{-\nu}\big)(u_t^{k}+u^{k-1}\cdot\nabla u^{k-1}
\\
&+a_1\phi^{k}\nabla l^{k}+l^{k}\nabla\phi^{k})\cdot\bar{u}^{k+1}_t
\leq  C|\nabla\bar{l}^{k+1}|_2|(l^{k+1})^{-\fr{\nu}{2}}\bar{u}^{k+1}_t|_2,\\
O_5=&2a_3\int(\bar{\psi}^{k+1}\cdot Q(u^k)+\psi^{k}\cdot Q(\bar{u}^k))\cdot\bar{u}^{k+1}_t\\
\leq & C(|\bar{\psi}^{k+1}|_2+|\sqrt{h^k}\nabla\bar{u}^k|_2)|(l^{k+1})^{-\fr{\nu}{2}}\bar{u}^{k+1}_t|_2,\\
O_6=&2a_2\int\big((l^{k+1})^{-\nu}-(l^{k})^{-\nu}\big)h^{k-1}\nabla(l^k)^{\nu}\cdot Q(u^{k-1})\cdot\bar{u}^{k+1}_t\\
\leq& C|\nabla\bar{l}^{k+1}|_2|(l^{k+1})^{-\fr{\nu}{2}}\bar{u}^{k+1}_t|_2,\\
O_7=&2a_2\int(l^{k+1})^{-\nu}\bigg(h^k\big(\nabla(l^{k+1})^{\nu}-\nabla(l^{k})^{\nu}\big)\cdot Q(u^k)
+h^k\nabla(l^k)^\nu\cdot Q(\bar{u}^k)\\
&+\bar{h}^k\nabla(l^k)^\nu\cdot Q(u^{k-1}) \bigg)\cdot\bar{u}^{k+1}_t\\
\leq& C(|\nabla\bar{l}^{k+1}|_2+|\bar{\psi}^k|_2+|{(h^k)}^{\fr{1}{4}}\nabla l^k|_6|{(h^k)}^{\fr{3}{4}}\nabla\bar{u}^k|_3)|(l^{k+1})^{-\fr{\nu}{2}}\bar{u}^{k+1}_t|_2\\
\leq& C(|\nabla\bar{l}^{k+1}|_2+|\bar{\psi}^k|_2+|h^k\nabla^2\bar{u}^k|_2^{\fr{1}{2}}|\sqrt{h^k}\nabla\bar{u}^k|_2^{\fr{1}{2}})|(l^{k+1})^{-\fr{\nu}{2}}\bar{u}^{k+1}_t|_2,\\
O_8=&a_2\int h^{k+1}_t(\alpha|\nabla\bar{u}^{k+1}|^2+(\alpha+\beta)|\text{div}\bar{u}^{k+1}|^2)
\leq C|\sqrt{h^{k+1}}\nabla\bar{u}^{k+1}|^2_2.
\end{split}
\end{equation}
It follows from \ef{utbar}-\ef{j7} and Young's inequality that
\begin{equation}\label{utbar1}
\begin{split}
&|(l^{k+1})^{-\fr{\nu}{2}}\bar{u}^{k+1}_t|_2^2+\fr{d}{dt}a_2\alpha|\sqrt{h^{k+1}}\nabla\bar{u}^{k+1}|^2_2\\
\leq& C(|\sqrt{h^k}\nabla\bar{u}^k|^2_2+|\bar{\psi}^k|^2_2+\epsilon_1^{-1}|\sqrt{h^k}\nabla\bar{u}^k|_2^2)+\epsilon_1|h^k\nabla^2\bar{u}^k|_2^2\\
&+C(|\sqrt{h^{k+1}}\nabla\bar{u}^{k+1}|^2_2+\|\bar{\phi}^{k+1}\|^2_1+|\nabla \bar{l}^{k+1}|^2_2+|\bar{\psi}^{k+1}|^2_2),
\end{split}
\end{equation}
where $\epsilon_1>0$ is a small enough constant.

On the other hand, it follows directly from  $\ef{k+3}_2$ that 
\begin{equation}\label{hubar}
\begin{split}
|h^k\nabla^2\bar{u}^k|^2_2\leq& C(|\bar{u}^k_t|^2_2+|\sqrt{h^k}\nabla\bar{u}^k|^2_2+|\sqrt{h^{k-1}}\nabla\bar{u}^{k-1}|^2_2+\|\bar{\phi}^{k}\|^2_1+|\bar{\psi}^{k-1}|_2^2\\
&+|\nabla\bar{l}^{k}|^2_2+|\bar{\psi}^{k}|^2_2+\epsilon_2|h^{k-1}\nabla^2\bar{u}^{k-1}|^2_2+\epsilon_2^{-1}|\sqrt{h^{k-1}}\nabla\bar{u}^{k-1}|^2_2),
\end{split}
\end{equation}
where $\epsilon_2>0$ is a small  constant to be chosen.

Hence \ef{psibar}, \ef{phibar2},  \ef{nbarl}-\ef{ubar}, \ef{utbar1}-\ef{hubar} imply that 
\begin{equation}\label{allbar}
\begin{split}
&\fr{d}{dt}(\|\bar{\phi}^{k+1}\|^2_1+|\bar{\psi}^{k+1}|^2_2+|\nabla\bar{l}^{k+1}|^2_2+|(l^{k+1})^{-\fr{\nu}{2}}\bar{u}^{k+1}|^2_2\\
&+\upsilon_1a_2\alpha|\sqrt{h^{k+1}}\nabla\bar{u}^{k+1}|^2_2)\\
&+a_2\alpha|\sqrt{h^{k+1}}\nabla\bar{u}^{k+1}|^2_2+\upsilon_1|(l^{k+1})^{-\fr{\nu}{2}}\bar{u}_t^{k+1}|^2_2+\upsilon_2|h^k\nabla^2\bar{u}^k|^2_2\\
\leq& C\sigma^{-1}(\|\bar{\phi}^{k+1}\|^2_1+|\nabla\bar{l}^{k+1}|^2_2+|\bar{\psi}^{k+1}|^2_2+|(l^{k+1})^{-\fr{\nu}{2}}\bar{u}^{k+1}|^2_2)\\
&+C\sigma^{-1}\upsilon_1a_2\alpha|\sqrt{h^{k+1}}\nabla\bar{u}^{k+1}|^2_2+C(\sigma+\upsilon_1+\upsilon_1\epsilon_1^{-1}+\upsilon_2)|\sqrt{h^k}\nabla\bar{u}^k|^2_2\\
&+C(\sigma+\upsilon_1+\upsilon_2)(\|\bar{\phi}^k\|^2_1+|\bar{\psi}^k|^2_2+|\nabla\bar{l}^k|^2_2+|\bar{\psi}^{k-1}|_2^2)\\
&+C(\upsilon_2\epsilon_2^{-1}+\upsilon_2)|\sqrt{h^{k-1}}\nabla\bar{u}^{k-1}|^2_2+(6\sigma+\upsilon_1\epsilon_1)|h^k\nabla^2\bar{u}^k|_2^2\\
&+C\upsilon_2\epsilon_2|h^{k-1}\nabla^2\bar{u}^{k-1}|_2^2+C\upsilon_2|(l^{k})^{-\fr{\nu}{2}}\bar{u}_t^{k}|^2_2,
\end{split}
\end{equation}
where $\upsilon_1>0$, $\upsilon_2>0$ are small constants  to  be determined later.

Now, define 
\begin{equation*}
\begin{split}
\varGamma^{k+1}(t,\upsilon_1)=&\sup_{0\leq s \leq t}\|\bar\phi^{k+1}\|_1^2+\sup_{0\leq s \leq t}|\bar\psi^{k+1}|_2^2+\sup_{0\leq s \leq t}|\nabla\bar{l}^{k+1}|_2^2\\
&+\sup_{0\leq s \leq t}\alpha a_2\upsilon_1|\sqrt{h^{k+1}}\nabla\bar{u}^{k+1}|_2^2
+\sup_{0\leq s \leq t}|(l^{k+1})^{-\fr{\nu}{2}}\bar{u}^{k+1}|^2_2,
\end{split}
\end{equation*}
and set $\upsilon_2=8\sigma$, $\sigma=\upsilon_1^{\fr{3}{2}}$, $\epsilon_1=\upsilon_1^{\fr{1}{2}}$, $\epsilon_2=\upsilon_2^{\fr{1}{2}}$ so that $$
\upsilon_2-6\sigma-\upsilon_1\epsilon_1=\upsilon_1^{\fr{3}{2}}>0.
$$ 
Then it follows from \ef{allbar}  and Gronwall's inequality  that 
\begin{equation}\label{Gamma}
\begin{split}
&\varGamma^{k+1}(t,\upsilon_1)+\int^t_0\Big(a_2\alpha|\sqrt{h^{k+1}}\nabla\bar{u}^{k+1}|^2_2+\upsilon_1|(l^{k+1})^{-\fr{\nu}{2}}\bar{u}^{k+1}_t|^2_2
\\
&+\upsilon_1^{\fr{3}{2}}|h^k\nabla^2\bar{u}^k|^2_2\Big)\text{d}s\\
\leq  &C\Big(\int^t_0\Big(a_2\alpha(\sigma+\upsilon_1+\upsilon_1\epsilon_1^{-1}+\upsilon_2)|\sqrt{h^{k}}\nabla\bar{u}^{k}|^2_2\\
&+(\upsilon_2\epsilon_2^{-1}+\upsilon_2)|\sqrt{h^{k-1}}\nabla\bar{u}^{k-1}|^2_2\\
&+\upsilon_2\epsilon_2|h^{k-1}\nabla^2\bar{u}^{k-1}|^2_2+\upsilon_2|(l^{k})^{-\fr{\nu}{2}}\bar{u}^{k}_t|^2_2\Big)\text{d}s\\
&+(\sigma+\upsilon_1+\upsilon_2)t(\varGamma^{k}(t,\upsilon_1)+\varGamma^{k-1}(t,\upsilon_1)\Big)\exp{(C\sigma^{-1}t)}\\
\leq&C\Big(\int^t_0\big(a_2\alpha\upsilon_1^{\fr{1}{2}}|\sqrt{h^{k}}\nabla\bar{u}^{k}|^2_2
+\upsilon_2^{\fr{1}{2}}|\sqrt{h^{k-1}}\nabla\bar{u}^{k-1}|^2_2\\
&+\upsilon_2^{\fr{3}{2}}|h^{k-1}\nabla^2\bar{u}^{k-1}|^2_2+\upsilon_2|(l^{k})^{-\fr{\nu}{2}}\bar{u}^{k}_t|^2_2\big)\text{d}s
\\
&+\upsilon_1t(\varGamma^{k}(t,\upsilon_1)+\varGamma^{k-1}(t,\upsilon_1))\Big)\exp{(C\sigma^{-1}t)}.
\end{split}
\end{equation}
Now one first chooses  $\upsilon_1=\bar{\upsilon}\in (0,1)$ such that
$C\bar{\upsilon}^{\fr{1}{2}}\leq \fr{1}{64}$, and then chooses  $T_*\in (0,\bar{T}]$ such that
\begin{equation*}
\begin{split}
&(T_*+1)\exp (C\bar{\upsilon}^{-\fr{3}{2}}T_*)\leq 2,\\
&C\upsilon_1^{\fr{1}{2}}\exp(C\upsilon_1^{-\fr{3}{2}}T_*)=C\bar{\upsilon}^{\fr{1}{2}}\exp(C\bar{\upsilon}^{-\fr{3}{2}}T_*)\leq \fr{1}{32},\\
&C\upsilon_2^{\fr{1}{2}}\exp(C\upsilon_1^{-\fr{3}{2}}T_*)=C(8\upsilon_1^{\fr{3}{2}})^{\fr{1}{2}}\exp(C\upsilon_1^{-\fr{3}{2}}T_*)\leq \fr{1}{16},\\
&C\upsilon_2^{\fr{3}{2}}\exp(C\upsilon_1^{-\fr{3}{2}}T_*)\leq \fr{\sqrt{2}}{4}\bar{\upsilon}^{\fr{3}{2}},\\
&C\upsilon_2\exp(C\upsilon_1^{-\fr{3}{2}}T_*)\leq \fr{\bar{\upsilon}}{2},\quad C\upsilon_1 T_*\exp(C\upsilon_1^{-\fr{3}{2}}T_*)\leq\fr{1}{32}.
\end{split}
\end{equation*}
We can get finally that
\begin{equation*}\label{Gamma1}
\begin{split}
&\sum_{k=1}^{\infty}\Big(\varGamma^{k+1}(T_*,\bar{\upsilon})+\int^{T_*}_0(a_2\alpha|\sqrt{h^{k+1}}\nabla\bar{u}^{k+1}|^2_2+\bar{\upsilon}|\bar{u}^{k+1}_t|^2_2
+\bar{\upsilon}^{\fr{3}{2}}|h^k\nabla^2\bar{u}^k|^2_2)\text{d}s\Big)<\infty,
\end{split}
\end{equation*}
which, along with   the local estimates \ef{key1kk} independent of $k$, yields in particular that 
\begin{equation}\label{barpsi}
\begin{split}
\lim_{k\rightarrow \infty}(\|\bar\phi^{k+1}\|_{s'}+\|\bar u^{k+1}\|_{s'}+\|\bar l^{k+1}\|_{L^\infty\cap D^1\cap D^{s'}})=&0,\\
\lim_{k\rightarrow \infty}(\|\bar\psi^{k+1}\|_{L^\infty\cap L^q}+|\bar h^{k+1}|_\infty)=&0.
\end{split}
\end{equation}
 for any $s'\in [1,3)$.
Then   there exist a subsequence (still denoted by $(\phi^k,u^k, l^k, \psi^k)$) and   limit functions  $(\phi^\eta,u^\eta,l^\eta,\psi^\eta)$ such that
\begin{equation}\label{key0}
\begin{split}
&(\phi^k-\eta,u^k)\rightarrow (\phi^\eta-\eta, u^\eta) \ \ \text{in}\ \ L^\infty([0,T_*];H^{s'}(\mathbb{R}^3)),\\
 & l^k-\bar{l}\rightarrow l^\eta-\bar{l}\ \ \text{in}\ \ L^\infty([0,T_*];L^\infty\cap D^1\cap D^{s'}(\mathbb{R}^3)),\\
 &\psi^k\rightarrow\psi^\eta \ \ \text{in}\ \ L^\infty([0,T_*];L^\infty\cap L^q(\mathbb{R}^3)),\\
&h^k\rightarrow h^\eta \ \ \text{in}\ \ L^\infty([0,T_*];L^\infty(\mathbb{R}^3)).
\end{split}
\end{equation}
 Again by virtue of   the local estimates  \ef{key1kk} independent of $k$,  there exists a subsequence (still denoted by $(\phi^k,u^k, l^k, \psi^k)$) converging  to the  limit $(\phi^\eta,u^\eta,l^\eta,\psi^\eta)$ in the weak or weak* sense.
According to the lower semi-continuity of norms,  the corresponding estimates in \ef{key1kk}  for  $(\phi^\eta,u^\eta, l^\eta,\psi^\eta)$ still hold except those weighted estimates on $u^\eta$.

Next, it remains  to show 
\begin{equation}\label{relationrecover}
 \psi^\eta=\frac{a\delta}{\delta-1}\nabla (\phi^\eta)^{2\iota}.
\end{equation}
Set
$$
  \psi^*=\psi^\eta-\frac{a\delta}{\delta-1}\nabla (\phi^\eta)^{2\iota}.
$$
Then it follows from $(\ref{nl})_1$ and  $(\ref{nl})_4$ that
\begin{equation}\label{liyhn-recover}
\begin{cases}
\displaystyle
\psi^*_t+\sum_{k=1}^3 A_k(u^\eta) \partial_k\psi^*+B^*(u^\eta)\psi^*=0,\\[12pt]
\displaystyle
 \psi^*|_{t=0}=0 \quad \text{in}\quad \mathbb{R}^3,\\[12pt]
\psi^*\rightarrow 0 \quad \text{as} \ \ |x|\rightarrow \infty \quad \rm for\quad t\geq 0,
 \end{cases}
\end{equation}
which, together with the standard energy method for symmetric  hyperbolic systems, implies that
$$
\psi^*=0 \quad \text{for}\quad (t,x)\in [0,T_*]\times \mathbb{R}^3.
$$
Thus (\ref{relationrecover}) has been verified.\\

Note also that the following weak convergence holds true:
\begin{equation*}
 h^k\nabla^2u^k\rightharpoonup h^\eta\nabla^2u^\eta \ \  \text{weakly*} \ \ \text{in} \ \ L^\infty([0,T_*];L^2).
\end{equation*}
Indeed, \eqref{key0} gives
\begin{equation*}
\begin{split}
&\int^{T_*}_0\int_{\mathbb{R}^3}(h^k\nabla^2u^k-h^\eta\nabla^2u^\eta)X\text{d}x\text{d}t\\
=&\int^{T_*}_0\int_{\mathbb{R}^3}\big((h^k-h^\eta)\nabla^2u^k+h^\eta(\nabla^2u^k-\nabla^2u^\eta)\big)X\text{d}x\text{d}t\\
\leq &C(\sup_{0\leq t\leq T_*}|h^k-h^\eta|_\infty+\|\nabla^2u^k-\nabla^2u^\eta\|_{L^\infty([0,T_*];L^2)})T_*\rightarrow 0\ \ \text{as}\ \ k\rightarrow\infty
\end{split}
\end{equation*}
for any test function $X(t,x)\in C^\infty_c([0,T_*)\times \mathbb{R}^3)$, which implies that
\begin{equation*}
 h^k\nabla^2u^k\rightharpoonup h^\eta\nabla^2u^\eta \ \  \text{weakly*} \ \ \text{in} \ \ L^\infty([0,T_*];L^2).
\end{equation*}
Similarly, one can also obtain that
\begin{equation}\label{wes}
\begin{split}
&\sqrt{h^k}(\nabla u^k,\nabla u^k_t)\rightharpoonup \sqrt{h^\eta}(\nabla u^\eta,\nabla u^\eta_t) \ \  \text{weakly*} \ \ \text{in} \ \ L^\infty([0,T_*];L^2),\\
&h^k\nabla^2u^k\rightharpoonup h^\eta\nabla^2u^\eta\ \  \text{weakly*} \ \ \text{in} \ \  L^\infty([0,T_*];D^1_*),\\
&h^k\nabla^2u^k\rightharpoonup h^\eta\nabla^2u^\eta\ \  \text{weakly} \ \ \text{in} \ \ L^2([0,T_*];D^1_*\cap D^2),\\
&(h^k\nabla^2u^k)_t\rightharpoonup (h^\eta\nabla^2u^\eta)_t\ \  \text{weakly} \ \ \text{in} \ \ L^2([0,T_*];L^2).\\
\end{split}
\end{equation}
Hence the corresponding weighted estimates for the velocity in  \ef{key1kk} still hold for the limit. Thus,  $(\phi^\eta,u^\eta, l^\eta,\psi^\eta)$ is a weak solution in the sense of distributions to the  Cauchy problem \eqref{nl}.

\textbf{Step 2:} Uniqueness.  Let $(\phi_1,  u_1, l_1,\psi_1)$ and $(\phi_2, u_2, l_2,\psi_2)$ be two strong solutions to the   Cauchy problem (\ref{nl}) satisfying the  estimates in \ef{key1kk}.

Set
\begin{equation*}
\begin{split}
h_i=&\phi^{2\iota}_i,\quad n_i=(ah_i)^b,\quad  i=1,2; \quad \bar{h}=h_1-h_2,\\
\bar{\phi}=& \phi_1-\phi_2,\  \    \bar{u}=u_1-u_2, \ \  \bar{l}=l_1-l_2,\ \  \bar{\psi}=\psi_1-\psi_2.
\end{split}
\end{equation*}
Then it follows from the equations in  (\ref{nl}) that
 \begin{equation}
\label{eq:1.2wcvb}
\begin{cases}
  \displaystyle
\quad \bar{\phi}_t+u_1\cdot \nabla\bar{\phi}+\bar{u} \cdot\nabla\phi_2+(\gamma-1)(\bar{\phi}\text{div}u_1 +\phi_2\text{div}\bar{u})=0,\\[8pt]
 \displaystyle
\quad  \bar{u}_t+ u_1\cdot\nabla \bar{u}+l_1\nabla \bar{\phi}+a_1\phi_1\nabla \bar{l}+ a_2l^\nu_1 h_1L\bar{u} \\[8pt]
\displaystyle
=- \bar{u} \cdot \nabla u_2-a_1\bar{\phi} \nabla l_2-\bar{l} \nabla \phi_2- a_2(l^\nu_1 h_1-l^\nu_2 h_2)Lu_2\\[8pt]
\displaystyle
\quad +a_2( h_1 \nabla l^\nu_1 \cdot Q(u_1)-h_2\nabla l^\nu_2 \cdot Q(u_2))\\[8pt]
\displaystyle
\quad +a_3(l^\nu_1 \psi_1 \cdot Q(u_1)-l^\nu_2 \psi_2 \cdot Q(u_2)),\\[8pt]
\quad \bar{l}_t+u_1\cdot\nabla \bar{l}+\bar{u}\cdot \nabla l_2=a_4(l^\nu_1 n_1\phi^{4\iota}_1 H(u_1)-l^\nu_2 n_2\phi^{4\iota}_2 H(u_2)),\\[8pt]
\quad \bar{h}_t+u_1\cdot \nabla\bar{h}+\bar{u}\cdot\nabla h_2+(\delta-1)(\bar{h} \text{div}u_2 +h_1\text{div}\bar{u})=0,\\[2pt]
 \displaystyle
\quad \bar{\psi}_t+\sum_{k=1}^3 A_k(u_1)
\partial_k\bar{\psi}+B(u_{1})\bar{\psi}+a\delta(\bar{h} \nabla\text{div}u_2 +h_1\nabla \text{div}\bar{u})\\[2pt]
\displaystyle
=-\sum_{l=k}^3A_k(\bar{u})
\partial_k\psi_{2}-B(\bar{u}) \psi_{2},\\[2pt]
\displaystyle
\quad
(\bar{\phi},\bar{u},\bar{l},\bar{h},\bar{\psi})|_{t=0}=(0,0,0,0,0) \quad \text{in}\quad \mathbb{R}^3,\\[4pt]
\displaystyle
\quad
(\bar{\phi},\bar{u},\bar{l},\bar{h},\bar{\psi})\longrightarrow (0,0,0,0,0) \quad \text{as} \ \ |x|\rightarrow \infty \quad \rm for\quad t\geq 0.
\end{cases}
\end{equation}

Set
\begin{equation*}
\begin{split}
\Phi(t)=&\|\bar\phi\|_1^2+|\bar\psi|_2^2+|\nabla\bar{l}|_2^2+a_2\alpha|\sqrt{h_1}\nabla\bar{u}|_2^2
+|l_1^{-\fr{\nu}{2}}\bar{u}|^2_2.
\end{split}
\end{equation*}
In a similar way for  the derivation of (\ref{Gamma}), one  can show that
\begin{equation}\label{gonm}\begin{split}
\frac{d}{dt}\Phi(t)+C\big( |\nabla \bar{u}|^2_2+ |l_1^{-\fr{\nu}{2}}\bar{u}_t|^2_2\big)\leq H(t)\Phi (t),
\end{split}
\end{equation}
with a continuous  function $H(t)$ satisfying 
$$ \int_{0}^{t}H(s)\ \text{\rm d}s\leq C \quad \text{for} \quad 0\leq t\leq T_*.$$ It follows from  Gronwall's inequality that
$\bar{\phi}=\bar{l}=0$ and $\bar{\psi}=\bar{u}=0$.
Thus the uniqueness is obtained.

\textbf{Step 3.} The time-continuity  follows easily from  the same  procedure as in Lemma \ref{ls}. 

Thus the proof of Theorem \ref{nltm} is completed.

\end{proof}

\begin{remark}\label{cancelproof}
It remains to prove Lemma {\rm\ref{cancel}}.
\begin{proof}

Define
$X_R(x)=X(x/R)$,
where $X(x)\in C^\infty_c(\mathbb{R}^3)$ is a  truncation function  satisfying
 \begin{equation}\label{xr}
0\leq X(x) \leq 1, \quad \text{and} \quad X(x)=
 \begin{cases}
1 \;\qquad  \text{if} \ \ |x|\leq 1,\\[8pt]
0   \ \ \ \ \ \ \    \text{if} \ \   |x|\geq 2.
 \end{cases}
 \end{equation}

Set $\bar{h}^{k+1,R}=\bar{h}^{k+1} X_R$. Then $\ef{hbar}$ yields
 \begin{equation}
\label{localversion}
\begin{split}
\displaystyle
& \bar{h}^{k+1,R}_t+u^k\cdot \nabla\bar{h}^{k+1,R} +(\delta-1)(\bar{h}^{k,R} \text{div}u^{k-1} +h^{k}\text{div}\bar{u}^kX_R)\\[6pt]
=&u^k\bar{h}^{k+1}\cdot \nabla X_R-\frac{\delta-1}{a\delta}\bar{u}^k\cdot\psi^{k}X_R.
\end{split}
\end{equation}
 Multiplying \ef{localversion} by $2\bar{h}^{k+1,R}$ and integrating over $\mathbb{R}^3$, one can get
 \begin{equation}
\label{localversion2}
\begin{split}
\displaystyle
\frac{d}{dt}|\bar{h}^{k+1,R}|_2\leq & C|\nabla u^k|_\infty|\bar{h}^{k+1,R}|_2 +C\big(|\bar{h}^{k}|_\infty |\text{div}u^{k-1}|_2 +|h^{k}|_\infty|\text{div}\bar{u}^k|_2)\\[6pt]
&+C\big(|u^k|_2|\bar{h}^{k+1}|_\infty+|\bar{u}^k|_2|\psi^{k}|_\infty\big)\\
\leq & \hat{C}|\bar{h}^{k+1,R}|_2+\hat{C},
\end{split}
\end{equation}
with $\hat{C}>0$ being a generic constant depending on $\eta$, but independent of $R$. Then  Gronwall's inequality yields that
$$
|\bar{h}^{k+1,R}(t)|_2\leq  \hat{C}\exp{(\hat{C}\bar{T})} \quad \text{for} \quad (t,R)\in [0,\bar{T}]\times [0,\infty).$$
Hence, $\bar{h}^{k+1}\in L^\infty([0,\bar{T}]; L^2(\mathbb{R}^3))$, which, along with
$\bar{h}^{k+1}=h^{k+1}-h^k$ and $$\frac{a\delta}{\delta-1}\nabla h^k=\psi^k\in L^\infty([0,\bar{T}];L^q\cap D^{1,3}\cap D^2),$$ implies  that $\bar{h}^{k+1}\in L^\infty([0,\bar{T}]; H^3(\mathbb{R}^3))$.

Similarly, one can show that $\bar{\phi}^{k+1},\ \bar{\psi}^{k+1} \in L^\infty([0,\bar{T}];H^2(\mathbb{R}^3))$.

\end{proof}

\end{remark}

\subsection{Limit from the non-vaccum flows to the flow with far field vacuum}
Based on the uniform estimates in \ef{key1kk}, we are ready to prove Theorem \ref{3.1}.

\begin{proof} \textbf{Step 1:} The locally uniform positivity of $\phi$. For any $\eta\in (0,1)$, set 
\begin{equation*}
\phi_0^\eta=\phi_0+\eta,\ \ \psi^\eta_0=\fr{a\delta}{\delta-1}\nabla (\phi_0+\eta)^{2\iota},\ \ h_0^\eta=(\phi_0+\eta)^{2\iota}.
\end{equation*}
Then the corresponding initial compatibility conditions can be written as 
\begin{equation}\label{cc2}
\begin{split}
&\nabla u_0=(\phi_0+\eta)^{-\iota}g^\eta_1,\ \  Lu_0=(\phi_0+\eta)^{-2\iota}g^\eta_2,\\
&\nabla((\phi_0+\eta)^{2\iota} Lu_0)=(\phi_0+\eta)^{-\iota}g^\eta_3,\ \ \nabla^2l_0=(\phi_0+\eta)^{-\iota}g^\eta_4,
\end{split}
\end{equation}
where $g_i^\eta (i=1,2,3,4)$ are given as
\begin{equation*}
\begin{cases}
\displaystyle
g_1^\eta=\fr{\phi_0^{-\iota}}{(\phi_0+\eta)^{-\iota}}g_1,\ \ g_2^\eta=\fr{\phi_0^{-2\iota}}{(\phi_0+\eta)^{-2\iota}}g_2,\\[4pt]
\displaystyle
g_3^\eta=\fr{\phi_0^{-3\iota}}{(\phi_0+\eta)^{-3\iota}}(g_3-\fr{\eta\nabla\phi_0^{2\iota}}{\phi_0+\eta}\phi_0^\iota Lu_0),\\[4pt]
\displaystyle
g_4^\eta=\fr{\phi_0^{-\iota}}{(\phi_0+\eta)^{-\iota}}g_4.
\end{cases}
\end{equation*}
It follows from  \ef{a}-\ef{2.8*} that there exists a $\eta_1>0$ such that if $0<\eta<\eta_1$, then 
\begin{equation}\label{inda}
\begin{split}
&1+\eta+\bar{l}+\|\phi_0^\eta-\eta\|_{D^1_*\cap D^3}+\|u_0\|_{3}+\|(h^\eta_0)^{-1}\|_{L^\infty\cap D^{1,q}\cap D^{2,3}\cap D^3}\\
&+\|\psi ^\eta_0\|_{L^q\cap D^{1,3}\cap D^2}+|\nabla(h^\eta_0)^{\fr{1}{2}}|_6+|g_1^\eta|_2+|g_2^\eta|_2+|g_3^\eta|_2+|g_4^\eta|_2\\
&+\|l_0-\bar{l}\|_{D^1_*\cap D^3}+|l_0^{-1}|_\infty\leq \bar{c}_0,
\end{split}
\end{equation}
where $\bar{c}_0$ is a positive constant independent of $\eta$. Therefore, it follows from Theorem \ref{nltm} that for initial data $(\phi^\eta_0,u^\eta_0,l^\eta_0,\psi^\eta_0)$, the problem \ef{nl} admits a unique strong solution  $(\phi^\eta,u^\eta,l^\eta, \psi^\eta)$ in $[0,T_*]\times \mathbb{R}^3$ satisfying the local estimate in \ef{key1kk} with $c_0$ replaced by $\bar{c}_0$, and the life span $T_*$ is also independent of $\eta$.

Moreover,  $\phi^\eta$ is uniform positive locally as shown below.
\begin{lemma}\label{phieta}
For any $R_0>0$ and $\eta\in (0,1]$, there esists a constant $a_{R_0}$ independent of  $\eta$ such that
\begin{equation}\label{phieta1}
\phi^\eta(t,x)\geq a_{R_0}>0,\ \ \ \forall (t,x)\in [0,T_*]\times B_{R_0}.
\end{equation}
\end{lemma}
\begin{proof} It suffices to consider the case that $R_0$ is sufficiently large. 

It follows from \ef{a} and Gagliardo-Nirenberg inequality that
$\nabla \phi^{2\iota}_0\in L^\infty$. This implies that  the initial  vacuum  does not  occur in the interior point but  in the far field, and for every $R'>2$, there exists a constant $C_{R'}$ such that
\begin{equation}\label{hao1}
\phi^\eta_0(x)  \geq C_{R'}+\eta>0, \quad \forall \ x\in  B_{R'},
\end{equation}
where $C_{R'}$ is independent of $\eta$.

Now, let $x(t;x_0)$ be the particle path starting from $x_0$ at $t=0$, i.e.,
\begin{equation}\label{gobn}
\begin{cases}
\displaystyle
\frac{d}{dt}x(t;x_0)=u^\eta(t,x(t;x_0)),\\
\displaystyle
x(0;x_0)=x_0,
\end{cases}
\end{equation}
and $B(t,R')$ be the image of $B_{R'}$  under the flow map (\ref{gobn}). 

It follows from $\ef{nl}_1$ that
\begin{equation}\label{zhengleme}
\phi^\eta(t,x)=\phi^\eta_0(x_0)\exp\Big(-\int_{0}^{t}(\gamma-1)\textrm{div} u^\eta(s;x(s;x_0))\text{d}s\Big).
\end{equation}
It follows from (\ref{key1kk})  that  for $ 0\leq t \leq T_*$,
\begin{equation}\label{zhengle}
\begin{split}
&\int_0^t|\textrm{div} u^\eta(t,x(t;x_0)|\text{d}s\leq 
\int_0^t \|\nabla u^\eta\|_2\text{d}s
\leq c_3T^{\frac{1}{2}}_*.
\end{split}
\end{equation}
This, together with $(\ref{hao1})$ and $(\ref{zhengleme})$, 
yields that for $ 0\leq t \leq T_*$,
\begin{equation}\label{hao2}
 \phi^\eta(t,x)\geq C^*(C_{R'}+\eta)>0, \quad \forall \ x\in  B(t,R'),
\end{equation}
where $C^*=\exp\big(-(\gamma-1)c_3T^{\frac{1}{2}}_*\big)$.

On the other hand, it follows from (\ref{gobn}) and (\ref{timedefinition})-(\ref{key1kk}) that
\begin{equation*}
\begin{split}
&|x_0-x|=|x_0-x(t;x_0)|
\leq  \int_0^t| u^\eta(s,x(s;x_0))|\text{d}s\leq c_3t\leq 1\leq  R'/2,
\end{split}
\end{equation*}
for all $(t,x)\in [0,T_*]\times B_{R'}$,  which  implies $
B_{R'/2} \subset B(t,R')
$.
Thus, one can  choose
$$R'=2R_0,\quad \text{and} \quad a_{R_0}=C^*C_{R'}.$$
This Lemma \ref{phieta} is proved.
\end{proof}
\textbf{Step 2:} Taking limit $\eta\rightarrow 0^+$. Due to the $\eta$-independent estimate \ef{key1kk},  there exists a subsequence $(\phi^\eta,u^\eta,l^\eta, \psi^\eta)$ converging to a limit $(\phi,u,l, \psi)$ in weak or weak$^*$ sense:
\begin{equation}\label{key}
\begin{split}
&\phi^\eta-\eta\rightharpoonup \phi \ \ \text{weakly}^* \ \ \text{in} \ \ L^\infty([0,T_*];D^1_*\cap D^3),\\
&u^\eta\rightharpoonup u \ \ \text{weakly} \ \ \text{in} \ \ L^2([0,T_*];H^4),\\
&\psi^\eta\rightharpoonup \psi  \ \ \text{weakly}^* \ \ \text{in} \ \ L^\infty([0,T_*];L^q\cap D^{1,3}\cap D^2),\\
&\phi_t^\eta\rightharpoonup \phi_t \ \ \text{weakly}^* \ \ \text{in} \ \ L^\infty([0,T_*];H^2),\\
&(u_t^\eta,\psi_t^\eta)\rightharpoonup (u_t,\psi_t) \ \ \text{weakly}^* \ \ \text{in} \ \ L^\infty([0,T_*];H^1),\\
&l^\eta-\bar{l}\rightharpoonup l-\bar{l} \ \ \text{weakly}^* \ \ \text{in} \ \ L^\infty([0,T_*];D^1_*\cap D^3),\\
&l_t^\eta\rightharpoonup l_t \ \ \text{weakly}^* \ \ \text{in} \ \ L^\infty([0,T_*];L^\infty\cap D^1_*\cap D^2).\\
\end{split}
\end{equation}
Then by the lower semi-continuity of weak convergences, $(\phi,u,l,\psi)$ satisfies the corresponding estimates as in \ef{key1kk} except weighted ones on $u$.

On  the other hand, for any $R>0$, the Aubin-Lions Lemma and Lemma \ref{phieta} imply that there exists a subsequence (still denote by ($\phi^\eta,u^\eta,l^\eta,\psi^\eta$)) satisfying 
\begin{equation}\label{svs}
\begin{split}
\phi^\eta-\eta\rightarrow \phi \ \ \text{in} \ \ C([0,T_*];D^1_*(B_R)),\ \ \psi^\eta\rightarrow\psi\ \ \text{in} \ \ C([0,T_*];D^{1,3}(B_R)),&\\
u^\eta\rightarrow u \ \ \text{in} \ \ C([0,T_*]; H^2(B_R)), \ \ l^\eta-\bar{l}\rightarrow l-\bar{l} \ \ \text{in}\ \  C([0,T_*];D^1_*(B_R)),&\\
h^\eta\rightarrow h\ \ \text{in} \ \ C([0,T_*];H^2(B_R)).
\end{split}
\end{equation}

Also,  one can verify that: 
\begin{equation}\label{rela}
h=\phi^{2\iota}, \ \ \ \psi=\fr{a\delta}{\delta-1}\nabla h=\fr{a\delta}{\delta-1}\nabla\phi^{2\iota},
\end{equation}
 by the same argument used in the proof of  \ef{relationrecover}.

Furthermore, one has
\begin{equation*}
\begin{split}
&\int^{T_*}_0\int_{\mathbb{R}^3}(h^\eta\nabla^2u^\eta-h\nabla^2u)X\text{d}x\text{d}t\\
&=\int^{T_*}_0\int_{\mathbb{R}^3}\big((h^\eta-h)\nabla^2u^\eta+h(\nabla^2u^\eta-\nabla^2u))\big)X\text{d}x\text{d}t
\end{split}
\end{equation*}
for any test function $X(t,x)\in C^\infty_c([0,T_*]\times \mathbb{R}^3)$. Due to \ef{svs} and Lemma \ref{phieta}, it holds that 
\begin{equation}\label{svu1}
h^\eta\nabla^2u^\eta\rightharpoonup h\nabla^2u \ \ \text{weakly}^*\ \ \text{in}\ \ L^\infty([0,T_*];L^2).
\end{equation}
Similarly, one can also get that 
\begin{equation}\label{svu2}
\begin{split}
\sqrt{h^
\eta}(\nabla u^\eta,\nabla u^\eta_t)\rightharpoonup \sqrt{h}(\nabla u,\nabla u_t)\ \ \text{weakly}^*\ \ \text{in}\ \ L^\infty([0,T_*];L^2),&\\
h^\eta\nabla^2 u^\eta\rightharpoonup h\nabla^2u\ \ \text{weakly}^*\ \ \text{in}\ \ L^\infty([0,T_*];D^1_*),&\\
h^\eta\nabla^2u^\eta\rightharpoonup  h\nabla^2u \ \ \text{weakly}\ \ \text{in}\ \ L^2([0,T_*];D^1_*\cap D^2),&\\
(h^\eta\nabla^2 u^\eta)_t\rightharpoonup (h\nabla^2u)_t\ \ \text{weakly}\ \ \text{in}\ \ L^2([0,T_*];L^2).&
\end{split}
\end{equation}
Hence, the corresponding weighted estimates for $u$ in \ef{key1kk} still hold for the limit functions. Furthermore, $(\phi,u,l,\psi)$ is a weak solution to the Cauchy problem \ef{2.3}-\ef{2.5} in the sense of distributions.

\textbf{Step 3.}  The uniqueness follows easily from  the same  procedure as that for     Theorem \ref{nltm}.

\textbf{Step 4:} Time continuity. The time continuity of $(\phi,\psi,l)$ can be obtained by a  similar argument as for Lemma \ref{ls}.

For the velocity $u$,  the a priori estimates obtained above and Sobolev embedding theorem imply that
\begin{equation}\label{zheng1}
\begin{split}
 u\in C([0,T_*]; H^2)\cap  C([0,T_*]; \text{weak-}H^3) \quad \text{and} \quad   \phi^{\iota}\nabla u\in  C([0,T_*]; L^2).
 \end{split}
\end{equation}
It then follows from $(\ref{2.3})_2$ that
$$
\phi^{-2\iota} u_t \in L^2([0,T_*];H^2),\quad (\phi^{-2\iota} u_t)_t \in L^2([0,T_*];L^2),
$$
which implies that
$
\phi^{-2\iota} u_t \in C([0,T_*];H^1)
$.
This and the classical elliptic estimates for
\begin{equation*}
\begin{split}
a_2Lu&=-l^{-\nu}\phi^{-2\iota}(u_t+u\cdot\nabla u +a_1\phi \nabla l+l\nabla\phi-a_2\phi^{2\iota}\nabla l^\nu \cdot Q(u)-a_3l^\nu \psi  \cdot Q(u))
\end{split}
\end{equation*}
show that   $
 u\in C([0,T_*]; H^{3})$ immediately.

Finally, note that 
$$
h\nabla^2 u \in L^\infty([0,T_*]; H^1)\cap L^2([0,T_*] ; D^2) \quad \text{and} \quad   (h\nabla^2 u)_t \in  L^2([0,T_*] ; L^2).
$$
Thus the classical Sobolev embedding theorem implies that
$$
h\nabla^2 u\in C([0,T_*]; H^1).
$$
Then the  time continuity of $u_t$ follows easily. We conclude that \ef{b} holds.

In summary,  $(\phi,u,l,\psi)$  is the unique strong  solution in $[0,T_*]\times \mathbb{R}^3$ to the Cauchy problem \ef{2.3}-\ef{2.5}.

Thus the proof of  Theorem \ref{3.1} is complete.
\end{proof}

\subsection{The proof for Theorem \ref{th21}.} Now we are ready to establish the local-in-time well-posedness of regular solutions  stated in Theorem \ref{th21}  to the Cauchy problem   \eqref{8} with \eqref{2} and \eqref{QH}-\eqref{7}.

\begin{proof} \textbf{Step 1.} 
It follows from the initial assumptions (\ref{2.7})-(\ref{2.8}) and Theorem \ref{3.1} that there exists  a time $T_{*}> 0$ such that the problem \ef{2.3}-\ef{2.5} has a unique strong  solution $(\phi,u,l,\psi)$ satisfying the regularity (\ref{b}), which implies that
\begin{equation*}
\phi\in C^1([0,T_*]\times\mathbb{R}^3), \ \ (u,\nabla u)\in C([0,T_*]\times\mathbb{R}^3), \ \ l\in C^1([0,T_*]\times\mathbb{R}^3).
\end{equation*}

Set $\rho=(\fr{\gamma-1}{A\gamma}\phi)^{\fr{1}{\gamma-1}}$ with $\rho(0,x)=\rho_0$. According to  the relations between $(\varphi, \psi)$  and $\phi$, one can obtain 
\begin{equation*}
\varphi=a\rho^{1-\delta}, \ \ \psi=\fr{\delta}{\delta-1}\nabla\rho^{\delta-1}.
\end{equation*}

Then multiplying $\ef{2.3}_1$ by $\fr{\partial\rho}{\partial\phi}$, 
$\ef{2.3}_2$ by $\rho$, and $\ef{2.3}_3$ by $Ac_v\rho^\gamma$ respectively
shows that the equations in \ef{8} are satisfied.

Hence, we have shown  that the triple  $(\rho,u,S)$ satisfied the Cauchy problem \ef{8} with \ef{2} and \eqref{QH}-\eqref{7}  in the sense of distributions and the regularities in Definition 1.1. Moreover, it follows  from the continuity equation that $\rho(t,x)>0$ for $(t,x)\in [0,T_*]\times \mathbb{R}^3$. In summary, the Cauchy problem \ef{8} with \ef{2} and \eqref{QH}-\eqref{7} has a unique regular solution $(\rho,u,S)$.

\textbf{Step 2.}  Now we will show that the regular solution obtained in the above step in fact is  also a classical one  within its life span.

First, according to the regularities of $(\rho,u,S)$ and the fact 
$$\rho(t,x)>0\quad \text{for}\quad (t,x)\in [0,T_*]\times \mathbb{R}^3,$$
one can obtain that 
\begin{equation*}
(\rho,\nabla\rho,\rho_t,u,\nabla u, S, S_t, \nabla S)\in C([0,T_*]\times\mathbb{R}^3).
\end{equation*}

Second, by the classical  Sobolev embedding theorem:
\begin{equation*}
L^2([0,T_*];H^1)\cap W^{1,2}([0,T_*];H^{-1})\hookrightarrow C([0,T_*];L^2),
\end{equation*}
and the regularity \ef{2.9}, one gets that
\begin{equation*}
tu_t\in C([0,T_*];H^2), \ \ \text{and}\ \ u_t\in C([\tau,T_*]\times \mathbb{R}^3).
\end{equation*}

Finally, it remains  to show that $\nabla^2u\in  C([\tau,T_*]\times \mathbb{R}^3)$. Note that the following elliptic system holds
\begin{equation*}
\begin{split}
a_2Lu&=-l^{-\nu}\phi^{-2\iota}(u_t+u\cdot\nabla u +a_1\phi \nabla l+l\nabla\phi-a_2\phi^{2\iota}\nabla l^\nu \cdot Q(u)-a_3l^\nu\psi \cdot  Q(u))\\
&\equiv l^{-\nu}\phi^{-2\iota}\mathbb{M}.
\end{split}
\end{equation*}

It follows from the definition of regular solutions and  \ef{2.9} directly  that
\begin{equation*}
t l^{-\nu}\phi^{-2\iota}\mathbb{M}\in L^\infty([0,T_*];H^2),
\end{equation*}
and 
\begin{equation*}\begin{split}
(t l^{-\nu}\phi^{-2\iota}\mathbb{M})_t
=& l^{-\nu}\phi^{-2\iota}\mathbb{M}+t(l^{-\nu})_t\phi^{-2\iota}\mathbb{M}+t l^{-\nu}(\phi^{-2\iota})_t\mathbb{M}\\
&+t l^{-\nu}\phi^{-2\iota}\mathbb{M}_t\in L^2([0,T_*];L^2),
\end{split}
\end{equation*}
which, along with  the classical Sobolev embedding theorem:
\begin{equation*}
L^\infty([0,T_*];H^2)\cap W^{1,2}([0,T_*];H^{-1})\hookrightarrow C([0,T_*];L^r),
\end{equation*}
for any $ r \in [2,6)$, yields  that 
\begin{equation*}
t l^{-\nu}\phi^{-2\iota}\mathbb{M}\in C([0,T_*];W^{1,4}), \ \ t\nabla^2u \in  C([0,T_*];W^{1,4}).
\end{equation*}
These and the standard elliptic regularity yield immediately that $\nabla^2 u\in C((0,T_*]\times \mathbb{R}^3)$. 

By the way, according to the relation \eqref{keyobsevarion1}, one also has that  $(\rho,u,S)$ is also  a classical solution to  the Cauchy problem   \eqref{1}-\eqref{3} with $\kappa=0$ and \eqref{6}-\eqref{7}) in $(0,T_*]\times \mathbb{R}^3$.

\textbf{Step 3.} We  show finally  that
if one assumes   $m(0)<\infty$ additionally, then $(\rho,u,S)$ preserves the conservation of   total mass, momentum and total energy within its life span.
First, we  show that $(\rho,u,S)$ has finite total mass $m(t)$,  momentum $\mathbb{P}(t)$ and  total energy $E(t)$.
\begin{lemma}
\label{lemmak-1} 
Under the additional assumption,   $0<m(0)<\infty$, it holds that
$$  m(t)+| \mathbb{P}(t)|+E(t)<\infty \quad \text{for} \quad t\in [0,T_*]. $$
\end{lemma}
\begin{proof}
Let $f: \mathbb{R}^{+} \to \mathbb{R}$ be a non-increasing $C^2$ function satisfying
\begin{align*}
f(s)=\begin{cases}
1 & s \in [0,\frac{1}{2}],\\[2pt]
\text{non-negative polynomial} & s \in [\frac{1}{2}, 1),\\[2pt]
e^{-s} & s \ge 1.
\end{cases}
\end{align*}
 It is obvious that  there exists a generic constant $C >0$ such that 
\begin{align*}
|f'(s)| \le C f(s).
\end{align*}
For any $R>1$, define $
f_R(x)=f (\tfrac{|x|}{R})$. Then it holds that for any $p\ge 0$,
\begin{align}\label{601}
|x|^p f_R(x) \leq C \quad \text{and} \quad \lim_{|x|\to \infty} |x|^p f_R(x) =0.
\end{align}

According to the regularity of solutions obtained and the definition of $f$, one can make sure that, 
\begin{align*}
&\int \big(\rho + |\dv (\rho u)|  \big) f_R \le C(R),\\
& \int \Big(\big| \rho u \cdot x f' (\tfrac{|x|}{R}) \tfrac{1 }{R|x|}\big|+ \rho |u| f(\tfrac{|x|}{R}) \tfrac{1 }{R}\Big) \le C(R),
\end{align*}
for any fixed $R>1$, where $C(R)$ is a positive constant depending on $R$.

Since the continuity equation $\eqref{1}_1$ holds  everywhere, one can multiply $\eqref{1}_1$  by $ f_R(x)$ and integrate with respect to $x$ to get
\begin{align}\label{xccv}
\frac{d}{dt} \int \rho  f_R(x) =-\int \dv (\rho u)  f_R(x) .
\end{align}
Then it follows from integration by parts that
\begin{align*}
&-\int \dv (\rho u)  f_R(x) 
=  \int \rho u \cdot x f' (\tfrac{|x|}{R}) \tfrac{1 }{R|x|} 
\le  C  \frac{|u|_{\infty} }{R} \int \rho   f_R (x),
\end{align*}
which,  along with (\ref{xccv}) and Gronwall's inequality,  shows that 
\begin{align*}
\text{ess}\sup_{0\leq t \leq T} \int \rho  f_R(x)  \le C \int \rho_0  f_R(x),
\end{align*}
with $C$ a generic  constant independent of $R$. 
Note that 
$$
\rho  f_R(x)  \to \rho  \quad \text{as}\quad  R\to \infty$$
for all $x\in \mathbb{R}^3$, thus by  Fatou's lemma (i.e., Lemma \ref{Fatou})
\begin{equation}\label{keji}
\text{ess}\sup_{0\leq t \leq T} \int \rho   \le \text{ess}\sup_{0\leq t \leq T}  \liminf_{R \to \infty}\int \rho  f_R(x)  <\infty.
\end{equation}

Second, based on \eqref{keji},  one has
\begin{equation}\label{finite}\begin{split}
\displaystyle
|\mathbb{P}(t)|=&\Big|\int  \rho u \Big|\leq C|\rho|_1|u|_\infty<\infty,\\
\displaystyle
E(t)=&\int  \Big(\frac{1}{2}\rho|u|^2+\frac{P}{\gamma-1}\Big) \\
\leq & C(|\rho|_1|u|^2_\infty+|\rho|^{\gamma-1}_\infty|\rho|_1|e^{\frac{S}{c_v}}|_\infty)<\infty.
\end{split}
\end{equation}
The proof of this lemma is complete.

\end{proof}

We are now ready to  prove the conservation of  total mass, momentum and total energy.
\begin{lemma}
\label{lemmak}Under the additional assumption,    $0<m(0)<\infty$,  it holds that 
$$  m(t)=m(0),\quad  \mathbb{P}(t)=\mathbb{P}(0),\quad  E(t)=E(0) \quad \text{for} \quad t\in [0,T_*]. $$

\end{lemma}
\begin{proof}
First, $(\ref{1})_2$ and the regularity of the solution imply  that
\begin{equation}\label{deng1}
\mathbb{P}_t=-\int \text{div}(\rho u \otimes u)-\int \nabla P+\int \text{div}\mathbb{T}=0,
\end{equation}
where one has used the fact that
$$
\rho u^{(i)}u^{(j)},\quad \rho^\gamma e^{\frac{S}{c_v}} \quad \text{and} \quad \rho^\delta e^{\frac{S}{c_v}\nu} \nabla u \in W^{1,1}(\mathbb{R}^3)\quad \text{for} \quad i,\ j=1,\ 2,\ 3.
$$

Second,  the energy equation $(\ref{1})_3$ implies  that
\begin{equation}\label{dengn1}
E_t=-\int \text{div}(\rho \mathcal{E} u+ Pu-u\mathbb{T})=0,
\end{equation}
where the following facts have been used:
$$
\frac{1}{2}\rho |u|^2u,\quad \rho^\gamma e^{\frac{S}{c_v}} u \quad \text{and} \quad \rho^\delta e^{\frac{S}{c_v}\nu} u \nabla u \in W^{1,1}(\mathbb{R}^3).
$$

Similarly, one  can show  the conservation of the total mass.
\end{proof}

Hence the proof of Theorem \ref{th21} is complete.
\end{proof}

\section{Remarks on the asymptotic behavior of  $u$}

\subsection{Non-existence of global solutions with $L^\infty$ decay on $u$ }
\subsubsection{Proof of Theorem \ref{th25}}
Now  we  prove Theorem \ref{th25}.  Let $T>0$ be any constant, and    $(\rho,u,S)\in D(T)$. It  follows from the definitions of  $m(t)$, $\mathbb{P}(t)$ and $E_k(t)$ that
$$
 |\mathbb{P}(t)|\leq \int \rho(t,x)|u(t,x)|\leq  \sqrt{2m(t)E_k(t)},
$$
which, together with the definition of the solution class $D(T)$, implies that
$$
0<\frac{|\mathbb{P}(0)|^2}{2m(0)}\leq E_k(t)\leq \frac{1}{2} m(0)|u(t)|^2_\infty \quad \text{for} \quad t\in [0,T].
$$
Then one obtains that there exists a positive constant $C_u=\frac{|\mathbb{P}(0)|}{m(0)}$ such that
$$
|u(t)|_\infty\geq C_u  \quad \text{for} \quad t\in [0,T].
$$
Thus one obtains     the desired conclusion as shown in Theorem \ref{th25}.

\subsubsection{Proof of Corollary \ref{co23}}

Let  $ (\rho,u,S)(t,x)$ defined  in $[0,T]\times \mathbb{R}^3$ be the regular solution obtained  in Theorem  \ref{th21}. It follows from Theorem  \ref{th21} that $(\rho,u,S)\in D(T)$, which, along with  Theorem \ref{th25}, yields that   Corollary \ref{co23} holds.

\subsection{Non-conservation of momentum  for  constant viscosities}
However, for flows of constant viscosities and thermal conductivity \cites{CK,  wenzhu}, in Corollary \ref{th:2.20-HLX}, we  will show that the classical solution exists globally and keeps the conservation of total mass,  but can not keep the conservation of momentum for large time for a class of initial data with far field vacuum.  This is essentially due to   that   $\mathbb{T}$  does not belong to $ W^{1,1}(\mathbb{R}^3)$.

Set:
$$
F\triangleq (2\mu+\lambda)\text{div} u-P,\quad \omega\triangleq\nabla \times u,
$$
to be  the effective viscous flux and the vorticity respectively. Then  $F$ and $\omega$ satisfy the following elliptic equations:
  \begin{equation}\label{fomega} 
  \triangle F=\text{div}(\rho \dot{u}) \quad \text{and}\quad  \mu\triangle \omega=\nabla \times (\rho \dot{u}).
  \end{equation}

 The proof of Corollary \ref{th:2.20-HLX} is divided into three steps: 

 \textbf{Step 1:} Local-in-time well-posedness.  It follows easily from the initial assumption \eqref{initialhxl}-\eqref{changxiangrong} and    the arguments  used in Appendix B of  \cite{wenzhu} that there exists a time $T_0>0$ and a unique classical solution $(\rho,u,\theta)$ in $(0,T_0]\times \mathbb{R}^3$  to the Cauchy  problem  \eqref{1}-\eqref{3} with \eqref{c6}-\eqref{c7}  satisfying \eqref{lagrangian3q}-\eqref{lagrangian5q}, and 
$$ \rho(t,x)\geq 0,\quad  \theta(t,x)\geq 0,  \quad (t,x)   \in [0,\infty)\times \mathbb{R}^3.$$

  \textbf{Step 2:} Global-in-time well-posedness.  It follows from the local-in-time well-posedenss obtained in Step 1,  the smallness assumption (\ref{smallness}) and the conclusions obtained   in  Theorems 2.13 and  2.17, Remark 2.14, Lemma 3.1 and Proposition 4.1 of  \cite{wenzhu} that   the Cauchy problem  \eqref{1}-\eqref{3} with \eqref{c6}-\eqref{c7} 
  has a unique  global classical solution $(\rho,u,\theta)$ in $(0,\infty)\times \mathbb{R}^3$ satisfying 
\begin{equation}\label{cvv2} m(t)=m(0)\quad \text{and} \quad 0\leq E_k(t)\leq \frac{1}{2}m(0)|u(t,\cdot)|^2_\infty<\infty \quad \text{for} \quad t\in [0,T],
\end{equation}
 for arbitrarily large $T>0$, and \eqref{lagrangian1q}-\eqref{largetimehlx} for any $0<\tau <T<\infty$.

 \textbf{Step 3:} Large time behavior on $u$.
First, by the Gagliardo-Nirenberg  inequality  in Lemma \ref{lem2as}  and the standard regularity theory for elliptic systems, one has
\begin{equation}\label{keyobservation}
\begin{split}
|u|_\infty\leq & C|u|^{\frac{1}{2}}_6|\nabla u|^{\frac{1}{2}}_6\leq C|\nabla u|^{\frac{1}{2}}_2 (|F|^{\frac{1}{2}}_6+|\omega|^{\frac{1}{2}}_6+|P|^{\frac{1}{2}}_6)\\
\leq &C|\nabla u|^{\frac{1}{2}}_2(|\rho \dot{u}|_2+|P|_6)^{\frac{1}{2}}\leq C|\nabla u|^{\frac{1}{2}}_2(|\sqrt{\rho} \dot{u}|_2+|\nabla \theta|_2)^{\frac{1}{2}},
\end{split}
\end{equation}
which, together with (\ref{lagrangian1q})-(\ref{lagrangian2q}) and (\ref{largetimehlx}), implies that
\begin{equation}\label{laile}
\limsup_{t\rightarrow \infty} |u(t,x)|_{\infty}=0.
\end{equation}

Finally, it follows from \eqref{cvv2},  (\ref{laile}), and  the proof of   Theorem \ref{th25} that  if  $m(0)>0$ and $|\mathbb{P}(0)|>0$,  the law of conservation of momentum of the global solution obtained in Step 2  can not  be preserved
 for all the time $t\in (0,\infty)$.

Thus the proof of Corollary \ref{th:2.20-HLX} is complete.

\bigskip

\section{Appendix}

For convenience of readers, we  list some basic facts which have been  used frequently in this paper.


\bigskip




The first one is the  well-known Gagliardo-Nirenberg inequality.

\begin{lemma}\cite{oar}\label{lem2as}\
	Let function $u\in L^{q_1}\cap D^{1,r}(\mathbb{R}^d)$ for   $1 \leq q_1,  r \leq \infty$.  Suppose also that  real numbers $\xi$ and $q_2$,  and  natural numbers $m$, $i$  and $j$ satisfy
	$$\frac{1}{{q_2}} = \frac{j}{d} + \left( \frac{1}{r} - \frac{i}{d} \right) \xi + \frac{1 - \xi}{q_1} \quad \text{and} \quad 
	\frac{j}{i} \leq \xi \leq 1.
	$$
	Then $u\in D^{j,{q_2}}(\mathbb{R}^d)$, and  there exists a constant $C$ depending only on $i$, $d$, $j$, $q_1$, $r$ and $\xi$ such that
	\begin{equation}\label{33}
	\begin{split}
	\| \nabla^{j} u \|_{L^{{q_2}}} \leq C \| \nabla^{i} u \|_{L^{r}}^{\xi} \| u \|_{L^{q_1}}^{1 - \xi}.
	\end{split}
	\end{equation}
	Moreover, if $j = 0$, $ir < d$ and $q_1 = \infty$, then it is necessary to make the additional assumption that either u tends to zero at infinity or that u lies in $L^s(\mathbb{R}^d)$ for some finite $s > 0$;
	if $1 < r < \infty$ and $i -j -d/r$ is a non-negative integer, then it is necessary to assume also that $\xi \neq 1$.
\end{lemma}

The second  one  concerns  commutator  estimates, which can be found in \cite{amj}.
\begin{lemma}\cite{amj}\label{zhen1}
	Let  $r$, ${r_1}$ and ${r_2}$  be constants such that
	$$\frac{1}{r}=\frac{1}{r_1}+\frac{1}{{r_2}},\quad \text{and} \quad 1\leq r_1,\ {r_2}, \ r\leq \infty.$$  For $  s\geq 1$, if $f, g \in W^{s,r_1} \cap  W^{s,{r_2}}(\mathbb{R}^3)$, then it holds that
	\begin{equation}\begin{split}\label{ku11}
	&|\nabla^s(fg)-f \nabla^s g|_r\leq C_s\big(|\nabla f|_{r_1} |\nabla^{s-1}g|_{r_2}+|\nabla^s f|_{r_2}|g|_{r_1}\big),\\
	\end{split}
	\end{equation}
	\begin{equation}\begin{split}\label{ku22}
	&|\nabla^s(fg)-f \nabla^s g|_r\leq C_s\big(|\nabla f|_{r_1} |\nabla^{s-1}g|_{r_2}+|\nabla^s f|_{r_1}|g|_{r_2}\big),
	\end{split}
	\end{equation}
	where $C_s> 0$ is a constant depending only on $s$, and $\nabla^s f$ ($s\geq 1$) is the set of  all $\partial^\varsigma_x f$  with $|\varsigma|=s$. Here $\varsigma=(\varsigma_1,\varsigma_2,\varsigma_3)^\top\in \mathbb{R}^3$ is a multi-index.
\end{lemma}

The third lemma gives some compactness results obtained via the Aubin-Lions Lemma.
\begin{lemma}\cite{jm}\label{aubin} Let $X_0\subset X\subset X_1$ be three Banach spaces.  Suppose that $X_0$ is compactly embedded in $X$ and $X$ is continuously embedded in $X_1$. Then the following statements hold.
	
	\begin{enumerate}
		\item[i)] If $J$ is bounded in $L^r([0,T];X_0)$ for $1\leq r < +\infty$, and $\frac{\partial J}{\partial t}$ is bounded in $L^1([0,T];X_1)$, then $J$ is relatively compact in $L^r([0,T];X)$;\\
		
		\item[ii)] If $J$ is bounded in $L^\infty([0,T];X_0)$  and $\frac{\partial J}{\partial t}$ is bounded in $L^r([0,T];X_1)$ for $r>1$, then $J$ is relatively compact in $C([0,T];X)$.
	\end{enumerate}
\end{lemma}

The following  lemma is used to improve weak convergence to strong one.
\begin{lemma}\cite{amj}\label{zheng5}
	If the  sequence $\{w_k\}^\infty_{k=1}$ converges weakly  to $w$ in a Hilbert space $X$, then it converges strongly to $w$ in $X$ if and only if
	$$
	\|w\|_X \geq \lim \text{sup}_{k \rightarrow \infty} \|w_k\|_X.
	$$
\end{lemma}

The following lemma is used to obtain the time-weighted estimates of $u$.
\begin{lemma}\cite{bjr}\label{bjr}
If $f(t,x)\in L^2([0,T]; L^2)$, then there exists a sequence $s_k$ such that
$$
s_k\rightarrow 0 \quad \text{and}\quad s_k |f(s_k,x)|^2_2\rightarrow 0 \quad \text{as} \quad k\rightarrow\infty.
$$
\end{lemma}

Finally, we list the well-known Fatou's lemma.
\begin{lemma}\label{Fatou}
Given a measure space $(V,\mathcal{F},\digamma)$ and a set $X\in \mathcal{F}$, let  $\{f_k\}$ be a sequence of $(\mathcal{F} , \mathcal{B}_{\mathbb{R}_{\geq 0}} )$-measurable non-negative functions $f_k: X\rightarrow [0,\infty]$. Define a function $f: X\rightarrow [0,\infty]$ by 
$$
f(x)= \liminf_{k\rightarrow \infty} f_k(x),
$$
for every $x\in X$. Then $f$ is $(\mathcal{F},  \mathcal{B}_{\mathbb{R}_{\geq 0}})$-measurable, and   
$$
\int_X f(x) \text{\rm d}\digamma \leq \liminf_{k\rightarrow \infty} \int_X f_k(x) \text{\rm d}\digamma.
$$
\end{lemma}

\bigskip

{\bf Acknowledgement:}  The research is partially supported by Zheng Ge Ru Foundation, Hong Kong RGC Earmarked Research Grants CUHK-14301421, CUHK-14300917, CUHK-14302819, and CUHK-14300819. Duan's research is also supported in part by National Natural Science Foundation of China under Grant 11771300, Natural Science Foundation of Guandong under Grant 2020A1515010554 and Research Foundation for "Kongque" Talents of Shenzhen. Xin's research is also supported in part by the key project of National Natural Science Foundation of China (No. 12131010) and Guangdong Province Basic and Applied Basic Research Foundation 2020B1515310002. Zhu's research is also  supported in part by National Natural Science Foundation of China under Grants 12101395 and 12161141004, The Royal Society-Newton International Fellowships NF170015, and Newton International Fellowships Alumni AL/201021 and AL/211005.

\bigskip

{\bf Conflict of Interest:} The authors declare that they have no conflict of interest.

\end{document}